\documentclass[12pt, a4paper, abstracton, bibliography=totoc]{scrartcl}
\pdfoutput=1

\usepackage{amsmath, amsthm, amsfonts, amssymb}
\usepackage[utf8]{inputenc}
\usepackage[T1]{fontenc}

\usepackage{slashed} 
\usepackage[all, cmtip]{xy}
\usepackage{hyperref} 

\let\counterwithout\relax

\usepackage{chngcntr} 

\usepackage{graphicx} 
\usepackage[english]{babel}
\usepackage{microtype}
\usepackage{bm} 
\usepackage{mathtools} 

\newcommand{\IU}{\mathcal{U}^\ast_{-\infty}}

\newcommand{\IN}{\mathbb{N}}
\newcommand{\IZ}{\mathbb{Z}}

\newcommand{\IR}{\mathbb{R}}
\newcommand{\IC}{\mathbb{C}}
\newcommand{\IH}{\mathcal{H}}
\newcommand{\frakS}{\mathfrak{S}}
\newcommand{\injrad}{\operatorname{inj-rad}}
\newcommand{\Rm}{\operatorname{Rm}}

\newcommand{\LLip}{L\text{-}\operatorname{Lip}}
\newcommand{\IB}{\mathfrak{B}}
\newcommand{\IK}{\mathfrak{K}}
\newcommand{\frakD}{\mathfrak{D}}
\newcommand{\frakC}{\mathfrak{C}}

\newcommand{\supp}{\operatorname{supp}}
\newcommand{\diam}{\operatorname{diam}}
\newcommand{\Hom}{\operatorname{Hom}}
\newcommand{\Mat}{\operatorname{Mat}}

\newcommand{\id}{\operatorname{id}}
\newcommand{\kernel}{\operatorname{ker}}

\newcommand{\image}{\operatorname{im}}

\newcommand{\card}{\#}
\newcommand{\Vect}{\operatorname{Vect}}
\newcommand{\Idem}{\operatorname{Idem}}

\newcommand{\closure}{\operatorname{cl}}
\newcommand{\trace}{\operatorname{tr}}
\newcommand{\ch}{\operatorname{ch}}
\newcommand{\vol}{\operatorname{vol}}
\newcommand{\ind}{\operatorname{ind}}

\newcommand{\op}{\mathrm{op}}
\newcommand{\pt}{\mathrm{pt}}
\newcommand{\Symb}{\operatorname{Symb}}

\newcommand{\HPucont}{H \! P_\mathrm{cont}}
\newcommand{\HCucont}{H \! C_\mathrm{cont}}
\newcommand{\HHucont}{H \! H_\mathrm{cont}}

\newcommand{\HudR}{H^{u, \mathrm{dR}}}

\newcommand{\HbdR}{H_{b, \mathrm{dR}}}
\newcommand{\HudRco}{H_{u, \mathrm{dR}}}

\newcommand{\Cucont}{C_{\mathrm{cont}}}
\newcommand{\Clucont}{C_{\lambda, \mathrm{cont}}}
\newcommand{\Winftyone}{W^{\infty, 1}}

\newcommand{\ev}{\mathrm{ev}}
\newcommand{\odd}{\mathrm{odd}}
\newcommand{\UPsiDO}{\mathrm{U}\Psi\mathrm{DO}}

\newcommand{\Frechet}{Fr\'{e}chet }
\newcommand{\Poincare}{Poincar\'{e} }
\newcommand{\Folner}{F{\o}lner }
\newcommand{\Spakula}{\v{S}pakula }
\newcommand{\spinc}{spin$^c$ }

\newcommand*{\largecdot}{\raisebox{-0.25ex}{\scalebox{1.2}{$\cdot$}}}

\DeclareMathOperator{\hatotimes}{\hat{\otimes}}
\DeclareMathOperator{\barotimes}{\bar{\otimes}}

\newtheorem{thm}{Theorem}[section]
\newtheorem*{thm*}{Theorem}

\newtheorem*{roesthm*}{Roe's Index Theorem}

\newtheorem*{atiyahsingerthm*}{Atiyah--Singer Index Theorem}

\newtheorem*{thmlocaltwisted*}{Theorem~\ref{thm:local_thm_twisted}}
\newtheorem*{thmlocalpseudos*}{Theorem~\ref{thm:local_thm_pseudos}}
\newtheorem*{corpairingcompactlysupported*}{Corollary~\ref{cor:pairing_compactly_supported}}
\newtheorem*{thmexternalprodhomology*}{Theorem~\ref{thm:external_prod_homology}}
\newtheorem*{corpairingglobal*}{Corollary~\ref{cor:pairing_global}}
\newtheorem{cor}[thm]{Corollary}
\newtheorem*{cor*}{Corollary}
\newtheorem{lem}[thm]{Lemma}
\newtheorem{prop}[thm]{Proposition}
\newtheorem{conj}[thm]{Conjecture}
\newtheorem{question}[thm]{Question}
\theoremstyle{definition}
\newtheorem{rem}[thm]{Remark}
\newtheorem{example}[thm]{Example}
\newtheorem{examples}[thm]{Examples}
\newtheorem{defn}[thm]{Definition}

\numberwithin{equation}{section}

\counterwithout{footnote}{section}

\makeatletter
\def\blfootnote{\gdef\@thefnmark{}\@footnotetext}
\makeatother

\begin{document}

\title{Index theory of uniform pseudodifferential operators}
\author{Alexander Engel}
\date{}
\maketitle

\vspace*{-3\baselineskip}
\begin{center}
\footnotesize{
\textit{
Fakult\"{a}t f\"{u}r Mathematik\\
Universit\"{a}t Regensburg\\
93040 Regensburg, GERMANY\\
\href{mailto:alexander.engel@mathematik.uni-regensburg.de}{alexander.engel@mathematik.uni-regensburg.de}
}}
\end{center}

\vspace*{\baselineskip}
\begin{abstract}
\blfootnote{\textit{$2010$ Mathematics Subject Classification.} Primary:\ 58J20; Secondary:\ 19K56, 47G30.}
\blfootnote{\textit{Keywords and phrases.} Index theory, pseudodifferential operators, uniform $K$-homology.}
We generalize Roe's index theorem for graded generalized Dirac operators on amenable manifolds to multigraded elliptic uniform pseudodifferential operators.

This generalization will follow from a local index theorem that is valid on any manifold of bounded geometry. This local formula incorporates the uniform estimates that are present in the definition of our class of pseudodifferential operators.

We will revisit \v{S}pakula's uniform $K$-homology and show that multigraded elliptic uniform pseudodifferential operators naturally define classes in it. For this we will investigate uniform $K$-homology more closely, e.g., construct the external product and show invariance under weak homotopies. The latter will be used to refine and extend \v{S}pakula's results about the rough Baum--Connes assembly map.

We will identify the dual theory of uniform $K$-homology. We will give a simple definition of uniform $K$-theory for all metric spaces and in the case of manifolds of bounded geometry we will give an interpretation of it via vector bundles of bounded geometry. Using a version of Mayer--Vietoris induction that is adapted to our needs, we will prove \Poincare duality between uniform $K$-theory and uniform $K$-homology for \spinc manifolds of bounded geometry.

We will construct Chern characters from uniform $K$-theory to bounded de Rham cohomology and from uniform $K$-homology to uniform de Rham homology. Using the adapted Mayer--Vietoris induction we will also show that these Chern characters induce isomorphisms after a certain completion that also kills torsion.
\end{abstract}

\newpage
\tableofcontents

\section{Introduction}

Recall the following index theorem of Roe for amenable manifolds (with notation adapted to the one used in this article).

\begin{thm*}[{\cite[Theorem 8.2]{roe_index_1}}]
Let $M$ be a Riemannian manifold of bounded geometry and $D$ a generalized Dirac operator associated to a graded Dirac bundle $S$ of bounded geometry over $M$.

Let $(M_i)_i$ be a \Folner sequence\footnote{That is to say, for every $r > 0$ we have $\frac{\vol B_r(\partial M_i)}{\vol M_i} \stackrel{i \to \infty}\longrightarrow 0$. Manifolds admitting such a sequence are called \emph{amenable}.} for $M$, $\tau \in (\ell^\infty)^\ast$ a linear functional associated to a free ultrafilter on $\IN$, and $\theta$ the corresponding trace on the uniform Roe algebra of $M$.

Then we have
\[\theta(\mu_u(D)) = \tau \Big( \frac{1}{\vol M_i} \int_{M_i} \ind(D) \Big).\]
\end{thm*}

Here $\ind(D)$ is the usual integrand for the topological index of $D$ in the Atiyah--Singer index formula, so the right hand side is topological in nature. On the left hand side of the formula we have the coarse index class $\mu_u(D) \in K_0(C_u^\ast(M))$ of $D$ in the $K$-theory of the uniform Roe algebra of $M$ evaluated under the trace $\theta$.  This is an analytic expression and may be computed as $\theta(\mu_u(D)) = \tau \Big( \frac{1}{\vol M_i} \int_{M_i} \trace_s k_{f(D)}(x,x) \ dx \Big)$, where $k_{f(D)}(x,y)$ is the integral kernel of the smoothing operator $f(D)$, where $f$ is an even Schwartz function with $f(0) = 1$.

In this article we will generalize this theorem to multigraded, elliptic, symmetric uniform pseudodifferential operators. So especially we also encompass Toeplitz operators since they are included in the ungraded case. This generalization will follow from a local index theorem that will hold on any manifold of bounded geometry, i.e., without an amenability assumption on $M$.

Let us state our local index theorem in the formulation using twisted Dirac operators associated to \spinc structures.

\begin{thmlocaltwisted*}
Let $M$ be an $m$-dimensional \spinc manifold of bounded geometry and without boundary. Denote the associated Dirac operator by $D$.

Then we have the following commutative diagram:
\[\xymatrix{
K^\ast_u(M) \ar[r]^-{\largecdot \cap [D]}_-{\cong} \ar[d]_-{\ch(\largecdot) \wedge \ind(D)} & K_{m-\ast}^u(M) \ar[d]^-{\alpha_\ast \circ \ch^\ast}\\
\HbdR^\ast(M) \ar[r]_-{\cong} & \HudR_{m-\ast}(M)}\]
where in the top row $\ast$ is either $0$ or $1$ and in the bottom row $\ast$ is either $\ev$ or $\odd$.
\end{thmlocaltwisted*}

Here $K_{m-\ast}^u(M)$ is the uniform $K$-homology of $M$ invented by \Spakula \cite{spakula_uniform_k_homology} and $K_u^\ast(M)$ is the corresponding uniform $K$-theory which we will define in Section \ref{sec:uniform_k_th}. The map $\largecdot \cap [D]$ is the cap product and that it is an isomorphism will be shown in Section \ref{sec:poincare_duality}. Moreover, $\HbdR^\ast(M)$ denotes the bounded de Rham cohomology of $M$ and $\ind(D)$ the topological index class of $D$ in there. Furthermore, $\HudR_{m-\ast}(M)$ is the uniform de Rham homology of $M$ to be defined in Section \ref{sec:u_de_Rham_currents} via Connes' cyclic cohomology, and that it is \Poincare dual to bounded de Rham cohomology is stated in Theorem \ref{thm:Poincare_de_Rham}. Finally, let us note that we will also prove in Section \ref{sec:chern_isos} that the Chern characters induce isomorphisms after a certain completion that also kills torsion, similar to the case of compact manifolds.

Using a series of steps as in Connes' and Moscovici's proof of \cite[Theorem 3.9]{connes_moscovici} we will generalize the above computation of the \Poincare dual of $(\alpha_\ast \circ \ch^\ast)([D]) \in \HudR_{m-\ast}(M)$ to symmetric and elliptic uniform pseudodifferential operators.

\begin{thmlocalpseudos*}
Let $M$ be an oriented Riemannian manifold of bounded geometry and without boundary, and $P$ be a symmetric and elliptic uniform pseudodifferential operator.

Then $\ind(P) \in \HbdR^\ast(M)$ is the \Poincare dual of $(\alpha_\ast \circ \ch^\ast)([P]) \in \HudR_\ast(M)$.
\end{thmlocalpseudos*}

Using the above local index theorem we will derive as a corollary the following local index formula:

\begin{corpairingcompactlysupported*}
Let $[\varphi] \in H_{c, \mathrm{dR}}^k(M)$ be a compactly supported cohomology class and define the analytic index $\ind_{[\varphi]}(P)$ as Connes--Moscovici \cite{connes_moscovici} for $P$ being a multigraded, symmetric, elliptic uniform pseudodifferential operator of positive order. Then we have
\[\ind_{[\varphi]}(P) = \int_M \ind(P) \wedge [\varphi]\]
and this pairing is continuous, i.e., $\int_M \ind(P) \wedge [\varphi] \le \|\ind(P) \|_\infty \cdot \| [\varphi] \|_1$, where $\| \largecdot \|_\infty$ denotes the sup-seminorm on $\HbdR^{m-k}(M)$ and $\| \largecdot \|_1$ the $L^1$-seminorm on $H_{c, \mathrm{dR}}^k(M)$.
\end{corpairingcompactlysupported*}

Note that the corollary reads basically the same as the local index formula of Connes and Moscovici \cite{connes_moscovici}. The fundamentally new thing in it is the continuity statement for which we need the uniformity assumption for $P$.

As a second corollary to the above local index theorem we will, as already written, derive the generalization of Roe's index theorem for amenable manifolds.

\begin{corpairingglobal*}
Let $M$ be a manifold of bounded geometry and without boundary, let $(M_i)_i$ be a \Folner sequence for $M$ and let $\tau \in (\ell^\infty)^\ast$ be a linear functional associated to a free ultrafilter on $\IN$. Denote the from the choice of \Folner sequence and functional $\tau$ resulting functional on $K_0(C_u^\ast(M))$ by $\theta$.

Then for both $p \in \{ 0,1 \}$, every $[P] \in K_p^u(M)$ for $P$ a $p$-graded, symmetric, elliptic uniform pseudodifferential operator over $M$, and every $u \in K_u^p(M)$ we have
\[\langle u, [P] \rangle_\theta = \langle \ch(u) \wedge \ind(P), [M] \rangle_{(M_i)_i, \tau}.\]
\end{corpairingglobal*}

Roe's theorem \cite{roe_index_1} is the special case where $P = D$ is a graded (i.e., $p = 0$) Dirac operator and $u = [\IC]$ is the class in $K_u^0(M)$ of the trivial, $1$-dimensional vector bundle over $M$.

To put the above index theorems into context, let us consider manifolds with cylindrical ends. These are the kind of non-compact manifolds which are studied to prove for example the Atiyah--Patodi--Singer index theorem. In the setting of this paper, the relevant algebra would be that of bounded functions with bounded derivatives, whereas in papers like \cite{melrose_eta} or \cite{moroianu_nistor} one imposes conditions at infinity like rapid decay of the integral kernels (see the definition of the suspended algebra in \cite[Section~1]{melrose_eta}).

Note that this global index theorem arising from a \Folner sequence is just a special case of a certain rough index theory, where one pairs classes from the so-called rough cohomology with classes in the $K$-theory of the uniform Roe algebra, and \Folner sequences give naturally classes in this rough cohomology. For details see the thesis \cite{mavra} of Mavra. It seems that it should be possible to combine the above local index theorem with this rough index theory, since it is possible in the special case of \Folner sequences. The author investigated this in \cite{engel_rough}.

Let us say a few words about the proof of the above index theorem for uniform pseudodifferential operators. Roe used in \cite{roe_index_1} the heat kernel method to prove his index theorem for amenable manifolds and therefore, since the heat kernel method does only work for Dirac operators, it can not encompass uniform pseudodifferential operators. So what we will basically do in this paper is to set up all the necessary theory in order to be able to reduce the index problem from pseudodifferential operators to Dirac operators.

As it turns out, the only useful definition of pseudodifferential operators on non-compact manifolds is the uniform one since otherwise we can not guarantee that the operators do have continuous extensions to Sobolev spaces (we will elaborate more on this at the beginning of the next Section \ref{sec:PDOs}). Now in the reduction of the index problem from pseudodifferential operators to Dirac operators one can use, e.g., the fact that for \spinc manifolds we have $K$-\Poincare duality between $K$-theory and $K$-homology (i.e., every class of an abstractly elliptic operator may be represented by a difference of twisted Dirac operators). In order to do the same here in our uniform case we therefore first need uniform $K$-theory and $K$-homology theories (since usual $K$-homology does not consider at all any uniformity that we might have for our operators, and since we are forced to work with uniform pseudodifferential operators, it is quite clear that we also need a new $K$-homology theory that incorporates uniformity). To our luck, \Spakula already did this for us, i.e., he already defined in \cite{spakula_uniform_k_homology} a uniform $K$-homology theory and it turns out that this is exactly what we need. Evidence for the latter statement is provided by the fact that \Spakula constructed a rough assembly map from uniform $K$-homology to the $K$-theory of the uniform Roe algebra, and Roe uses in \cite{roe_index_1} the latter groups as receptacles for the analytic index classes for his index theorem.

So from the above it is quite clear what we have to do: after defining and investigating the class of uniform pseudodifferential operators, that we are interested in, we have to look at uniform $K$-homology more closely (i.e., \Spakula did not construct the Kasparov product for it, but we need it crucially to show homotopy invariance of uniform $K$-homology and therefore we will have to do this construction here). Then we have to identify the corresponding dual theory to uniform $K$-homology and prove the uniform $K$-\Poincare duality theorem. With this at our disposal we will then be able to reduce the uniform index problem for uniform pseudodifferential operators to Dirac operators (for this we will also have to prove a uniform version of the Thom isomorphism in order to be able to conclude that symbol classes of uniform pseudodifferential operators may be represented by symbol classes of Dirac operators). So it remains to show the uniform index theorem for Dirac operators. But since up to this point we will already have set up all the needed machinery, its proof will be basically the same as the proof of the local index theorem of Connes and Moscovici in \cite{connes_moscovici}.

Let us highlight some of the above mentioned theorems that we will prove along our path to the index theorems stated at the beginning of the introduciton. Let us start with two properties of uniform $K$-homology that were not proved by \Spakula in his paper where he introduced this theory, but which will be crucial to us.

\begin{thmexternalprodhomology*}
Let $X,Y$ be locally compact and separable metric spaces of jointly bounded geometry. Then there exists an associative product
\[\times\colon K^u_p(X) \otimes K^u_q(Y) \to K^u_{p+q}(X \times Y)\]
having the same properties as the usual external product in $K$-homology of compact spaces.
\end{thmexternalprodhomology*}

\begin{cor*}[Theorem~\ref{thm:weak_homotopy_equivalence_K_hom}]
Weakly homotopic uniform Fredholm modules define the same uniform $K$-homology class.
\end{cor*}

A fundamental result for this paper is of course that symmetric and elliptic uniform pseudodifferential operators define uniform $K$-homology classes and that these classes do only depend on the principal symbols of the operators. Note that to prove this result we will have to derive a lot of analytical properties of uniform pseudodifferential operators.

\begin{thm*}[Theorem~\ref{thm:elliptic_symmetric_PDO_defines_uniform_Fredholm_module} and Proposition~\ref{prop:same_symbol_same_k_hom_class}]
Let $P$ be a symmetric and elliptic uniform pseudodifferential operator acting on a vector bundle of bounded geometry over a manifold $M$ of bounded geometry.

Then $P$ defines naturally a uniform $K$-homology class $[P] \in K_\ast^u(M)$ and this class does only depend on the principal symbol of $P$.
\end{thm*}

The last collection of results that we want to highlight in this introduction are all the various uniform \Poincare duality and uniform Chern character isomorphisms. Note that the uniform $K$-theory and the uniform de Rham homology appearing in the following theorem are for the first time defined in this article, and note also that we are not mentioning some intermediate steps like the identification of uniform de Rham homology with periodic cyclic cohomology of a certain Sobolev space over $M$.

\begin{thm*}[Theorem~\ref{thm23w2}, Theorem~\ref{thm:Poincare_de_Rham} and Theorem~\ref{thm:Poincare_duality_K}]
Let $M$ be an $m$-dimensional manifold of bounded geometry and without boundary. Then the Chern characters induce linear, continuous isomorphisms
\[K^\ast_u(M) \barotimes \IC \cong \HbdR^\ast(M) \text{ and } K^u_\ast(M) \barotimes \IC \cong \HudR_\ast(M).\]
If $M$ is oriented we have also an isomorphism
\[\HbdR^\ast(M) \cong \HudR_{m-\ast}(M)\]
and if $M$ is \spinc we additionally have
\[K_u^\ast(M) \cong K^\ast_{m-\ast}(M).\]
\end{thm*}

\paragraph{Acknowledgements} This article arose out of the Ph.D.\ thesis \cite{engel_phd} of the author written at the University of Augsburg. Therefore, first and foremost, I would like to thank my doctoral advisor Professor Bernhard Hanke for his support and guidance. Without his encouragement throughout the last years, my thesis, and so also this article, would not have been possible.

Further, I want to express my gratitude to Professor John Roe from Pennsylvania State University with whom I had many inspirational conversations especially during my research stay there in summer 2013. I would also like to thank the Professors Alexandr Sergeevich Mishchenko and Evgeni Vadimovich Troitsky from Lomonosov Moscow State University, with whom I had the pleasure to work with during my stay in Moscow from September 2010 to June 2011. Moreover, I thank Professors Ulrich Bunke and Bernd Ammann from the University of Regensburg, with whom I had fruitful discussions related to this topic.

It is a pleasure for me to thank all my fellow graduate students at the University of Augsburg for the delightful time I spent with them, and furthermore I thank all my colleagues from the University of Regensburg, since also with them I had a good time.

I gratefully acknowledge the financial and academic support from the graduate program TopMath of the Elite Network of Bavaria, the TopMath Graduate Center of TUM Graduate School at the Technische Universität München, the German National Academic Foundation (Studienstiftung des deutschen Volkes) and from the SFB 1085 ``Higher Invariants'' situated at the University of Regensburg and financed by the DFG (Deutsche Forschungsgemeinschaft) .


\section{Uniform pseudodifferential operators on manifolds of bounded geometry}\label{sec:PDOs}

Let us explain why on non-compact manifolds we necessarily have to look at uniform pseudodifferential operators. Recall that on $\IR^m$ an operator $P$ is called pseudodifferential, if it is given by
\[(Pu)(x) = (2 \pi)^{-n/2} \int_{\IR^m} e^{i \langle x, \xi\rangle} p(x, \xi) \hat u(\xi) \ d\xi,\]
where $\hat u$ denotes the Fourier transform of $u$ and the function $p(x,\xi)$ satisfies for some $k \in\IZ$ the estimates $\|D_x^\alpha D_\xi^\beta p(x,\xi)\| \le C^{\alpha \beta} (1 + |\xi|)^{k - |\beta|}$ for all multi-indices $\alpha$ and $\beta$. On manifolds one calls an operator pseudodifferential if one has the above representation in any local chart. But if the manifold is not compact, we get the problem that this is not sufficient to guarantee that the operator has continuous extensions to Sobolev spaces.\footnote{We are ignoring in this discussion the fact that on non-compact manifolds we also need a condition on the behaviour of the integral kernel of $P$ at infinity. So assume for the moment that $P$ has finite propagation. Our final definition will require $P$ to be quasilocal at infinity (see Section \ref{sec:quasiloc}).} For this we additionally have to require that the above bounds $C^{\alpha \beta}$ are uniform across all the local charts. But since this is not well-defined (choosing a different atlas may distort the bounds arbitrarily large across the charts of the atlas), we will have to restrict the charts to exponential charts and additionally we will have to assume that our manifold has bounded geometry (these restrictions become clear when one looks at Lemma \ref{lem:transition_functions_uniformly_bounded}).

The local definition of uniform pseudodifferential operators that we investigate was already given by Kordyukov in \cite{kordyukov}, by Shubin in \cite{shubin} and by Taylor in \cite{taylor_pseudodifferential_operators_lectures}. Besides the local uniformity one also needs to control the behaviour of the integral kernels of these operators at infinity. One possibility is to impose finite propagation, i.e., demanding that there is an $R > 0$ such that the integral kernel $k(x,y)$ of the pseudodifferential operator vanishes for all $x,y$ with $d(x,y) > R$ (note that pseudodifferential operators always have an integral kernel that is smooth outside the diagonal). More generally, one can require an exponential decay of the integral kernel at infinity, and usually this decay should be faster than the volume growth of the manifold. In the present article we will require that our pseudodifferential operators are quasilocal\footnote{An operator $A \colon H^r(E) \to H^s(F)$ is \emph{quasilocal}, if there is some function $\mu\colon \IR_{> 0} \to \IR_{\ge 0}$ with $\mu(R) \to 0$ for $R \to \infty$ and such that for all $L \subset M$ and all $u \in H^r(E)$ with $\supp u \subset L$ we have $\|A u\|_{H^s, M - B_R(L)} \le \mu(R) \cdot \|u\|_{H^r}$.}, since this seems to be in a certain sense the most general notion which we may impose (see, e.g., the proof of Corollary \ref{cor:quasilocal_neg_order_uniformly_locally_compact} for how quasi-locality is used).

Let us explain why we want our operators to be quasilocal: Recall that in order to compute Roe's analytic index of an operator $D$ of Dirac type, we have to consider the operator $f(D)$, where $f$ is an even Schwartz function with $f(0) = 1$. Now usually $f(D)$ will not have finite propagation, but it will be a quasilocal operator. This was proven by Roe and we will generalize this crucial fact to uniform pseudodifferential operators. So though we could restrict to finite propagation uniform pseudodifferential operators and use the fact that $f(P)$ will be quasilocal where we need it, we would leave by this our class of finite propagation operators that we consider. So working from the beginning with quasilocal operators leads to the fact that we never have to leave this class. Note that the proof of the fact that $f(P)$ is quasilocal requires substantial analysis and is one of our key technical lemmas.

\subsection{Bounded geometry}

We will define in this section the notion of bounded geometry for manifolds and for vector bundles and discuss basic facts about uniform $C^r$-spaces and Sobolev spaces on them.

\begin{defn}
We will say that a Riemannian manifold $M$ has \emph{bounded geometry}, if
\begin{itemize}
\item the curvature tensor and all its derivatives are bounded, i.e., $\| \nabla^k \Rm (x) \| < C_k$ for all $x \in M$ and $k \in \IN_0$, and
\item the injectivity radius is uniformly positive, i.e., $\injrad_M(x) > \varepsilon > 0$ for all points $x \in M$ and for a fixed $\varepsilon > 0$.
\end{itemize}
If $E \to M$ is a vector bundle with a metric and compatible connection, we say that \emph{$E$ has bounded geometry}, if the curvature tensor of $E$ and all its derivatives are bounded.
\qed
\end{defn}

\begin{examples}
There are plenty of examples of manifolds of bounded geometry. The most important ones are coverings of compact Riemannian manifolds equipped with the pull-back metric, homogeneous manifolds with an invariant metric, and leafs in a foliation of a compact Riemannian manifold (this is proved by Greene in \cite[lemma on page 91 and the paragraph thereafter]{greene}).

For vector bundles, the most important examples are of course again pull-back bundles of bundles over compact manifolds equipped with the pull-back metric and connection, and the tangent bundle of a manifold of bounded geometry.

Furthermore, if $E$ and $F$ are two vector bundles of bounded geometry, then the dual bundle $E^\ast$, the direct sum $E \oplus F$, the tensor product $E \otimes F$ (and so especially also the homomorphism bundle $\Hom(E, F) = F \otimes E^\ast$) and all exterior powers $\Lambda^l E$ are also of bounded geometry. If $E$ is defined over $M$ and $F$ over $N$, then their external tensor product\footnote{The fiber of $E \boxtimes F$ over the point $(x,y) \in M \times N$ is given by $E_x \otimes F_y$.} $E \boxtimes F$ over $M \times N$ is also of bounded geometry.
\qed
\end{examples}

Greene proved in \cite[Theorem 2']{greene} that there are no obstructions against admitting a metric of bounded geometry, i.e., every smooth manifold without boundary admits one. On manifolds of bounded geometry there is also no obstruction for a vector bundle to admit a metric and compatible connection of bounded geometry. The proof (i.e., the construction of the metric and the connection) is done in a uniform covering of $M$ by normal coordinate charts and subordinate uniform partition of unity (we will discuss these things in a moment) and we have to use the local characterization of bounded geometry for vector bundles from Lemma \ref{lem:equiv_characterizations_bounded_geom_bundles}.

We will now state an important characterization in local coordinates of bounded geometry since it will allow us to show that certain local definitions (like the one of uniform pseudodifferential operators) are independent of the chosen normal coordinates.

\begin{lem}[{\cite[Appendix A1.1]{shubin}}]\label{lem:transition_functions_uniformly_bounded}
Let the injectivity radius of $M$ be positive.

Then the curvature tensor of $M$ and all its derivatives are bounded if and only if for any $0 < r < \injrad_M$ all the transition functions between overlapping normal coordinate charts of radius $r$ are uniformly bounded, as are all their derivatives (i.e., the bounds can be chosen to be the same for all transition functions).
\end{lem}

Another fact which we will need about manifolds of bounded geometry is the existence of uniform covers by normal coordinate charts and corresponding partitions of unity. A proof may be found in, e.g., \cite[Appendix A1.1]{shubin} (Shubin addresses the first statement about the existence of such covers actually to the paper \cite{gromov_curvature_diameter_betti_numbers} of Gromov).

\begin{lem}\label{lem:nice_coverings_partitions_of_unity}
Let $M$ be a manifold of bounded geometry.

For every $0 < \varepsilon < \tfrac{\injrad_M}{3}$ there exists a covering of $M$ by normal coordinate charts of radius $\varepsilon$ with the properties that the midpoints of the charts form a uniformly discrete set in $M$ and that the coordinate charts with double radius $2\varepsilon$ form a uniformly locally finite cover of $M$.

Furthermore, there is a subordinate partition of unity $1 = \sum_i \varphi_i$ with $\supp \varphi_i \subset B_{2\varepsilon}(x_i)$, such that in normal coordinates the functions $\varphi_i$ and all their derivatives are uniformly bounded (i.e., the bounds do not depend on $i$).
\end{lem}

If the manifold $M$ has bounded geometry, we have analogous equivalent local characterizations of bounded geometry for vector bundles as for manifolds. The equivalence of the first two bullet points in the next lemma is stated in, e.g., \cite[Proposition 2.5]{roe_index_1}. Concerning the third bullet point, the author could not find any citable reference in the literature (though Shubin uses in \cite{shubin} this as the actual definition).

\begin{lem}\label{lem:equiv_characterizations_bounded_geom_bundles}
Let $M$ be a manifold of bounded geometry and $E \to M$ a vector bundle. Then the following are equivalent:

\begin{itemize}
\item $E$ has bounded geometry,
\item the Christoffel symbols $\Gamma_{i \alpha}^\beta(y)$ of $E$ with respect to synchronous framings (considered as functions on the domain $B$ of normal coordinates at all points) are bounded, as are all their derivatives, and this bounds are independent of $x \in M$, $y \in \exp_x(B)$ and $i, \alpha, \beta$, and
\item the matrix transition functions between overlapping synchronous framings are uniformly bounded, as are all their derivatives (i.e., the bounds are the same for all transition functions).
\end{itemize}
\end{lem}

We will now give the definition of uniform $C^\infty$-spaces together with a local characterization on manifolds of bounded geometry. The interested reader is refered to, e.g., the papers \cite[Section 2]{roe_index_1} or \cite[Appendix A1.1]{shubin} of Roe and Shubin for more information regarding these uniform $C^\infty$-spaces.

\begin{defn}[$C^r$-bounded functions]
Let $f \in C^\infty(M)$. We will say that $f$ is a \emph{$C_b^r$-function}, or equivalently that it is \emph{$C^r$-bounded}, if $\| \nabla^i f \|_\infty < C_i$ for all $0 \le i \le r$.
\qed
\end{defn}

If $M$ has bounded geometry, being $C^r$-bounded is equivalent to the statement that in every normal coordinate chart $|\partial^\alpha f(y)| < C_\alpha$ for every multiindex $\alpha$ with $|\alpha| \le r$ (where the constants $C_\alpha$ are independent of the chart).

Of course, the definition of $C^r$-boundedness and its equivalent characterization in normal coordinate charts for manifolds of bounded geometry make also sense for sections of vector bundles of bounded geometry (and so especially also for vector fields, differential forms and other tensor fields).

\begin{defn}[Uniform $C^\infty$-spaces]
\label{defn:uniform_frechet_spaces}
Let $E$ be a vector bundle of bounded geometry over $M$. We will denote the \emph{uniform $C^r$-space} of all $C^r$-bounded sections of $E$ by $C_b^r(E)$.

Furthermore, we define the \emph{uniform $C^\infty$-space $C_b^\infty(E)$}
\[C_b^\infty(E) := \bigcap_r C_b^r(E)\]
which is a \Frechet space.
\qed
\end{defn}

Now we get to Sobolev spaces on manifolds of bounded geometry. Much of the following material is from \cite[Appendix A1.1]{shubin} and \cite[Section 2]{roe_index_1}, where an interested reader can find more thorough discussions of this matters.

Let $s \in C^\infty_c(E)$ be a compactly supported, smooth section of some vector bundle $E \to M$ with metric and connection $\nabla$. For $k \in \IN_0$ and $p \in [1,\infty)$ we define the global $W^{k,p}$-Sobolev norm of $s$ by
\begin{equation}\label{eq:sobolev_norm}
\|s\|_{W^{k,p}}^p := \sum_{i=0}^k \int_M \|\nabla^i s(x)\|^p dx.
\end{equation}

\begin{defn}[Sobolev spaces $W^{k,p}(E)$]\label{defn:sobolev_spaces}
Let $E$ be a vector bundle which is equipped with a metric and a connection. The \emph{$W^{k,p}$-Sobolev space of $E$} is the completion of $C^\infty_c(E)$ in the norm $\|\largecdot\|_{W^{k,p}}$ and will be denoted by $W^{k,p}(E)$.
\qed
\end{defn}

If $E$ and $M^m$ both have bounded geometry than the Sobolev norm \eqref{eq:sobolev_norm} for $1 < p < \infty$ is equivalent to the local one given by
\begin{equation}\label{eq:sobolev_norm_local}
\|s\|_{W^{k,p}}^p \stackrel{\text{equiv}}= \sum_{i=1}^\infty \|\varphi_i s\|^p_{W^{k,p}(B_{2\varepsilon}(x_i))},
\end{equation}
where the balls $B_{2\varepsilon}(x_i)$ and the subordinate partition of unity $\varphi_i$ are as in Lemma \ref{lem:nice_coverings_partitions_of_unity}, we have chosen synchronous framings and $\|\largecdot\|_{W^{k,p}(B_{2\varepsilon}(x_i))}$ denotes the usual Sobolev norm on $B_{2\varepsilon}(x_i) \subset \IR^m$. This equivalence enables us to define the Sobolev norms for all $k \in \IR$, see Triebel \cite{triebel_2} and Gro{\ss}e--Schneider \cite{grosse_sobolev}. There are some issues in the case $p=1$, see the discussion by Triebel \cite[Section~2.2.3]{triebel_1}, \cite[Remark~4 on Page~13]{triebel_2}.

Assuming bounded geometry, the usual embedding theorems are true:

\begin{thm}[{\cite[Theorem 2.21]{aubin_nonlinear_problems}}]\label{thm:sobolev_embedding}
Let $E$ be a vector bundle of bounded geometry over a manifold $M^m$ of bounded geometry and without boundary.

Then we have for all values $(k-r)/m > 1/p$ continuous embeddings
\[W^{k,p}(E) \subset C^r_b(E).\]
\end{thm}

We define the space
\begin{equation}
\label{eq:defn_W_infty}
W^{\infty,p}(E) := \bigcap_{k \in \IN_0} W^{k,p}(E)
\end{equation}
and equip it with the obvious \Frechet topology. The Sobolev Embedding Theorem tells us now that we have for all $p$ a continuous embedding
\[W^{\infty,p}(E) \hookrightarrow C^\infty_b(E).\]

For $p=2$ we will write $H^k(E)$ for $W^{k,2}(E)$. This are Hilbert spaces and for $k < 0$ the space $H^k(E)$ coincides with the dual of $H^{-k}(E)$, regarded as a space of distributional sections of $E$.

We will now investigate the Sobolev spaces $H^\infty(E)$ and $H^{-\infty}(E)$ of infinite orders. They are crucial since they will allow us to define smoothing operators and hence the important algebra $\IU(E)$ in the next section.

\begin{lem}
The topological dual of $H^\infty(E)$ is given by
\[H^{-\infty}(E) := \bigcup_{k \in \IN_0} H^{-k}(E).\]
\end{lem}

Let us equip the space $H^{-\infty}(E)$ with the locally convex topology defined as follows: the \Frechet space $H^\infty(E) = \operatorname{\underleftarrow{\lim}} H^k(E)$ is the projective limit of the Banach spaces $H^k(E)$, so using dualization we may put on the space $H^{-\infty}(E)$ the inductive limit topology denoted $\iota(H^{-\infty}(E), H^\infty(E))$:
\[H^{-\infty}_\iota(E) := \operatorname{\underrightarrow{\lim}} H^{-k}(E).\]
It enjoys the following universal property: a linear map $A \colon H^{-\infty}_\iota(E) \to F$ to a locally convex topological vector space $F$ is continuous if and only if $A|_{H^{-k}(E)}\colon H^{-k}(E) \to F$ is continuous for all $k \in \IN_0$.

Later we will need to know how the bounded\footnote{A subset $B \subset H^{-\infty}_\iota(E)$ is \emph{bounded} if and only if for all open neighbourhoods $U \subset H^{-\infty}_\iota(E)$ of $0$ there exists $\lambda > 0$ with $B \subset \lambda U$.} subsets of $H^{-\infty}_\iota(E)$ look like, which is the content of the following lemma. In its proof we will also deduce, solely for the enjoyment of the reader, some nice properties of the spaces $H^\infty(E)$ and $H^{-\infty}_\iota(E)$.

\begin{lem}\label{lem:regular_inductive_limit}
The space $H^{-\infty}_\iota(E) := \operatorname{\underrightarrow{\lim}} H^{-k}(E)$ is a \emph{regular inductive limit}, i.e., for every bounded subset $B \subset H^{-\infty}_\iota(E)$ exists some $k \in \IN_0$ such that $B$ is already contained in $H^{-k}(E)$ and bounded there.\footnote{Note that the converse does always hold for inductive limits, i.e., if $B \subset H^{-k}(E)$ is bounded, then it is also bounded in $H^{-\infty}_\iota(E)$.}
\end{lem}

\begin{proof}
Since all $H^{-k}(E)$ are \Frechet spaces, we may apply the following corollary of Grothendieck's Factorization Theorem: the inductive limit $H^{-\infty}_\iota(E)$ is regular if and only if it is locally complete (see, e.g., \cite[Lemma 7.3.3(i)]{perezcarreras_bonet}). To avoid introducing more burdensome vocabulary, we won't define the notion of local completeness here since we will show something stronger: $H^{-\infty}_\iota(E)$ is actually complete\footnote{That is to say, every Cauchy net converges. In locally convex spaces, being Cauchy and to converge is meant with respect to each of the seminorms simultaneously.}.

From \cite[Sections 3.(a \& b)]{bierstedt_bonet} we conclude the following: since each $H^k(E)$ is a Hilbert space, the \Frechet space $H^\infty(E)$ is the projective limit of reflexive Banach spaces and therefore totally reflexive\footnote{That is to say, every quotient of it is reflexive, i.e., the canonical embeddings of the quotients into their strong biduals are isomorphisms of topological vector spaces.}. It follows that $H^\infty(E)$ is distinguished, which can be characterized by $H^{-\infty}_\beta(E) = H^{-\infty}_\iota(E)$, where $\beta(H^{-\infty}(E), H^\infty(E))$ is the strong topology on $H^{-\infty}(E)$. Now without defining the strong topology we just note that strong dual spaces of \Frechet space are always complete.
\end{proof}

\subsection{Quasilocal smoothing operators}\label{sec:quasiloc}

We will discuss in this section the definition and basic properties of smoothing operators on manifolds of bounded geometry and we will introduce the notion of quasilocal operators. The quasilocal smoothing operators will be the $(-\infty)$-part of our uniform pseudodifferential operators that we are going to define in the next section.

\begin{defn}[Smoothing operators]
Let $M$ be a manifold of bounded geometry and $E$ and $F$ two vector bundles of bounded geometry over $M$. We will call a continuous linear operator $A \colon H^{-\infty}_\iota(E) \to H^\infty(F)$ a \emph{smoothing operator}.
\qed
\end{defn}

\begin{lem}\label{lem:smoothing_operator_iff_bounded}
A linear operator $A \colon H^{-\infty}_\iota(E) \to H^\infty(F)$ is continuous if and only if it is bounded as an operator $H^{-k}(E) \to H^l(F)$ for all $k, l \in \IN_0$.
\end{lem}

Let us denote by $\IB(H^{-\infty}_\iota(E), H^\infty(E))$ the algebra of all smoothing operators from $E$ to itself. Due to the above lemma we may equip it with the countable family of norms $(\| \largecdot \|_{-k,l})_{k,l \in \IN_0}$ so that it becomes a \Frechet space\footnote{That is to say, a topological vector space whose topology is Hausdorff and induced by a countable family of seminorms such that it is complete with respect to this family of seminorms.}.

Now let us get to the main property of smoothing operators that we will need, namely that they can be represented as integral operators with a uniformly bounded smooth kernel. Let $A \colon H^{-\infty}_\iota(E) \to H^\infty(F)$ be given. Then we get by the Sobolev Embedding Theorem \ref{thm:sobolev_embedding} a continuous operator $A \colon H^{-\infty}_\iota(E) \to C_b^\infty(F)$ and so may conclude by the Schwartz Kernel Theorem for regularizing operators\footnote{Note that the usual wording of the Schwartz Kernel Theorem for regularizing operators requires the domain $H^{-\infty}(E)$ to be equipped with the weak-$^\ast$ topology $\sigma(H^{-\infty}(E), H^{\infty}(F))$ and $A$ to be continuous against it. But one actually only needs the domain to be equipped with the inductive limit topology. To see this, one can look at the proof of the Schwartz Kernel Theorem for regularizing kernels as in, e.g., \cite[Theorem 3.18]{ganglberger}.} that $A$ has a smooth integral kernel $k_A \in C^\infty(F \boxtimes E^\ast)$, which is uniformly bounded as are all its derivatives, because of the bounded geometry of $M$ and the vector bundles $E$ and $F$, i.e., $k_A \in C_b^\infty(F \boxtimes E^\ast)$.

From the proof of the Schwartz Kernel Theorem for regularizing operators we also see that the assignment of the kernel to the operator is continuous against the \Frechet topology on $\IB(H^{-\infty}_\iota(E), H^\infty(F))$. Furthermore, because of Lemma \ref{lem:regular_inductive_limit} this topology coincides with the topology of bounded convergence\footnote{A basis of neighbourhoods of zero for the topology of bounded convergence is given by the subsets $M(B, U) \subset \IB(H^{-\infty}_\iota(E), H^\infty(F))$ of all operators $T$ with $T(B) \subset U$, where $B$ ranges over all bounded subsets of $H^{-\infty}_\iota(E)$ and $U$ over a basis of neighbourhoods of zero in $H^\infty(F)$.} on $\IB(H^{-\infty}_\iota(E), H^\infty(F))$. We need this equality of topologies in order for the next proposition (which is a standard result in distribution theory) to be equivalent to the version stated in \cite[Proposition 2.9]{roe_index_1}.

\begin{prop}\label{prop:smoothing_op_kernel}
Let $A\colon H^{-\infty}_\iota(E) \to H^\infty(F)$ be a smoothing operator. Then $A$ can be written as an integral operator with kernel $k_A \in C_b^\infty(F \boxtimes E^\ast)$. Furthermore, the map
\[\IB(H^{-\infty}_\iota(E), H^\infty(F)) \to C_b^\infty(F \boxtimes E^\ast),\]
which associates a smoothing operator its kernel, is continuous.
\end{prop}

Let $L \subset M$ be any subset. We will denote by $\|\largecdot\|_{H^r, L}$ the seminorm on the Sobolev space $H^r(E)$ given by
\[\|u\|_{H^r, L} := \inf \{ \|u^\prime\|_{H^r} \ | \ u^\prime \in H^r(E), u^\prime = u \text{ on a neighbourhood of }L\}.\]

\begin{defn}[Quasilocal operators, {\cite[Section 5]{roe_index_1}}]\label{defn:quasiloc_ops}
We will call a continuous operator $A \colon H^r(E) \to H^s(F)$ \emph{quasilocal}, if there is a function $\mu\colon \IR_{> 0} \to \IR_{\ge 0}$ with $\mu(R) \to 0$ for $R \to \infty$ and such that for all $L \subset M$ and all $u \in H^r(E)$ with $\supp u \subset L$ we have
\[\|A u\|_{H^s, M - B_R(L)} \le \mu(R) \cdot \|u\|_{H^r}.\]

Such a function $\mu$ will be called a \emph{dominating function for $A$}.

We will say that an operator $A\colon C_c^\infty(E) \to C^\infty(F)$ is a \emph{quasilocal operator of order $k$}\footnote{Roe calls such operators ``\emph{uniform} operators of order $k$'' in \cite[Definition 5.3]{roe_index_1}. But since the word ``uniform'' will have another meaning for us (see, e.g., the definition of uniform $K$-homology), we changed the name.} for some $k \in \IZ$, if $A$ has a continuous extension to a quasilocal operator $H^s(E) \to H^{s-k}(F)$ for all $s \in \IZ$.

A smoothing operator $A\colon H^{-\infty}_\iota(E) \to H^\infty(F)$ will be called \emph{quasilocal}, if $A$ is quasilocal as an operator $H^{-k}(E) \to H^l(F)$ for all $k,l \in \IN_0$ (from which it follows that $A$ is also quasilocal for all $k,l \in \IZ$).
\qed
\end{defn}

If we regard a smoothing operator $A$ as an operator $L^2(E) \to L^2(F)$, we get a uniquely defined adjoint $A^\ast\colon L^2(F) \to L^2(E)$. Its integral kernel will be given by
\[k_{A^\ast}(x,y) := k_A(y,x)^\ast \in C_b^\infty(E \boxtimes F^\ast),\]
where $k_A(y,x)^\ast \in F_y^\ast \otimes E_x$ is the dual element of $k_A(y,x) \in F_y \otimes E^\ast_x$.

\begin{defn}[cf. {\cite[Definition 5.3]{roe_index_1}}]\label{defn:quasiloc_smoothing}
We will denote the set of all quasilocal smoothing operators $A\colon H^{-\infty}_\iota(E) \to H^\infty(F)$ with the property that their adjoint operator $A^\ast$ is also a quasilocal smoothing operator $H^{-\infty}_\iota(F) \to H^\infty(E)$ by $\IU(E,F)$.

If $E=F$, we will just write $\IU(E)$.
\qed
\end{defn}

\begin{rem}
Roe defines in \cite[Definition 5.3]{roe_index_1} the algebra $\mathcal{U}_{-\infty}(E)$ instead of $\IU(E)$, i.e., he does not demand that the adjoint operator is also quasilocal smoothing. The reason why we do this is that we want adjoints of uniform pseudodifferential operators to be again uniform pseudodifferential operators (and the algebra $\IU(E)$ is used in the definition of uniform pseudodifferential operators).
\qed
\end{rem}

\subsection{Definition of uniform pseudodifferential operators}

Let $M^m$ be an $m$-dimensional manifold of bounded geometry and let $E$ and $F$ be two vector bundles of bounded geometry over $M$. Now we will get to the definition of uniform pseudodifferential operators acting on sections of vector bundles of bounded geometry over manifolds of bounded geometry.

Our definition is almost the same as the ones of Shubin \cite{shubin} and Kordyukov \cite{kordyukov}. The difference is that our definition is slightly more general, because we do not restrict to finite propagation operators (since we allow the term $P_{-\infty}$ in the definition below). The reason why we have to do this is because of our results in Section~\ref{sec:functions_of_PDOs}: there we always only get quasi-local operators and not operators of finite propagation (in fact, the main technical result of that section is probably Lemma~\ref{lem:exp(itP)_quasilocal} stating that the wave operators are quasi-local), and therefore we would leave our calculus of pseudodifferential operators if we would insist of them having finite propagation. So most of the results stated in this subsection and in Subsection~\ref{secio23ed} are basically already known, resp., it is straight-forward to generalize the corresponding statements in the finite propagation case to our quasi-local case here. We nevertheless include a discussion of these statements in order for our exposition here to be self-contained.

\begin{defn}\label{defn:pseudodiff_operator}
An operator $P\colon C_c^\infty(E) \to C^\infty(F)$ is a \emph{uniform pseudodifferential operator of order $k \in \IZ$}, if with respect to a uniformly locally finite covering $\{B_{2\varepsilon}(x_i)\}$ of $M$ with normal coordinate balls and corresponding subordinate partition of unity $\{\varphi_i\}$ as in Lemma \ref{lem:nice_coverings_partitions_of_unity} we can write
\begin{equation}
\label{eq:defn_pseudodiff_operator_sum}
P = P_{-\infty} + \sum_i P_i
\end{equation}
satisfying the following conditions:
\begin{itemize}
\item $P_{-\infty} \in \IU(E,F)$, i.e., it is a quasilocal smoothing operator,
\item for all $i$ the operator $P_i$ is with respect to synchronous framings of $E$ and $F$ in the ball $B_{2\varepsilon}(x_i)$ a matrix of pseudodifferential operators on $\IR^m$ of order $k$ with support\footnote{An operator $P$ is \emph{supported in a subset $K$}, if $\supp Pu \subset K$ for all $u$ in the domain of $P$ and if $Pu = 0$ whenever we have $\supp u \cap K = \emptyset$.} in $B_{2\varepsilon}(0) \subset \IR^m$, and
\item the constants $C_i^{\alpha \beta}$ appearing in the bounds
\[\|D_x^\alpha D_\xi^\beta p_i(x,\xi)\| \le C^{\alpha \beta}_i (1 + |\xi|)^{k - |\beta|}\]
of the symbols of the operators $P_i$ can be chosen to not depend on $i$, i.e., there are $C^{\alpha \beta} < \infty$ such that
\begin{equation}
\label{eq:uniformity_defn_PDOs}
C^{\alpha \beta}_i \le C^{\alpha \beta}
\end{equation}
for all multi-indices $\alpha, \beta$ and all $i$. We will call this the \emph{uniformity condition} for pseudodifferential operators on manifolds of bounded geometry.
\end{itemize}
We denote the set of all such operators by $\UPsiDO^k(E,F)$.
\qed
\end{defn}

\begin{rem}
From Lemma \ref{lem:transition_functions_uniformly_bounded} and Lemma \ref{lem:equiv_characterizations_bounded_geom_bundles} together with \cite[Theorem III.§3.12]{lawson_michelsohn} (and its proof which gives the concrete formula how the symbol of a pseudodifferential operator transforms under a coordinate change) we conclude that the above definition of uniform pseudodifferential operators on manifolds of bounded geometry does neither depend on the chosen uniformly locally finite covering of $M$ by normal coordinate balls, nor on the subordinate partition of unity with uniformly bounded derivatives, nor on the synchronous framings of $E$ and $F$.
\qed
\end{rem}

\begin{rem}
We could also have given an equivalent definition of uniform pseudodifferential operators, which does not need a choice of covering: firstly, for each $\varepsilon > 0$ there should be a quasilocal smoothing operator $P_\varepsilon$ such that for any $\phi, \psi \in C_c^\infty(M)$ with $d(\supp \phi, \supp \psi) > \varepsilon$ and any $v \in C_c^\infty(E)$ we have $\psi \cdot P(\phi \cdot v) = \psi \cdot P_\varepsilon(\phi \cdot v)$. This encodes that the integral kernel of a uniform pseudodifferential operator $P$ is off-diagonally a quasilocal smoothing operator.

Secondly, to encode the behaviour of the integral kernel of $P$ at its diagonal, we must demand that in any normal coordinate chart of radius less than the injectivity radius of the manifold with any choice of cut-off function for this coordinate chart and with any choice of synchronous framings of $E$ and $F$ in this coordinate chart the operator $P$ looks like a pseudodifferential operator on $\IR^m$, and for the collection of all of these local representatives of $P$ computed with respect to cut-off functions that have common bounds on their derivatives we have the Uniformity Condition \eqref{eq:uniformity_defn_PDOs}.
\qed
\end{rem}

\begin{rem}
We consider only operators that would correspond to H\"ormander's class $S_{1, 0}^k(\Omega)$, if we consider open subsets $\Omega$ of $\IR^m$ instead of an $m$-dimensional manifold $M$, i.e., we do not investigate operators corresponding to the more general classes $S_{\rho, \delta}^k(\Omega)$. The paper \cite[Definition 2.1]{hormander_ess_norm} is the one where H\"ormander introduced these classes.
\qed
\end{rem}

Recall that in the case of compact manifolds a pseudodifferential operator $P$ of order~$k$ has an extension to a continuous operator $H^s(E) \to H^{s-k}(F)$ for all $s \in \IZ$ (see, e.g., \cite[Theorem III.§3.17(i)]{lawson_michelsohn}). Due to the uniform local finiteness of the sum in \eqref{eq:defn_pseudodiff_operator_sum} and due to the Uniformity Condition \eqref{eq:uniformity_defn_PDOs}, this result does also hold in our case of a manifold of bounded geometry.

\begin{prop}\label{prop:pseudodiff_extension_sobolev}
Let $P \in \UPsiDO^k(E,F)$. Then $P$ has for all $s \in \IZ$ an extension to a continuous operator $P\colon H^s(E) \to H^{s-k}(F)$.
\end{prop}

\begin{rem}\label{rem:bound_operator_norm_PDO}
Later we will need the following fact: we can bound the operator norm of $P\colon H^s(E) \to H^{s-k}(F)$ from above by the maximum of the constants $C^{\alpha 0}$ with $|\alpha| \le K_s$ from the Uniformity Condition \eqref{eq:uniformity_defn_PDOs} for $P$ multiplied with a constant $C_s$, where $K_s \in \IN_0$ and $C_s$ only depend on $s \in \IZ$ and the dimension of the manifold $M$. This can be seen by carefully examining the proof of \cite[Proposition III.§3.2]{lawson_michelsohn} which is the above proposition for the compact case.\footnote{To be utterly concrete, we have to choose normal coordinate charts and a subordinate partition of unity as in Lemma \ref{lem:nice_coverings_partitions_of_unity} and also synchronous framings for $E$ and $F$ and then use Formula \eqref{eq:sobolev_norm_local} which gives Sobolev norms that can be computed locally and that are equivalent to the global norms \eqref{eq:sobolev_norm}.}
\qed
\end{rem}

Let us define
\[\UPsiDO^{-\infty}(E,F) := \bigcap_k \UPsiDO^k(E,F).\]

We will show $\UPsiDO^{-\infty}(E,F) = \IU(E,F)$: from the previous Proposition \ref{prop:pseudodiff_extension_sobolev} we conclude that $P \in \UPsiDO^{-\infty}(E,F)$ is a smoothing operator (using Lemma \ref{lem:smoothing_operator_iff_bounded}). Since we can write $P = P_{-\infty} + \sum_i P_i$, where $P_{-\infty} \in \IU(E,F)$ and the $P_i$ are supported in balls with uniformly bounded radii, the operator $\sum_i P_i$ is of finite propagation. So $P$ is the sum of a quasilocal smoothing operator $P_{-\infty}$ and a smoothing operator $\sum_i P_i$ of finite propagation, and therefore a quasilocal smoothing operator. The same arguments also apply to the adjoint $P^\ast$ of $P$, so that in the end we can conclude $P \in \IU(E,F)$, i.e., we have shown $\UPsiDO^{-\infty}(E,F) \subset \IU(E,F)$.

Since the other inclusion does hold by definition, we get the claim.\footnote{Of course, our definition of pseudodifferential operators was arranged such that this lemma holds.}

\begin{lem}\label{lem:PDO_-infinity_equal_quasilocal_smoothing}
$\UPsiDO^{-\infty}(E,F) = \IU(E,F)$.
\end{lem}

One of the important properties of pseudodifferential operators on compact manifolds is that the composition of an operator $P \in \UPsiDO^k(E,F)$ and $Q \in \UPsiDO^l(F,G)$ is again a pseudodifferential operator of order $k+l$: $PQ \in \UPsiDO^{k+l}(E,G)$. We can prove this also in our setting by writing
\begin{align*}
PQ & = \Big(P_{-\infty} + \sum_i P_i\Big) \Big(Q_{-\infty} + \sum_j Q_j\Big)\\
& = P_{-\infty} Q_{-\infty} + \sum_i P_i Q_{-\infty} + \sum_j P_{-\infty} Q_j + \sum_{i,j} P_i Q_j
\end{align*}
and then arguing as follows.
\begin{itemize}
\item The first summand is an element of $\IU(E,G)$: in \cite[Proposition 5.2]{roe_index_1} it was shown that the composition of two quasilocal operators is again quasilocal and it is clear that composing smoothing operators again gives smoothing operators, resp. it is easy to see that composing two operators which may be approximated by finite propagation operators again gives such an operator.

\item The second and third summands are from $\IU(E,G)$ due to Proposition \ref{prop:pseudodiff_extension_sobolev} and since the sums are uniformly locally finite, the operators $P_i$ and $Q_j$ are supported in coordinate balls of uniform radii (i.e., have finite propagation which is uniformly bounded from above) and their operator norms are uniformly bounded due to the uniformity condition in the definition of pseudodifferential operators.

\item The last summand is a uniformly locally finite sum of pseudodifferential operators of order $k+l$ (here we use the corresponding result for compact manifolds) and to see the Uniformity Condition \eqref{eq:uniformity_defn_PDOs} we use \cite[Theorem III.§3.10]{lawson_michelsohn}: it states that the symbol of $P_i Q_j$ has formal development $\sum_\alpha \frac{i^{|\alpha|}}{\alpha !} (D_\xi^\alpha p_i)(D_x^\alpha q_j)$. So we may deduce the uniformity condition for $P_i Q_j$ from the one for $P_i$ and for $Q_j$.
\end{itemize}

Other properties that immediately generalize from the compact to the bounded geometry case is firstly, that the commutator of two uniform pseudodifferential operators whose symbols commute (Definition~\ref{defnnkdf893}) is of one order lower than it should a priori be, and secondly, that multiplication with a function $f \in C_b^\infty(M)$ defines a uniform pseudodifferential operator of order $0$.

So we have the following important proposition:

\begin{prop}\label{prop:PsiDOs_filtered_algebra}
$\UPsiDO^\ast(E)$ is a filtered $^\ast$-algebra, i.e., for all $k, l \in\IZ $ we have
\[\UPsiDO^k(E) \circ \UPsiDO^l(E) \subset \UPsiDO^{k+l}(E),\]
and so $\UPsiDO^{-\infty}(E)$ is a two-sided $^\ast$-ideal in $\UPsiDO^\ast(E)$.

Furthermore, we have $[P, Q] \in \UPsiDO^{k+l-1}(E)$ for $P \in \UPsiDO^k(E)$, $Q \in \UPsiDO^l(E)$, $k,l \in \IZ$, provided the symbols of $P$ and $Q$ commute.

Moreover, multiplication with a function $f \in C_b^\infty(M)$ defines a uniform pseudodifferential operator of order $0$ whose symbol commutes with any other symbol.
\end{prop}

The last property that generalizes to our setting and that we want to mention is the following (the proof of \cite[Theorem III.§3.9]{lawson_michelsohn} generalizes directly):

\begin{prop}\label{prop:Pu_smooth_if_u_smooth}
Let $P \in \UPsiDO^k(E,F)$ be a uniform pseudodifferential operator of arbitrary order and let $u \in H^s(E)$ for some $s \in \IZ$.

Then, if $u$ is smooth on some open subset $U \subset M$, $Pu$ is also smooth on $U$.
\end{prop}

\subsection{Uniformity of operators of nonpositive order}\label{sec:uniformity_PDOs}

Now we get to the important statement that the uniform pseudodifferential operators we have defined are, in fact, ``uniform'' in the meaning to be defined now (the discussion here is strongly related to the fact that symmetric and elliptic uniform pseudodifferential operators will define uniform $K$-homology classes).

Let $T \in \IK(L^2(E))$ be a compact operator. We know that $T$ is the limit of finite rank operators, i.e., for every $\varepsilon > 0$ there is a finite rank operator $k$ such that $\|T - k\| < \varepsilon$. Now given a collection $\mathcal{A} \subset \IK(L^2(E))$ of compact operators, it may happen that for every $\varepsilon > 0$ the rank needed to approximate an operator from $\mathcal{A}$ may be bounded from above by a common bound for all operators. This is formalized in the following definition.

\begin{defn}[Uniformly approximable collections of operators]\label{defn:uniformly_approximable_collection}
A collection of operators $\mathcal{A} \subset \IK(L^2(E))$ is said to be \emph{uniformly approximable}, if for every $\varepsilon > 0$ there is an $N > 0$ such that for every $T \in \mathcal{A}$ there is a rank-$N$ operator $k$ with $\|T - k\| < \varepsilon$.
\qed
\end{defn}

\begin{examples}\label{ex:uniformly_approximable_collections}
Every collection of finite rank operators with uniformly bounded rank is uniformly approximable.

Furthermore, every finite collection of compact operators is uniformly approximable and so also every totally bounded subset of $\IK(L^2(E))$.

The converse is in general false since a uniformly approximable family need not be bounded (take infinitely many rank-$1$ operators with operator norms going to infinity).

Even if we assume that the uniformly approximable family is bounded we do not necessarily get a totally bounded set: let $(e_i)_{i \in \IN}$ be an orthonormal basis of $L^2(E)$ and $P_i$ the orthogonal projection onto the $1$-dimensional subspace spanned by the vector $e_i$. Then the collection $\{P_i\} \subset \IK(L^2(E))$ is uniformly approximable (since all operators are of rank $1$) but not totally bounded (since $\|P_i - P_j\| = 1$ for $i \not= j$)\footnote{Another way to see that $\{P_i\}$ is not totally bounded is to use the characterization of totally bounded subsets of $\IK(H)$ from \cite[Theorem 3.5]{anselone_palmer}: a family $\mathcal{A} \subset \IK(H)$ is totally bounded if and only if both $\mathcal{A}$ and $\mathcal{A}^\ast$ are collectively compact, i.e., the sets $\{T v \ | \ T \in \mathcal{A}, v \in H \text{ with } \|v\| = 1\} \subset H$ and $\{T^\ast v \ | \ T \in \mathcal{A}, v \in H \text{ with } \|v\| = 1\} \subset H$ have compact closure.}.
\qed
\end{examples}

Let us define
\begin{equation*}\label{defn:L-Lip_R(M)}
\LLip_R(M) := \{ f \in C_c(M) \ | \ f \text{ is }L\text{-Lipschitz}, \diam(\supp f) \le R \text{ and } \|f\|_\infty \le 1\}.
\end{equation*}

\begin{defn}[{\cite[Definition 2.3]{spakula_uniform_k_homology}}]\label{defn:uniform_operators_manifold}
Let $T \in \IB(L^2(E))$. We say that $T$ is \emph{uniformly locally compact}, if for every $R, L > 0$ the collection
\[\{fT, Tf \ | \ f \in \LLip_R(M)\}\]
is uniformly approximable.

We say that $T$ is \emph{uniformly pseudolocal}, if for every $R, L > 0$ the collection
\[\{[T, f] \ | \ f \in \LLip_R(M)\}\]
is uniformly approximable.
\qed
\end{defn}

\begin{rem}\label{rem:renaming_l_dependence}
In \cite{spakula_uniform_k_homology} uniformly locally compact operators were called ``$l$-uniform'' and uniformly pseudolocal operators ``$l$-uniformly pseudolocal''.
\qed
\end{rem}

We will now show that uniform pseudodifferential operators of negative order are uniformly locally compact and that uniform pseudodifferential operators of order $0$ are uniformly pseudolocal. We will start with the operators of negative order.

\begin{prop}\label{prop:quasilocal_negative_order_uniformly_locally_compact}
Let $A\in \IB(L^2(E))$ be a finite propagation operator of negative order $k < 0$\footnote{See Definition \ref{defn:quasiloc_ops}. Note that we do not assume that $A$ is a pseudodifferential operator.} such that its adjoint operator $A^\ast$ also has finite propagation and is of negative order $k^\prime < 0$. Then $A$ is uniformly locally compact. Even more, the collection
\[\{fT, Tf \ | \ f \in B_R(M)\}\]
is uniformly approximable for every $R > 0$, where $B_R(M)$ consists of all bounded Borel functions $h$ on $M$ with $\diam(\supp h) \le R$ and $\|h\|_\infty \le 1$.
\end{prop}

\begin{proof}
Let $f \in B_R(M)$, $K := \supp f \subset M$ and $r$ be the propagation of $A$. The operator $\chi A f = A f$, where $\chi$ is the characteristic function of $B_r(K)$, factores as
\[L^2(E) \stackrel{\cdot f}\longrightarrow L^2(E|_K) \stackrel{\chi \cdot A}\longrightarrow H^{-k}(E|_{B_r(K)}) \hookrightarrow L^2(E|_{B_r(K)}) \to L^2(E).\]
The following properties hold:
\begin{itemize}
\item multiplication with $f$ has operator norm $\le 1$, since $\|f\|_\infty \le 1$, and analogously for the multiplication with $\chi$,
\item the norm of $\chi \cdot A\colon L^2(E|_K) \to H^{-k}(E|_{B_r(K)})$ can be bounded from above by the norm of $A\colon L^2(E) \to H^{-k}(E)$ (i.e., the upper bound does not depend on $K$ nor $r$),
\item the inclusion $H^{-k}(E|_{B_r(K)}) \hookrightarrow L^2(E|_{B_r(K)})$ is compact (due to  the Theorem of Rellich--Kondrachov) and this compactness is uniform, i.e., its approximability by finite rank operators\footnote{Here we mean the existence of an upper bound on the rank needed to approximate the operator by finite rank operators, given an $\varepsilon > 0$.} depends only on $R$ (the upper bound for the diameter of $\supp f$) and $r$, but not on $K$ (this uniformity is due to the bounded geometry of $M$ and of the bundles $E$ and $F$), and
\item the inclusion $L^2(E|_{B_r(K)}) \to L^2(E)$ is of norm $\le 1$.
\end{itemize}
From this we conclude that the operator $\chi A f = A f$ is compact and this compactness is uniform, i.e., its approximability by finite rank operators depends only on $R$ and $r$. So we can conclude that $\{Af \ | \ f \in B_R(M)\}$ is uniformly approximable.

Applying the same reasoning to the adjoint operator,\footnote{By assumption the adjoint operator also has finite propagation and is of negative order. So we conclude that $\{A^\ast f \ | \ f \in B_R(M)\}$ is uniformly approximable. But a collection $\mathcal{A}$ of compact operators is uniformly approximable if and only if the adjoint family $\mathcal{A}^\ast$ is uniformly approximable. So we get that $\{(A^\ast f)^\ast = \overline{f} A \ | \ f \in B_R(M)\}$ is uniformly approximable.} we conclude that $A$ is uniformly locally compact.
\end{proof}

Using an approximation argument\footnote{Note that we will not approximate the quasilocal operator $A$ itself by finite propagation operators in this argument. In fact, it is an open problem whether quasilocal operators may be approximated by finite propagation operators; see Section \ref{sec:quasilocs_approximable}.} we may also show the following corollary:

\begin{cor}\label{cor:quasilocal_neg_order_uniformly_locally_compact}
Let $A$ be a quasilocal operator of negative order and let the same hold true for its adjoint $A^\ast$. Then $A$ is uniformly locally compact; in fact, it even satisfies the stronger condition from the above Proposition \ref{prop:quasilocal_negative_order_uniformly_locally_compact}.
\end{cor}

\begin{proof}
We have to show that $\{Af \ | \ f \in B_R(M)\}$ is uniformly approximable. Let $\varepsilon > 0$ be given and let $r_\varepsilon$ be such that $\mu_A(r) < \varepsilon$ for all $r \ge r_\varepsilon$, where $\mu_A$ is the dominating function of $A$. Then $\chi_{B_{r_\varepsilon}(\supp f)} A f$ is $\varepsilon$-away from $Af$ and the same reasoning as in the proof of the above Proposition \ref{prop:quasilocal_negative_order_uniformly_locally_compact} shows that the approximability (up to an error of $\varepsilon$) of $\chi_{B_{r_\varepsilon}(\supp f)} A f$ does only depend on $R$ and $r_\varepsilon$. From this the claim that $\{Af \ | \ f \in B_R(M)\}$ is uniformly approximable follows.

Using the adjoint operator and the same arguments for it, we conclude that $A$ is uniformly locally compact.
\end{proof}

\begin{cor}\label{prop:PsiDOs_negative_order_uniformly_locally_compact}
Let $P \in \UPsiDO^k(E)$ be a uniform pseudodifferential operator of negative order $k < 0$. Then $P$ is uniformly locally compact.
\end{cor}

Let us now get to the case of uniform pseudodifferential operators of order $0$, where we want to show that such operators are uniformly pseudolocal.

Recall the following fact for compact manifolds: $T$ is pseudolocal\footnote{That is to say, $[T, f]$ is a compact operator for all $f \in C_c(M)$.} if and only if $f T g$ is a compact operator for all $f, g \in C(M)$ with disjoint supports. This observation is due to Kasparov and a proof might be found in, e.g., \cite[Proposition 5.4.7]{higson_roe}. We can add another equivalent characterization which is basically also proved in the cited proposition: an operator $T$ is pseudolocal if and only if $f T g$ is a compact operator for all bounded Borel functions $f$ and $g$ on $M$ with disjoint supports.

We have analogous equivalent characterizations for uniformly pseudolocal operators, which we will state in the following lemma. The proof of it is similar to the compact case (and uses the fact that the subset of all uniformly pseudolocal operators is closed in operator norm, which is proved in \cite[Lemma 4.2]{spakula_uniform_k_homology}). Furthermore, in order to prove that the Points 4 and 5 in the statement of the next lemma are equivalent to the other points we need the bounded geometry of $M$. For the convenience of the reader we will give a full proof of the lemma.

Let us introduce the notions $B_b(M)$ for all bounded Borel functions on $M$ and $B_R(M)$ for its subset consisting of all function $h$ with $\diam(\supp h) \le R$ and $\|h\|_\infty \le 1$.

\begin{lem}\label{lem:kasparov_lemma_uniform_approx_manifold}
The following are equivalent for an operator $T \in \IB(L^2(E))$:
\begin{enumerate}
\item $T$ is uniformly pseudolocal,
\item for all $R, L > 0$ the following collection is uniformly approximable:
\[\{f T g, g T f \ | \ f \in B_b(M), \ \! \|f\|_\infty \le 1, \ \! g \in \LLip_R(M), \ \! \supp f \cap \supp g = \emptyset\},\]
\item for all $R, L > 0$ the following collection is uniformly approximable:
\[\{f T g, g T f \ | \ f \in B_b(M), \ \! \|f\|_\infty \le 1, \ \! g \in B_R(M), \ \! d(\supp f, \supp g) \ge L\},\]
\item for every $L > 0$ there is a sequence $(L_j)_{j \in \IN}$ of positive numbers (not depending on the operator $T$) such that
\begin{align*}
\{ f T g, g T f \ | \ & f \in B_b(M)\text{ with }\|f\|_\infty \le 1,\\
& g \in B_R(M) \cap C_b^\infty(M)\text{ with }\|\nabla^j g\|_\infty \le L_j,\text{ and}\\
& \supp f \cap \supp g = \emptyset\}
\end{align*}
is uniformly approximable for all $R, L > 0$.
\item for every $L > 0$ there is a sequence $(L_j)_{j \in \IN}$ of positive numbers (not depending on the operator $T$) such that
\[\{ [T,g] \ | \ g \in B_R(M) \cap C_b^\infty(M)\text{ with }\|\nabla^j g\|_\infty \le L_j\}\]
is uniformly approximable for all $R, L > 0$.
\end{enumerate}
\end{lem}

\begin{proof}
\bm{$1 \Rightarrow 2$}\textbf{:} Let $f \in B_b(M)$ with $\|f \|_\infty \le 1$ and $g \in \LLip_R(M)$ have disjoint supports, i.e., $\supp f \cap \supp g = \emptyset$. From the latter we conclude $f T g = f [T,g]$, from which the claim follows (because $T$ is uniformly pseudolocal and because the operator norm of multiplication with $f$ is $\le 1$). Of course such an argument also works with the roles of $f$ and $g$ changed.

\bm{$2 \Rightarrow 3$}\textbf{:} Let $f \in B_b(M)$ with $\|f \|_\infty \le 1$ and $g \in B_R(M)$ with $d(\supp f, \supp g) \ge L$. We define $g^\prime(x) := \max\big( 0, 1 - 1/L \cdot d(x, \supp g) \big) \in 1/L\text{-}\operatorname{Lip}_{R+2L}(M)$. Since $g^\prime g = g$, the claim follows from writing $f T g = f T g^\prime g$ and because multiplication with $g$ has operator norm $\le 1$, and we of course also may change the roles of $f$ and $g$.

\bm{$3 \Rightarrow 1$}\textbf{:} Let $f \in \LLip_R(M)$. For given $\varepsilon > 0$ we partition the range of $f$ into a sequence of non-overlapping half-open intervals $U_1, \ldots, U_n$, each having diameter less than $\varepsilon$, such that $\overline{U_i}$ intersects $\overline{U_j}$ if and only if $|i - j| \le 1$. Denoting by $\chi_i$ the characteristic function of $f^{-1}(U_i)$, we get that $\chi_i \in B_R(M)$ if $0 \notin U_i$, since the support of $f$ has diameter less than or equal to $R$, and furthermore $d(\supp \chi_i, \supp \chi_j) \ge \tfrac{\varepsilon}{L}$ if $|i-j| > 1$, since $f$ is $L$-Lipschitz.

By Point 3 we have that the collections $\{\chi_i T \chi_j, \chi_j T \chi_i\}$ are uniformly approximable for all $i,j$ with $|i-j| > 1$. Choosing points $x_1, \ldots, x_n$ from $f^{-1}(U_1), \ldots, f^{-1}(U_n)$ and defining $f^\prime := f(x_1) \chi_1 + \cdots + f(x_n) \chi_n$, we get $\|f - f^\prime\|_\infty < \varepsilon$. The operator $[T,f]$ is $2\varepsilon \|T\|$-away from $[T, f^\prime]$, and since $\chi_1 + \cdots + \chi_n = 1$ we have
\[Tf^\prime - f^\prime T = \sum_{i,j} \chi_j T f(x_i)\chi_i - f(x_j) \chi_j T \chi_i.\]
Since we already know that $\{\chi_i T \chi_j, \chi_j T \chi_i\}$ are uniformly approximable for all $i,j$ with $|i-j| > 1$, it remains to treat the sum (note that the summand for $i=j$ is zero)
\[\sum_{|i-j|=1} \chi_j T f(x_i)\chi_i - f(x_j) \chi_j T \chi_i = \sum_{|i-j|=1} \big( f(x_i) - f(x_j) \big) \chi_j T \chi_i.\]
We split the sum into two parts, one where $i = j+1$ and the other one where $i=j-1$. The first part takes the form
\[\sum_j \big( f(x_{j+1}) - f(x_j) \big) \chi_j T \chi_{j+1},\]
i.e., is a direct sum of operators from $\chi_{j+1} \cdot L^2(E)$ to $\chi_j \cdot L^2(E)$. Therefore its norm is the maximum of the norms of its summands. But the latter are $\le 2 \varepsilon \|T\|$ since $|f(x_{j+1}) - f(x_j)| \le 2\varepsilon$. We treat the second part of the sum in the above display the same way and conclude that the sum in the above display is in norm $\le 4\varepsilon T$. Putting it all together it follows that $T$ is the operator norm limit of uniformly pseudolocal operators, from which it follows that $T$ itself is uniformly pseudolocal (it is proved in \cite[Lemma 4.2]{spakula_uniform_k_homology} that the uniformly pseudolocal operators are closed in operator norm, as are also the uniformly locally compact ones).

\bm{$2 \Rightarrow 4$}\textbf{:} Clear. We have to set $L_1 := L$ and the other values $L_{j \ge 2}$ do not matter (i.e., may be set to something arbitrary).

\bm{$4 \Rightarrow 3$}\textbf{:} This is similar to the proof of $2 \Rightarrow 3$, but we have to smooth the function $g^\prime$ constructed there. Let us make this concrete, i.e., let $f \in B_b(M)$ with $\|f \|_\infty \le 1$ and $g \in B_R(M)$ with $d(\supp f, \supp g) \ge L$ be given. We define
\[g^\prime(x) := \max\big( 0, 1 - 2 / L \cdot d(x, B_{L / 4}(\supp g)) \big) \in 2/L\text{-}\operatorname{Lip}_{R+3L/2}(M).\]
Note that $g^\prime \equiv 1$ on $B_{L/4}(\supp g)$ and $g^\prime \equiv 0$ outside $B_{3L/4}(\supp g)$. We cover $M$ by normal coordinate charts and choose a ``nice'' subordinate partition of unity $\varphi_i$ as in Lemma \ref{lem:nice_coverings_partitions_of_unity}. If $\psi$ is now a mollifier on $\IR^m$ supported in $B_{L/8}(0)$, we apply it in every normal coordinate chart to $\varphi_i g^\prime$ and reassemble then all the mollified parts of $g^\prime$ again to a (now smooth) function $g^\prime{}^\prime$ on $M$. This function $g^\prime{}^\prime$ is now supported in $B_{7L/8}(\supp g)$, and is constantly $1$ on $B_{L/8}(\supp g)$. So $f T g = f T g^\prime{}^\prime g$ from which we may conclude the uniform approximability of the collection $\{f T g\}$ for $f$ and $g$ satisfying $f \in B_b(M)$ with $\|f \|_\infty \le 1$ and $g \in B_R(M)$ with $d(\supp f, \supp g) \ge L$. Note that the constants $L_j$ appearing in $\|\nabla^j g^\prime{}^\prime\|_\infty \le L_j$ depend on $L$, $\varphi_i$ and $\psi$, but not on $f$, $g$ or $R$. The dependence on $\varphi_i$ and $\psi$ is ok, since we may just fix a particular choice of them (note that the choice of $\psi$ also depends on $L$), and the dependence on $L$ is explicitly stated in the claim.

Of course we may also change the roles of $f$ and $g$ in this argument.

\bm{$5 \Rightarrow 4$}\textbf{:} Clear. We just have to write $fTg = f[T,g]$ and analogously for $gTf$.

\bm{$1 \Rightarrow 5$}\textbf{:} Clear.
\end{proof}

With the above lemma at our disposal we may now prove the following proposition.

\begin{prop}\label{prop:PDO_order_0_l-uniformly-pseudolocal}
Let $P \in \UPsiDO^0(E)$. Then $P$ is uniformly pseudolocal.
\end{prop}

\begin{proof}
Writing $P = P_{-\infty} + \sum_i P_i$ with $P_{-\infty} \in \IU(E)$, we may without loss of generality assume that $P$ has finite propagation $R^\prime$ (since $P_{-\infty}$ is uniformly locally compact by the above Corollary \ref{cor:quasilocal_neg_order_uniformly_locally_compact} and uniformly locally compact operators are uniformly pseudolocal).

We will use the equivalent characterization in Point 4 of the above lemma: let $R, L > 0$ and the corresponding sequence $(L_j)_{j \in \IN}$ be given. We have to show that
\begin{align*}
\{ f P g, g P f \ | \ & f \in B_b(M)\text{ with }\|f\|_\infty \le 1,\\
& g \in B_R(M) \cap C_b^\infty(M)\text{ with }\|\nabla^j g\|_\infty \le L_j,\text{ and}\\
& \supp f \cap \supp g = \emptyset\}
\end{align*}
is uniformly approximable for all $R, L > 0$.

We have
\[f P g = f \chi_{B_{R^\prime}(\supp g)} P g = f \chi_{B_{R^\prime}(\supp g)} [P, g]\]
since the supports of $f$ and $g$ are disjoint.

With Proposition \ref{prop:PsiDOs_filtered_algebra} we conclude that multiplication with $g$ is a uniform pseudodifferential operator of order $0$ (since $g \in C_b^\infty(M)$) and furthermore, that the commutator $[P, g]$ is a pseudodifferential operator of order $-1$. Therefore, by Corollary \ref{prop:PsiDOs_negative_order_uniformly_locally_compact}, we know that the set $\{f \chi_{B_{R^\prime}(\supp g)} [P, g] \ | \ f \in B_R(M)\}$ is uniformly approximable. So we conclude that our operators $f[P,g]$ have the needed uniformity in the functions $f$.

It remains to show that we also have the needed uniformity in the functions $g$. Writing $P = \sum_i P_i$\footnote{Recall that we assumed without loss of generality that there is no $P_{-\infty}$.}, we get $[P,g] = \sum_i [P_i,g]$. Now each $[P_i, g]$ is a uniform pseudodifferential operator of order $-1$, their supports\footnote{Recall that an operator $P$ is \emph{supported in a subset $K$}, if $\supp Pu \subset K$ for all $u$ in the domain of $P$ and if $Pu = 0$ whenever we have $\supp u \cap K = \emptyset$.} depend only on the propagation of $P$ and on the value of $R$ (but not on $i$ nor on the concrete choice of $g$) and their operator norms as maps $L^2(E) \to H^1(E)$ are bounded from above by a constant that only depends on $P$, on $R$ and on the values of all the $L_j$ (but again, neither on $i$ nor on $g$). The last fact follows from a combination of Remark \ref{rem:bound_operator_norm_PDO} together with the estimates on the symbols of the $[P_i,g]$ that we get from the proof that they are uniform pseudodifferential operators of order $-1$. So examining the proof of Proposition \ref{prop:quasilocal_negative_order_uniformly_locally_compact} more closely, we see that these properties suffice to conclude the needed uniformity of $f[P,g]$ in the functions $g$.

The operators $g P f$ may be treated analogously.
\end{proof}

\subsection{Elliptic operators}\label{secio23ed}
In this section we will define the notion of ellipticity for uniform pseudodifferential operators and discuss important consequences of it (elliptic regularity, fundamental elliptic estimates and essential self-adjointness). Most of the results in this section are already known and can be found in the literature (at least in the case of finite propagation operators). We nevertheless include a discussion of them in order for our exposition here to be self-contained.

Let $\pi^\ast E$ and $\pi^\ast F$ denote the pull-back bundles of $E$ and $F$ to the cotangent bundle $\pi\colon T^\ast M \to M$ of the $m$-dimensional manifold $M$.

\begin{defn}[Symbols]\label{defnnkdf893}
Let $p$ be a section of the bundle $\Hom(\pi^\ast E, \pi^\ast F)$ over $T^\ast M$. We call $p$ a \emph{symbol of order $k \in \IZ$}, if the following holds: choosing a uniformly locally finite covering $\{ B_{2 \varepsilon}(x_i) \}$ of $M$ through normal coordinate balls and corresponding subordinate partition of unity $\{ \varphi_i \}$ as in Lemma \ref{lem:nice_coverings_partitions_of_unity}, and choosing synchronous framings of $E$ and $F$ in these balls $B_{2\varepsilon}(x_i)$, we can write $p$ as a uniformly locally finite sum $p = \sum_i p_i$, where $p_i(x,\xi) := p(x,\xi) \varphi(x)$ for $x \in M$ and $\xi \in T^\ast_x M$, and interpret each $p_i$ as a matrix-valued function on $B_{2 \varepsilon}(x_i) \times \IC^m$. Then for all multi-indices $\alpha$ and $\beta$ there must exist a constant $C^{\alpha \beta} < \infty$ such that for all $i$ and all $x, \xi$ we have
\begin{equation}\label{eq:symbol_uniformity}
\|D^\alpha_x D^\beta_\xi p_i(x,\xi) \| \le C^{\alpha \beta}(1 + |\xi|)^{k - |\beta|}.
\end{equation}
We denote the vector space all symbols of order $k \in \IZ$ by $\Symb^k(E,F)$.
\qed
\end{defn}

From Lemma \ref{lem:transition_functions_uniformly_bounded} and Lemma \ref{lem:equiv_characterizations_bounded_geom_bundles} we conclude that the above definition of symbols does neither depend on the chosen uniformly locally finite covering of $M$ through normal coordinate balls, nor on the subordinate partition of unity (as long as the functions $\{\varphi_i\}$ have uniformly bounded derivatives), nor on the synchronous framings of $E$ and $F$.

If all the choices above are fixed, we immediately see from the definition of uniform pseudodifferential operators that $P \in \UPsiDO^k(E,F)$ has a symbol $p \in \Symb^k(E,F)$. Analogously as in the case of compact manifolds,\footnote{see, e.g., \cite[Theorem III.§3.19]{lawson_michelsohn}} we may show that if we make other choices for the coordinate charts, subordinate partition of unity and synchronous framings, the symbol $p$ of $P$ changes by an element of $\Symb^{k-1}(E,F)$. So $P$ has a well-defined principal symbol class $[p] \in \Symb^k(E,F) / \Symb^{k-1}(E,F) =: \Symb^{k-[1]}(E,F)$.

\begin{defn}[Elliptic symbols]\label{defn:elliptic_symbols}
Let $p \in \Symb^k(E,F)$. Recall that $p$ is a section of the bundle $\Hom(\pi^\ast E, \pi^\ast F)$ over $T^\ast M$. We will call $p$ \emph{elliptic}, if there is an $R > 0$ such that $p|_{| \xi | > R}$\footnote{This notation means the following: we restrict $p$ to the bundle $\Hom(\pi^\ast E, \pi^\ast F)$ over the space $\{(x,\xi) \in T^\ast M \ | \ |\xi| > R\} \subset T^\ast M$.} is invertible and this inverse $p^{-1}$ satisfies the Inequality \eqref{eq:symbol_uniformity} for $\alpha, \beta = 0$ and order $-k$ (and of course only for $|\xi| > R$ since only there the inverse is defined). Note that as in the compact case it follows that $p^{-1}$ satisfies the Inequality \eqref{eq:symbol_uniformity} for all multi-indices $\alpha$, $\beta$.
\qed
\end{defn}

The proof of the following lemma is easy.

\begin{lem}\label{lem:ellipticity_independent_of_representative}
If $p \in \Symb^k(E,F)$ is elliptic, then every other representative $p^\prime$ of the class $[p] \in \Symb^{k-[1]}(E,F)$ is also elliptic.
\end{lem}

Due to the above lemma we are now able to define what it means for a pseudodifferential operator to be elliptic:

\begin{defn}[Elliptic $\UPsiDO$s]\label{defn:elliptic_operator}
Let $P \in \UPsiDO^k(E,F)$. We will call $P$ \emph{elliptic}, if its principal symbol $\sigma(P)$ is elliptic.
\qed
\end{defn}

The importance of elliptic operators lies in the fact that they admit an inverse modulo operators of order $-\infty$. We may prove this analogously as in the case of pseudodifferential operators defined over a compact manifold. See also \cite[Theorem 3.3]{kordyukov} where Kordyukov proves the existence of parametrices for his class of pseudodifferential operators (which coincides with our class with the additional requirement that the operators must have finite propagation).

\begin{thm}[Existence of parametrices]
Let $P \in \UPsiDO^k(E,F)$ be elliptic.

Then there exists an operator $Q \in\UPsiDO^{-k}(F, E)$ such that
\[PQ = \id - S_1 \text{ and }QP = \id - S_2,\]
where $S_1 \in \UPsiDO^{-\infty}(F)$ and $S_2 \in \UPsiDO^{-\infty}(E)$.
\end{thm}

Using parametrices, we can prove a lot of the important properties of elliptic operators, e.g., \emph{elliptic regularity} (which is a converse to Proposition \ref{prop:Pu_smooth_if_u_smooth} and a proof of it may be found in, e.g. \cite[Theorem III.§4.5]{lawson_michelsohn}):

\begin{thm}[Elliptic regularity]\label{thm:elliptic_regularity}
Let $P \in \UPsiDO^k(E,F)$ be elliptic and let furthermore $u \in H^s(E)$ for some $s \in \IZ$.

Then, if $Pu$ is smooth on an open subset $U \subset M$, $u$ is already smooth on $U$. Furthermore, for $k > 0$: if $Pu = \lambda u$ on $U$ for some $\lambda \in \IC$, then $u$ is smooth on $U$.
\end{thm}

Later we will also need the following \emph{fundamental elliptic estimate} (the proof from \cite[Theorem III.§5.2(iii)]{lawson_michelsohn} generalizes directly):

\begin{thm}[Fundamental elliptic estimate]\label{thm:elliptic_estimate}
Let $P \in \UPsiDO^k(E,F)$ be elliptic. Then for each $s \in \IZ$ there is a constant $C_s > 0$ such that
\[\|u\|_{H^s(E)} \le C_s\big(\|u\|_{H^{s-k}(E)} + \|Pu\|_{H^{s-k}(F)}\big)\]
for all $u \in H^s(E)$.
\end{thm}

Another implication of ellipticity is that symmetric\footnote{This means that we have $\langle Pu, v\rangle_{L^2(E)} = \langle u, Pv\rangle_{L^2(E)}$ for all $u,v \in C_c^\infty(E)$.}, elliptic uniform pseudodifferential operators of positive order are essentially self-adjoint\footnote{Recall that a symmetric, unbounded operator is called \emph{essentially self-adjoint}, if its closure is a self-adjoint operator.}. We need this since we will have to consider functions of uniform pseudodifferential operators. But first we will show that a symmetric and elliptic operator is also symmetric as an operator on Sobolev spaces.

\begin{lem}\label{lem:symmetric_on_Sobolev}
Let $P \in \UPsiDO^k(E)$ with $k \ge 1$ be symmetric on $L^2(E)$ and elliptic. Then $P$ is also symmetric on the Sobolev spaces $H^{lk}(E)$ for $l \in \IZ$, where we use on $H^{lk}(E)$ the scalar product as described in the proof.
\end{lem}

\begin{proof}
Due to the fundamental elliptic estimate the norm $\|u\|_{H^0} + \|Pu\|_{H^0}$ (note that $H^0(E) = L^2(E)$ by definition) on $H^k(E)$ is equivalent to the usual\footnote{We have of course possible choices here, e.g., the global norm \eqref{eq:sobolev_norm} or the local definition \eqref{eq:sobolev_norm_local}, but they are all equivalent to each other since $M$ and $E$ have bounded geometry.} norm $\|u\|_{H^k}$ on it. Now $\|u\|_{H^0} + \|Pu\|_{H^0}$ is equivalent to $\big( \|u\|^2_{H^0} + \|Pu\|^2_{H^0} \big)^{1/2}$ which is induced by the scalar product
\[\langle u, v \rangle_{H^k, P} := \langle u, v \rangle_{H^0} + \langle Pu, Pv \rangle_{H^0}.\]
Since $P$ is symmetric for the $H^0$-scalar product, we immediately see that it is also symmetric for this particular scalar product $\langle \largecdot, \largecdot \rangle_{H^k,P}$ on $H^k(E)$.

To extend to the Sobolev spaces $H^{lk}(E)$ for $l > 0$ we repeatedly invoke the above arguments, e.g., on $H^{2k}(E)$ we have the equivalent norm $\big( \|u\|^2_{H^k, P} + \|Pu\|^2_{H^k, P} \big)^{1/2}$ (again due to the fundamental elliptic estimate) which is induced by the scalar product $\langle u, v \rangle_{H^k,P} + \langle Pu, Pv \rangle_{H^k,P}$ and now we may use that we already know that $P$ is symmetric with respect to $\langle \largecdot, \largecdot \rangle_{H^k,P}$.

Finally, for $H^{lk}(E)$ for $l < 0$ we use the fact that they are the dual spaces to $H^{-lk}(E)$ where we know that $P$ is symmetric, i.e., we equip $H^{lk}(E)$ for $l < 0$ with the scalar product induced from the duality: $\langle u, v \rangle_{H^{lk},P} := \langle u^\prime, v^\prime\rangle_{H^{-lk},P}$, where $u^\prime, v^\prime \in H^{-lk}(E)$ are the dual vectors to $u,v \in H^{lk}(E)$ (note that the induced norm on $H^{lk}(E)$ is exactly the operator norm if we regard $H^{lk}(E)$ as the dual space of $H^{-lk}(E)$).
\end{proof}

Now we get to the proof that elliptic and symmetric operators are essentially self-adjoint. Note that if we work with differential operators $D$ of first order on open manifolds we do not need ellipticity for this result to hold, but weaker conditions suffice, e.g., that the symbol $\sigma_D$ of $D$ satisfies $\sup_{x \in M, \|\xi\| = 1} \|\sigma_D(x, \xi)\| < \infty$ (by the way, this condition is incorporated in our definition of uniform pseudodifferential operators by the uniformity condition). But if we want essential self-adjointness of higher order operators, we have to assume stronger conditions (see the counterexample \cite{MO_elliptic_essentially_self_adjoint}).

Note that the following proposition is well-known in the case $l=0$, see Shubin \cite{shubin}. But for us it will be of crucial importance in the next Subsection~\ref{sec:functions_of_PDOs} (see the proof of Lemma~\ref{lem:exp(itP)_quasilocal}) that we also have the statement for all the other cases $l \not= 0$. Furthermore, note that in order for the next proposition to make sense at all we have to invoke the above Lemma~\ref{lem:symmetric_on_Sobolev}.

\begin{prop}[Essential self-adjointness]\label{prop:elliptic_PDO_essentially_self-adjoint}
Let $P \in \UPsiDO^k(E)$ with $k \ge 1$ be elliptic and symmetric. Then the unbounded operator $P\colon H^{lk}(E) \to H^{lk}(E)$ is essentially self-adjoint for all $l \in \IZ$, where we equip these Sobolev spaces with the scalar products as described in the proof of the above Lemma \ref{lem:symmetric_on_Sobolev}.
\end{prop}

\begin{proof}
This proof is an adapted version of the proof of this statement for compact manifolds from \cite{MO_elliptic_essentially_self_adjoint}.

We will use the following sufficient condition for essential self-adjointness: if we have a symmetric and densely defined operator $T$ such that $\kernel (T^\ast \pm i) = \{0\}$, then the closure $\overline{T}$ of $T$ is self-adjoint and is the unique self-adjoint extension of $T$.

So let $u \in \kernel (P^\ast \pm i) \subset H^{lk}(E)$, i.e., $P^\ast u = \pm i u$. From elliptic regularity we get that $u$ is smooth and using the fundamental elliptic estimate for $P^\ast$\footnote{Note that $P^\ast$ is elliptic if and only if $P$ is.} we can then conclude $\|u\|_{H^{k+lk}} \le C_{k+lk}\big(\|u\|_{H^{lk}} + \|P^\ast u\|_{H^{lk}}\big) = 2 C_{k+lk} \|u\|_{H^{lk}} < \infty$, i.e., $u \in H^{k+lk}(E)$. Repeating this argument gives us $u \in H^\infty(E)$, i.e., $u$ lies in the domain of $P$ itself and is therefore an eigenvector of it to the eigenvalue $\pm i$. But since $P$ is symmetric we must have $u = 0$. This shows $\kernel (P^\ast \pm i) = \{0\}$ and therefore $P$ is essentially self-adjoint.
\end{proof}

\subsection{Functions of symmetric, elliptic operators}\label{sec:functions_of_PDOs}

Let $P \in \UPsiDO^k(E)$ be a symmetric and elliptic uniform pseudodifferential operator of positive order $k \ge 1$. By Proposition \ref{prop:elliptic_PDO_essentially_self-adjoint} we know that $P\colon L^2(E) \to L^2(E)$ is essentially self-adjoint. So, if $f$ is a Borel function defined on the spectrum of $P$, the operator $f(P)$ is defined by the functional calculus. In this whole section $P$ will denote such an operator, i.e., a symmetric and elliptic one of positive order.

Given such a uniform pseudodifferential operator $P$, we will later show that it defines naturally a class in uniform $K$-homology. For this we will have to consider $\chi(P)$, where $\chi$ is a so-called normalizing function, and we will have to show that $\chi(P)$ is uniformly pseudolocal and $\chi(P)^2 - 1$ is uniformly locally compact. For this we will need the analysis done in this section, i.e., this section is purely technical in nature.

If $f$ is a Schwartz function, we have the formula $f(P) = \frac{1}{\sqrt{2\pi}}\int_\IR \hat{f}(t) e^{itP} dt$, where $\hat{f}$ is the Fourier transform of $f$. In the case that $P = D$ is an elliptic, first-order differential operator and its symbol satisfies $\sup_{x \in M, \|\xi\| = 1} \|\sigma_D(x, \xi)\| < \infty$, the operator $e^{itD}$ has finite propagation (a proof of this may be found in, e.g., \cite[Proposition 10.3.1]{higson_roe}) from which (exploiting the above formula for $f(D)$) we may deduce the needed properties of $\chi(P)$ and $\chi(P)^2 - 1$. But this is no longer the case for a general elliptic pseudodifferential operator $P$ and therefore the analysis that we have to do here in this general case is much more sophisticated.

Note that the restriction to operators of order $k \ge 1$ in this section is no restriction on the fact that symmetric and elliptic uniform pseudodifferential operators define uniform $K$-homology classes. In fact, if $P$ has order $k \le 0$, then we know from Proposition \ref{prop:PDO_order_0_l-uniformly-pseudolocal} that $P$ is uniformly pseudolocal, i.e., there is no need to form the expression $\chi(P)$ in order for $P$ to define a uniform $K$-homology class.

We start with the following crucial technical lemma which is a generalization of the fact that $e^{itD}$ has finite propagation to pseudodifferential operators. Note that we do not have to assume something like $\sup_{x \in M, \|\xi\| = 1} \|\sigma_D(x, \xi)\| < \infty$ that we had to for first-order differential operators, since such an assumption is subsumed in the uniformity condition that we have in the definition of pseudodifferential operators.

\begin{lem}\label{lem:exp(itP)_quasilocal}
Let $P \in \UPsiDO^{k\ge 1}(E)$ be symmetric and elliptic. Then the operator $e^{itP}$ is a quasilocal operator $H^{l}(E) \to H^{l-(k-1)}(E)$ for all $l \in \IR$ and $t \in \IR$.
\end{lem}

\begin{proof}
This proof is a watered down version of the proof of \cite[Theorem 3.1]{mcintosh_morris}.

We will need the following two facts:
\begin{enumerate}
\item $\|e^{itP}\|_{l,l} = 1$ for all $l \in \IR$, where $\|\largecdot\|_{l,l}$ denotes the operator norm of operators $H^{l}(E) \to H^{l}(E)$ and
\item there is a $\kappa > 0$ such that $\|[\eta, P]\|_{s,s-(k-1)} \le \kappa \cdot \sum_{j=1}^N \|\nabla^j \eta\|_\infty$ for all smooth $\eta \in C_b^\infty(M)$, where $N$ does not depend on $\eta$.
\end{enumerate}

The first point above holds since $e^{itP}$ is a unitary operator $H^{lk}(E) \to H^{lk}(E)$ with $l\in \IZ$ by using Proposition \ref{prop:elliptic_PDO_essentially_self-adjoint}, and by interpolation between the different Sobolev exponents we get the needed norm estimate on any $H^l(E)$ with $l \in \IR$, i.e., not only for integer multiples of $k$.

The second point above is due to the facts that by Proposition \ref{prop:PsiDOs_filtered_algebra} the commutator $[\eta, P]$ is a pseudodifferential operator of order $k-1$ (recall that smooth functions with bounded derivatives are operators of order $0$) and due to Remark \ref{rem:bound_operator_norm_PDO} (where we have to recall the formula how to compute the symbol of the composition of two pseudodifferential operators from, e.g., \cite[Theorem III.§3.10]{lawson_michelsohn}).

Let $L \subset M$ and let $u \in H^{l}(E)$ be supported within $L$. Furthermore, we choose an $R > 0$ and a smooth, real-valued function $\eta$ with $\eta \equiv 1$ on $L$, $\eta \equiv 0$ on $M - B_{R+1}(L)$ and the first $N$ derivatives of $\eta$ (for $N$ as above) bounded from above by $C/R$ for a constant $C$ which does not depend on $u,L,R,\eta$. Concretely, one can construct $\eta$ by mollifying the function $\eta_0(x) := \max\big\{0,1-d(x,B_{1/2}(L))/R\big\}$ with a uniform collection of local mollifiers that are supported in balls of radius less than $1/2$ and with midpoints in the region $B_{R+1/2}(L) - B_{1/2}(L)$; see the proof of Lemma \ref{lem:norm_completion_C_b_infty} for more details and combine it with the following fact: if we denote a local mollifier by $\psi$, then we have for the Lipschitz constant the estimate
\[\operatorname{Lip}\!\big(D^\alpha(\eta_0 \ast \psi)\big) = \operatorname{Lip}( \eta_0 \ast D^\alpha \psi) \le \operatorname{Lip}(\eta_0) \cdot \|D^\alpha \psi\|_{L^1} = 1/R \cdot \|D^\alpha \psi\|_{L^1}\]
from which the needed property on the derivatives of $\eta$ follows. Note that we only need to do this proof for large $R$, i.e., we do the arguments here only for $R$ bigger than, say, the injectivity radius of $M$. This means that the derivatives of the local mollifiers that we use do not explode since there is now a lower bound on the size of the coordinate charts in which we apply our mollifiers.

For all $v \in H^{l-(k-1)}(E)$ that are supported in $M - B_{R+1}(L)$ we have
\begin{align*}
\langle e^{itP} u, v\rangle_{H^{l-(k-1)}} & = \langle e^{itP} \eta u, v\rangle_{H^{l-(k-1)}} - \langle e^{itP} u, \eta v\rangle_{H^{l-(k-1)}}\\
& = \langle [e^{itP},\eta] u, v\rangle_{H^{l-(k-1)}},
\end{align*}
i.e., $|\langle e^{itP} u, v\rangle_{H^{l-(k-1)}}| \le \|[e^{itP},\eta]\|_{l,l-(k-1)} \cdot \|u\|_{H^{l}} \cdot \|v\|_{H^{l-(k-1)}}$ and it remains to give an estimate for $\|[e^{itP},\eta]\|_{l,l-(k-1)}$: we have (the expressions are to be considered point-wise, i.e., after application to a fixed vector $v$)
\begin{align*}
[e^{itP}, \eta] & = \int_0^1 \tfrac{d}{dx} \big( e^{ixtP} \eta e^{i(1-x)tP} \big) dx\\
& = -it \int_0^1 e^{ixtP} [\eta,P] e^{i(1-x)tP} dx
\end{align*}
which gives by factorizing the integrand as
\[H^{l}(E) \stackrel{e^{i(1-x)tP}}\longrightarrow H^{l}(E) \stackrel{[\eta, P]}\longrightarrow H^{l-(k-1)}(E) \stackrel{e^{ixtP}}\longrightarrow H^{l-(k-1)}(E)\]
the estimate
\[\|[e^{itP},\eta]\|_{l,l-(k-1)} \le |t| \int_0^1 \|[\eta, P]\|_{l,l-(k-1)} dx \le |t| \cdot \kappa \cdot \sum_{j=1}^N \|\nabla^j \eta\|_\infty.\]
Since $\|\nabla^j \eta\|_\infty < C/R$ for all $1 \le j \le N$, we have shown
\begin{equation}
\label{eq:dominating_function_exp(itP)}
|\langle e^{itP} u, v\rangle_{H^{l-(k-1)}}| < \frac{|t| \kappa N C}{R} \cdot \|u\|_{H^{l}} \cdot \|v\|_{H^{l-(k-1)}}
\end{equation}
for all $u$ supported in $L$ and all $v$ in $M-B_{R+1}(L)$. Because $R > 0$ and $l \in \IR$, $t \in\IR$ were arbitrary, the claim that $e^{itP}$ is a quasilocal operator $H^{l}(E) \to H^{l-(k-1)}(E)$ for all $l \in \IR$ and $t \in \IR$ follows.
\end{proof}

\begin{cor}[cf.~{\cite[Lemma 1.1 in Chapter XII.§1]{taylor_pseudodifferential_operators}}]
\label{cor:lth_derivative_integrable_defines_quasilocal_operator}
Let $q(t)$ be a function on $\IR$ such that for an $n \in \IN_0$ the functions $q(t)|t|$, $q^\prime(t)|t|$,  $\ldots$, $q^{(n)}(t)|t|$ are integrable, i.e., belong to $L^1(\IR)$.

Then the operator defined by $\int_\IR q(t) e^{itP} dt$ is for all values $l \in \IR$ a quasilocal operator $H^{l-nk+k-1}(E) \to H^{l}(E)$, i.e., is of order $-nk + k - 1$.
\end{cor}

\begin{proof}
Let $Q \in \UPsiDO^{-k}(E)$ be a parametrix for $P$, i.e., $PQ = \id - S_1$ and $QP = \id - S_2$, where $S_1, S_2 \in \UPsiDO^{-\infty}(E)$. Integration by parts $n$ times yields:
\begin{equation}
\label{eq:formula_integration_by_parts_quasilocal}
(i Q)^n \int_\IR q^{(n)}(t) e^{itP} dt = (i Q)^n (-i P)^n \int_\IR q(t) e^{itP} dt = (\id - S_2)^n \int_\IR q(t)e^{itP} dt.
\end{equation}
Since $q(t)|t|$ and $q^{(n)}(t)|t|$ are integrable and due to the Estimate \eqref{eq:dominating_function_exp(itP)}, we conclude with Lemma \ref{lem:exp(itP)_quasilocal} that both integrals $\int_\IR q(t)e^{itP} dt$ and $\int_\IR q^{(n)}(t)e^{itP} dt$ define quasilocal operators of order $k-1$ on $H^{l}(E)$. Note that for $\int_\IR q(t)e^{itP} dt$ this is just a first result which we will need now in order to show that the order of this operator is in fact lower.

Now $(\id - S_2)^n = \id + \sum_{j=1}^n \binom{n}{j}(-S_2)^j$ and the sum is a quasilocal smoothing operator because $S_2$ is one. Since the composition of quasilocal operators is again a quasilocal operator (see \cite[Proposition 5.2]{roe_index_1}), we conclude that the second summand $R$ of
\begin{equation}
\label{eq:formula_integration_by_parts_quasilocal_2}
(\id - S_2)^n \int_\IR q(t)e^{itP} dt = \int_\IR q(t)e^{itP} dt + \underbrace{\sum_{j=1}^n \binom{n}{j}(-S_2)^j \int_\IR q(t)e^{itP} dt}_{=: R}
\end{equation}
is also a quasilocal smoothing operator. Now Equations \eqref{eq:formula_integration_by_parts_quasilocal} and \eqref{eq:formula_integration_by_parts_quasilocal_2} together yield
\[\int_\IR q(t)e^{itP} dt = (i Q)^n \int_\IR q^{(n)}(t) e^{itP} dt - R,\]
from which the claim follows.
\end{proof}

Recall that if $f$ is a Schwartz function, then the operator $f(P)$ is given by
\begin{equation}\label{eq:schwartz_function_of_PDO}
f(P) = \frac{1}{\sqrt{2\pi}}\int_\IR \hat{f}(t) e^{itP} dt,
\end{equation}
where $\hat{f}$ is the Fourier transform of $f$. Since $\hat{f}$ is also a Schwartz function, it satisfies the assumption in Corollary \ref{cor:lth_derivative_integrable_defines_quasilocal_operator} for all $n \in \IN_0$, i.e., $f(P)$ is a quasilocal smoothing operator. Applying this argument to the adjoint operator $f(P)^\ast = \overline{f}(P)$, we get with Lemma \ref{lem:PDO_-infinity_equal_quasilocal_smoothing} our next corollary:

\begin{cor}\label{cor:schwartz_function_of_PDO_quasilocal_smoothing}
If $f$ is a Schwartz function, then $f(P) \in \UPsiDO^{-\infty}(E)$.
\end{cor}

Recall from \cite[Lemma 4.2]{spakula_uniform_k_homology} that the uniformly pseudolocal operators form a $C^\ast$-algebra and that the uniformly locally compact operators form a closed, two-sided $^\ast$-ideal in there. Since Schwartz functions are dense in $C_0(\IR)$ and quasilocal smoothing operators are uniformly locally compact (Corollary \ref{cor:quasilocal_neg_order_uniformly_locally_compact}), we get with the above corollary that $g(P)$ is uniformly locally compact if $g \in C_0(\IR)$.

\begin{cor}\label{cor:g(P)_uniformly_locally_compact_g_vanishing_at_infinity}
Let $g \in C_0(\IR)$. Then $g(P)$ is uniformly locally compact.
\end{cor}

Now we turn our attention to functions which are more general than Schwartz functions. To be concrete, we consider functions of the following type:

\begin{defn}[Symbols on $\IR$]\label{defn:symbols_on_R}
For arbitrary $m \in \IZ$ we define
\[\mathcal{S}^m(\IR) := \{f \in C^\infty(\IR) \ | \ |f^{(n)}(x)| < C_n(1 + |x|)^{m-n} \text{ for all } n \in \IN_0\}.\]

Note that we have $\mathcal{S}(\IR) = \bigcap_m \mathcal{S}^m(\IR)$, where $\mathcal{S}(\IR)$ denotes the Schwartz space.
\qed
\end{defn}

Let us state now the generalization of \cite[Theorem 5.5]{roe_index_1} from operators of Dirac type to uniform pseudodifferential operators:

\begin{prop}[cf.~{\cite[Theorem 5.5]{roe_index_1}}]\label{prop:f(P)_quasilocal_of_symbol_order}
Let $f \in \mathcal{S}^m(\IR)$ with $m \le 0$. Then for all $l \in \IR$ the operator $f(P)$ is a quasilocal operator of order $mk+k-1$, i.e., is an operator $f(P) \colon H^{l}(E) \to H^{l-(mk+k-1)}(E)$.
\end{prop}

\begin{proof}
The proof is analogous to Roe's proof of \cite[Theorem 5.5]{roe_index_1}, but more technical. First let us note that $f(P)$ is a bounded operator of order $mk$. To see this note that $(1+|x|)^{-m} \cdot f(x)$ is a bounded function and therefore $(1+|P|)^{-m} \circ f(P)$ is a bounded operator of order $0$. Combining the fact that $(1+|P|)^{-m}$ is an operator of order $-mk$ together with the fundamental elliptic estimate from Theorem \ref{thm:elliptic_estimate} we get the result that $f(P)$ is bounded of order $mk$.

Now we want not only boundedness of $f(P)$ but also that it is quasilocal. Roe uses in his proof of \cite[Theorem 5.5]{roe_index_1} the fact that $e^{itD}$ has propagation $|t|$ for $D$ a Dirac operator. But for pseudodifferential operators the best that we have is our Lemma \ref{lem:exp(itP)_quasilocal} and that's the reason why we loose $k-1$ orders for the statement that $f(P)$ is quasilocal. The rest of our proof is analogous to Roe's proof.
\end{proof}

At last, let us turn our attention to a result regarding differences $\psi(P) - \psi(P^\prime)$ of operators defined via functional calculus. We will need the following proposition in the proof of the proposition where we show that symmetric, elliptic uniform pseudodifferential operators with the same symbol define the same uniform $K$-homology class.

\begin{prop}[{\cite[Proposition 10.3.7]{higson_roe}}\footnote{The cited proposition requires additionally a common invariant domain for $P$ and $P^\prime$. In our case here this domain is given by, e.g., $H^\infty(E)$.}]\label{prop:norm_estimate_difference_func_calc}
Let $\psi$ be a bounded Borel function whose distributional Fourier transform $\hat{\psi}$ is such that the product $s\hat{\psi}(s)$ is in $L^1(\IR)$.

If $P$ and $P^\prime$ are symmetric and elliptic uniform pseudodifferential operators of positive order $k \ge 1$ such that their difference $P - P^\prime$ has order $q$, then we have for all $l \in \IR$
\[\|\psi(P) - \psi(P^\prime)\|_{l,l-q} \le C_\psi \cdot \|P - P^\prime\|_{l,l-q},\]
where the constant $C_\psi = \frac{1}{2\pi} \int |s \hat{\psi}(s)| ds$ does not depend on the operators.
\end{prop}

\begin{proof}
We first assume that $\hat{\psi}$ is compactly supported and that $s\hat{\psi}(s)$ is a smooth function. Then we use the result \cite[Proposition 10.3.5]{higson_roe}\footnote{Though stated there only for differential operators, its proof also works word-for-word for pseudodifferential ones.}, which is a generalization of Equation \ref{eq:schwartz_function_of_PDO} to more general functions than Schwartz functions, and get
\[\Big\langle \big( \psi(P) - \psi(P^\prime) \big) u, v \Big\rangle_{H^{l-q}} = \frac{1}{2 \pi} \int \left\langle \big( e^{isP} - e^{isP^\prime} \big) u, v \right\rangle_{H^{l-q}} \cdot \hat{\psi}(s) ds,\]
for all $u,v \in C_c^\infty(E)$. From the Fundamental Theorem of Calculus we get
\[\left\langle \big( e^{isP} - e^{isP^\prime} \big) u, v \right\rangle_{H^{l-q}} = i \cdot \int_0^s \left\langle \big( e^{itP} (P - P^\prime) e^{i(s-t)P^\prime} \big) u, v \right\rangle_{H^{l-q}} dt\]
and therefore
\[\left| \left\langle \big( e^{isP} - e^{isP^\prime} \big) u, v \right\rangle_{H^{l-q}}\right| \le s \cdot \|P - P^\prime\|_{l,l-q} \cdot \|u\|_{l} \cdot \|v\|_{l-q}.\]
Putting it all together, we get
\[\left| \Big\langle \big( \psi(P) - \psi(P^\prime) \big) u, v \Big\rangle_{H^{l-q}} \right| \le C_\psi \cdot \|P - P^\prime\|_{l,l-q} \cdot \|u\|_{l} \cdot \|v\|_{l-q}.\]

Now the general claim follows from an approximation argument analogous to the one at the end of the proof of \cite[Proposition 10.3.5]{higson_roe}.
\end{proof}

\section{Uniform \texorpdfstring{$K$}{K}-homology}

Since we are considering uniform pseudodifferential operators, we need a $K$-homology theory that incorporates into its definition this uniformity. Such a theory was introduced by \Spakula and the goal of this section if to revisit it and to prove certain properties (existence of the Kasparov product and deducing from it homotopy invariance of uniform $K$-homology) that we will crucially need later and which were not proved by \v{S}pakula. Furthermore, we will use in Section \ref{sec:rough_BC} homotopy invariace to deduce useful facts about the rough Baum--Connes assembly map.

\subsection{Definition and basic properties of uniform \texorpdfstring{$K$}{K}-homology}

Let us first recall briefly the notion of multigraded Hilbert spaces.

A \emph{graded Hilbert space} is a Hilbert space $H$ with a decomposition $H = H^+ \oplus H^-$ into closed, orthogonal subspaces. This is equivalent to the existence of a \emph{grading operator} $\epsilon$ such that its $\pm 1$-eigenspaces are exactly $H^\pm$ and such that $\epsilon$ is a selfadjoint unitary.

If $H$ is a graded space, then its \emph{opposite} is the graded space $H^\op$ whose underlying vector space is $H$, but with the reversed grading, i.e., $(H^\op)^+ = H^-$ and $(H^\op)^- = H^+$. This is equivalent to $\epsilon_{H^\op} = -\epsilon_H$.

An operator on a graded space $H$ is called \emph{even} if it maps $H^\pm$ again to $H^\pm$, and it is called \emph{odd} if it maps $H^\pm$ to $H^\mp$. Equivalently, an operator is even if it commutes with the grading operator $\epsilon$ of $H$, and it is odd if it anti-commutes with it.

\begin{defn}[Multigraded Hilbert spaces]
Let $p \in \IN_0$. A \emph{$p$-multigraded Hilbert space} is a graded Hilbert space which is equipped with $p$ odd unitary operators $\epsilon_1, \ldots, \epsilon_p$ such that $\epsilon_i \epsilon_j + \epsilon_j \epsilon_i = 0$ for $i \not= j$, and $\epsilon_j^2 = -1$ for all $j$.
\qed
\end{defn}

Note that a $0$-multigraded Hilbert space is just a graded Hilbert space. We make the convention that a $(-1)$-multigraded Hilbert space is an ungraded one.

\begin{defn}[Multigraded operators]
Let $H$ be a $p$-multigraded Hilbert space. Then an operator on $H$ will be called \emph{multigraded}, if it commutes with the multigrading operators $\epsilon_1, \ldots, \epsilon_p$ of $H$.
\qed
\end{defn}

Let us now recall the usual definition of multigraded Fredholm modules, where $X$ is a locally compact, separable metric space:

\begin{defn}[Multigraded Fredholm modules]
Let $p \in \IZ_{\ge -1}$.

A triple $(H,\rho,T)$ consisting of
\begin{itemize}
\item a separable $p$-multigraded Hilbert space $H$,
\item a representation $\rho\colon C_0(X) \to \IB(H)$ by even, multigraded operators, and
\item an odd multigraded operator $T \in \IB(H)$ such that
\begin{itemize}
\item the operators $T^2 - 1$ and $T - T^\ast$ are locally compact and
\item the operator $T$ itself is pseudolocal
\end{itemize}
\end{itemize}
is called a \emph{$p$-multigraded Fredholm module over $X$}.

Here an operator $S$ is called \emph{locally compact}, if for all $f \in C_0(X)$ the operators $\rho(f) S$ and $S \rho(f)$ are compact, and $S$ is called \emph{pseudolocal}, if for all $f \in C_0(X)$ the operator $[S, \rho(f)]$ is compact.
\qed
\end{defn}

Let us define
\begin{equation*}
\LLip_R(X) := \{ f \in C_c(X) \ | \ f \text{ is }L\text{-Lipschitz}, \diam(\supp f) \le R \text{ and } \|f\|_\infty \le 1\}.
\end{equation*}

\begin{defn}[{\cite[Definition 2.3]{spakula_uniform_k_homology}}]\label{defn:uniform_operators}
Let $T \in \IB(H)$ be an operator on a Hilbert space $H$ and $\rho\colon C_0(X) \to \IB(H)$ a representation.

We say that $T$ is \emph{uniformly locally compact}, if for every $R, L > 0$ the collection
\[\{\rho(f)T, T\rho(f) \ | \ f \in \LLip_R(X)\}\]
is uniformly approximable (Definition \ref{defn:uniformly_approximable_collection}).

We say that $T$ is \emph{uniformly pseudolocal}, if for every $R, L > 0$ the collection
\[\{[T, \rho(f)] \ | \ f \in \LLip_R(X)\}\]
is uniformly approximable.
\qed
\end{defn}

Note that by an approximation argument we get that the above defined collections are still uniformly approximable if we enlargen the definition of $\LLip_R(X)$ from $f \in C_c(X)$ to $f \in C_0(X)$.

The following lemma states that on proper spaces we may drop the $L$-dependence for uniformly locally compact operators.

\begin{lem}[{\cite[Remark 2.5]{spakula_uniform_k_homology}}]\label{lem:l_uniformly_loc_compact_without_l}
Let $X$ be a proper space. If $T$ is uniformly locally compact, then for every $R > 0$ the collection
\[\{\rho(f) T, T \rho(f) \ | \ f \in C_c(X), \diam(\supp f) \le R \text{ and } \|f\|_\infty \le 1\}\]
is also uniformly approximable (i.e., we can drop the $L$-dependence).
\end{lem}

Note that an analogous lemma for uniformly pseudolocal operators does not hold. We may see this via the following example: if we have an operator $D$ of Dirac type on a manifold $M$ and if $g$ is a smooth function on $M$, then we have the equation $([D,g]u)(x) = \sigma_D(x, dg) u(x)$, where $u$ is a section into the Dirac bundle $S$ on which $D$ acts, $\sigma_D(x, \xi)$ is the symbol of $D$ regarded as an endomorphism of $S_x$ and $\xi \in T^\ast_x M$. So we see that the norm of $[D,g]$ does depend on the first derivative of the function $g$.

\begin{defn}[Uniform Fredholm modules, cf. {\cite[Definition 2.6]{spakula_uniform_k_homology}}]\label{defn:uniform_fredholm_modules}
A Fredholm module $(H, \rho, T)$ is called \emph{uniform}, if $T$ is uniformly pseudolocal and the operators $T^2-1$ and $T - T^\ast$ are uniformly locally compact.
\qed
\end{defn}

\begin{example}[{\cite[Theorem 3.1]{spakula_uniform_k_homology}}]
\Spakula showed that the usual Fredholm module arising from a generalized Dirac operator is uniform if we assume bounded geometry: if $D$ is a generalized Dirac operator acting on Dirac bundle $S$ of bounded geometry over a manifold $M$ of bounded geometry, then the triple $(L^2(S), \rho, \chi(D))$, where $\rho$ is the representation of $C_0(M)$ on $L^2(S)$ by multiplication operators and $\chi$ is a normalizing function (see Definition \ref{defn:normalizing_function}), is a uniform Fredholm module.

In Section \ref{sec:homology_classes_of_PDOs} we will generalize this statement to symmetric and elliptic uniform pseudodifferential operators.
\qed
\end{example}

For a totally bounded metric space uniform Fredholm modules are the same as usual Fredholm modules. Since \v{S}pakula does not give a proof of it, we will do it now:

\begin{prop}\label{prop:compact_space_every_module_uniform}
Let $X$ be a totally bounded metric space. Then every Fredholm module over $X$ is uniform.
\end{prop}

\begin{proof}
Let $(H, \rho, T)$ be a Fredholm module.

First we will show that $T$ is uniformly pseudolocal. We will use the fact that the set $\LLip_R(X) \subset C(X)$ is relatively compact (i.e., its closure is compact) by the Theorem of Arzel\`{a}--Ascoli.\footnote{Since Lipschitz functions are uniformly continuous they have a unique extension to the completion $\overline{X}$ of $X$. Since $\overline{X}$ is compact, Arzel\`{a}--Ascoli applies.} Assume that $T$ is not uniformly pseudolocal. Then there would be $R, L > 0$ and $\varepsilon > 0$, so that for all $N > 0$ we would have an $f_N \in \LLip_R(X)$ such that for all rank-$N$ operators $k$ we have $\|[T, \rho(f_N)] - k\| \ge \varepsilon$. Since $\LLip_R(X)$ is relatively compact, the sequence $f_N$ has an accumulation point $f_\infty \in \LLip_R(X)$. Then we have $\|[T, \rho(f_\infty)] - k\| \ge \varepsilon / 2$ for all finite rank operators $k$, which is a contradiction.

The proofs that $T^2 - 1$ and $T - T^\ast$ are uniformly locally compact are analogous.
\end{proof}

A collection $(H, \rho, T_t)$ of uniform Fredholm modules is called an \emph{operator homotopy} if $t \mapsto T_t \in \IB(H)$ is norm continuous. As in the non-uniform case, we have an analogous lemma about compact perturbations:

\begin{lem}[Compact perturbations, {\cite[Lemma 2.16]{spakula_uniform_k_homology}}]
\label{lem:compact_perturbations}
Let $(H, \rho, T)$ be a uniform Fredholm module and $K \in \IB(H)$ a uniformly locally compact operator.

Then $(H, \rho, T)$ and $(H, \rho, T + K)$ are operator homotopic.
\end{lem}

\begin{defn}[Uniform $K$-homology, {\cite[Definition 2.13]{spakula_uniform_k_homology}}]
We define the \emph{uniform $K$-homology group $K_{p}^u(X)$} of a locally compact and separable metric space $X$ to be the abelian group generated by unitary equivalence classes of $p$-multigraded uniform Fredholm modules with the relations:
\begin{itemize}
\item if $x$ and $y$ are operator homotopic, then $[x] = [y]$, and
\item $[x] + [y] = [x \oplus y]$,
\end{itemize}
where $x$ and $y$ are $p$-multigraded uniform Fredholm modules.
\qed
\end{defn}

All the basic properties of usual $K$-homology do also hold for uniform $K$-homology (e.g., that degenerate uniform Fredholm modules represent the zero class, that we have formal $2$-periodicity $K_{p}^u(X) \cong K_{p+2}^u(X)$ for all $p \ge -1$, etc.).

For discussing functoriality of uniform $K$-homology we need the following definition:

\begin{defn}[Uniformly cobounded maps, {\cite[Definition 2.15]{spakula_uniform_k_homology}}]
Let us call a map $g\colon X \to Y$ with the property
\[\sup_{y \in Y} \diam (g^{-1}(B_r(y))) < \infty \text{ for all }r > 0\]
\emph{uniformly cobounded}\footnote{Block and Weinberger call this property \emph{effectively proper} in \cite{block_weinberger_1}. The author called it \emph{uniformly proper} in his thesis \cite{engel_phd}.}.

Note that if $X$ is proper, then every uniformly cobounded map is proper (i.e., preimages of compact subsets are compact).
\qed
\end{defn}

The following lemma about functoriality of uniform $K$-homology was proved by \Spakula (see the paragraph directly after \cite[Definition 2.15]{spakula_uniform_k_homology}).

\begin{lem}
Uniform $K$-homology is functorial with respect to uniformly cobounded, proper Lipschitz maps, i.e., if $g\colon X \to Y$ is uniformly cobounded, proper and Lipschitz, then it induces maps $g_\ast\colon K_\ast^u(X) \to K_\ast^u(Y)$ on uniform $K$-homology via
\[g_\ast [(H, \rho, T)] := [(H, \rho \circ g^\ast, T)],\]
where $g^\ast\colon C_0(Y) \to C_0(X)$, $f \mapsto f \circ g$ is the by $g$ induced map on functions.
\end{lem}

Recall that $K$-homology may be normalized in various ways, i.e., we may assume that the Fredholm modules have a certain form or a certain property and that this holds also for all homotopies.

Combining Lemmas 4.5 and 4.6 and Proposition 4.9 from \cite{spakula_uniform_k_homology}, we get the following:

\begin{lem}\label{lem:normalization_involutive}
We can normalize uniform $K$-homology $K_\ast^u(X)$ to involutive modules.
\end{lem}

The proof of the following Lemma \ref{lem:normalization_non-degenerate} in the non-uniform case may be found in, e.g., \cite[Lemma 8.3.8]{higson_roe}. The proof in the uniform case is analogous and the arguments similar to the ones in the proofs of \cite[Lemmas 4.5 \& 4.6]{spakula_uniform_k_homology}.

\begin{lem}\label{lem:normalization_non-degenerate}
Uniform $K$-homology $K_\ast^u(X)$ may be normalized to \emph{non-degenerate} Fredholm modules, i.e., such that all occuring representations $\rho$ are non-degenerate\footnote{This means that $\rho(C_0(X)) H$ is dense in $H$.}.
\end{lem}

Note that in general we can not normalize uniform $K$-homology to be simultaneously involutive and non-degenerate, just as usual $K$-homology.

Later we will also have to normalize Fredholm modules to finite propagation. But this is not always possible if the underlying metric space $X$ is badly behaved. Therefore we get now to the definition of bounded geometry for metric spaces.

\begin{defn}[Coarsely bounded geometry]
\label{defn:coarsely_bounded_geometry}
Let $X$ be a metric space. We call a subset $\Gamma \subset X$ a \emph{quasi-lattice} if
\begin{itemize}
\item there is a $c > 0$ such that $B_c(\Gamma) = X$ (i.e., $\Gamma$ is \emph{coarsely dense}) and
\item for all $r > 0$ there is a $K_r > 0$ such that $\card(\Gamma \cap B_r(y)) \le K_r$ for all $y \in X$.
\end{itemize}
A metric space is said to have \emph{coarsely bounded geometry}\footnote{Note that most authors call this property just ``bounded geometry''. But since later we will also have the notion of locally bounded geometry, we use for this one the term ``coarsely'' to distinguish them.} if it admits a quasi-lattice.
\qed
\end{defn}

Note that if we have a quasi-lattice $\Gamma \subset X$, then there also exists a uniformly discrete quasi-lattice $\Gamma^\prime \subset X$. The proof of this is an easy application of the Lemma of Zorn: given an arbitrary $\delta > 0$ we look at the family $\mathcal{A}$ of all subsets $A \subset \Gamma$ with $d(x,y) > \delta$ for all $x,y \in A$. These subsets are partially ordered under inclusion of sets and every totally ordered chain $A_1 \subset A_2 \subset \ldots \subset \Gamma$ has an upper bound given by the union $\bigcup_i A_i \in \mathcal{A}$. So the Lemma of Zorn provides us with a maximal element $\Gamma^\prime \in \mathcal{A}$. That $\Gamma^\prime$ is a quasi-lattice follows from its maximality.

\begin{examples}\label{ex:coarsely_bounded_geometry}
Every Riemannian manifold $M$ of bounded geometry\footnote{That is to say, the injectivity radius of $M$ is uniformly positive and the curvature tensor and all its derivatives are bounded in sup-norm.} is a metric space of coarsely bounded geometry: any maximal set $\Gamma \subset M$ of points which are at least a fixed distance apart (i.e., there is an $\varepsilon > 0$ such that $d(x, y) \ge \varepsilon$ for all $x \not= y \in \Gamma$) will do the job. We can get such a maximal set by invoking Zorn's lemma. Note that a manifold of bounded geometry will also have locally bounded geometry, so no confusion can arise by not distinguishing between ``coarsely'' and ``locally'' bounded geometry in the terminology for manifolds.

If $(X,d)$ is an arbitrary metric space that is bounded, i.e., $d(x,x^\prime) < D$ for all $x, x^\prime \in X$ and some $D$, then \emph{any} finite subset of $X$ will constitute a quasi-lattice.

Let $K$ be a simplicial complex of bounded geometry\footnote{That is, the number of simplices in the link of each vertex is uniformly bounded.}. Equipping $K$ with the metric derived from barycentric coordinates the subset of all vertices of the complex $K$ becomes a quasi-lattice in $K$.
\qed
\end{examples}

If $X$ has coarsely bounded geometry it will be crucial for us that we can normalize uniform $K$-homology to uniform finite propagation, i.e., such there is an $R > 0$ depending only on $X$ such that every uniform Fredholm module has propagation at most $R$\footnote{This means $\rho(f) T \rho(g) = 0$ if $d(\supp f, \supp g) > R$.}. This was proved by \Spakula in \cite[Proposition 7.4]{spakula_uniform_k_homology}. Note that it is in general not possible to make this common propagation $R$ arbitrarily small. Furthermore, we can combine the normalization to finite propagation with the other normalizations.

\begin{prop}[{\cite[Section 7]{spakula_uniform_k_homology}}]\label{prop:normalization_finite_prop_speed}
If $X$ has coarsely bounded geometry, then there is an $R > 0$ depending only on $X$ such that uniform $K$-homology may be normalized to uniform Fredholm modules that have propagation at most $R$.

Furthermore, we can additionally normalize them to either involutive modules or to non-degenerate ones.
\end{prop}

Having discussed the normalization to finite propagation modules, we can now compute an easy but important example:

\begin{lem}\label{lem:uniform_k_hom_discrete_space}
Let $Y$ be a uniformly discrete, proper metric space of coarsely bounded geometry. Then $K_0^u(Y)$ is isomorphic to the group $\ell_\IZ^\infty(Y)$ of all bounded, integer-valued sequences indexed by $Y$, and $K_1^u(Y) = 0$.
\end{lem}

\begin{proof}
We use Proposition \ref{prop:normalization_finite_prop_speed} to normalize uniform $K$-homology to operators of finite propagation, i.e., there is an $R > 0$ such that every uniform Fredholm module over $Y$ may be represented by a module $(H, \rho, T)$ where $T$ has propagation no more than $R$ and all homotopies may be also represented by homotopies where the operators have propagation at most $R$.

Going into the proof of Proposition \ref{prop:normalization_finite_prop_speed}, we see that in our case of a uniformly discrete metric space $Y$ we may choose $R$ less than the least distance between two points of $Y$, i.e., $0 < R < \inf_{x \not= y \in Y} d(x,y)$. So given a module $(H, \rho, T)$ where $T$ has propagation at most $R$, the operator $T$ decomposes as a direct sum $T = \bigoplus_{y \in Y} T_y$ with $T_y \colon H_y \to H_y$. The Hilbert space $H_y$ is defined as $H_y := \rho(\chi_y) H$, where $\chi_y$ is the characteristic function of the single point $y \in Y$. Note that $\chi_y$ is a continuous function since the space $Y$ is discrete. Hence $(H, \rho, T) = \bigoplus (H_y, \rho_y, T_y)$ with $\rho_y\colon C_0(Y) \to \IB(H_y)$, $f \mapsto \rho(\chi_y) \rho(f) \rho(\chi_y)$. Now each $(H_y, \rho_y, T_y)$ is a Fredholm module over the point $y$ and so we get a map
\[K_\ast^u(Y) \to \prod_{y \in Y} K_\ast^u(y).\]
Note that we need that the homotopies also all have propagation at most $R$ so that the above defined decomposition of a uniform Fredholm module descends to the level of uniform $K$-homology.

Since a point $y$ is for itself a compact space, we have $K_\ast^u(y) = K_\ast(y)$, and the latter group is isomorphic to $\IZ$ for $\ast = 0$ and it is $0$ for $\ast = 1$. Since the above map $K_\ast^u(Y) \to \prod_{y \in Y} K_\ast^u(y)$ is injective, we immediately conclude $K_1^u(Y) = 0$.

So it remains to show that the image of this map in the case $\ast = 0$ consists of the \emph{bounded} integer-valued sequences indexed by $Y$. But this follows from the uniformity condition in the definition of uniform $K$-homology: the isomorphism $K_0(y) \cong \IZ$ is given by assigning a module $(H_y, \rho_y, T_y)$ the Fredholm index of $T$ (note that $T_y$ is a Fredholm operator since $(H_y, \rho_y, T_y)$ is a module over a single point). Now since $(H, \rho, T) = \bigoplus (H_y, \rho_y, T_y)$ is a \emph{uniform} Fredholm module, we may conclude that the Fredholm indices of the single operators $T_y$ are bounded with respect to $y$.
\end{proof}

\subsection{Differences to \v{S}pakula's version}
We will discuss now the differences between our version of uniform $K$-homology and \v{S}pakula's version from his Ph.D.\ thesis \cite{spakula_thesis}, resp., his publication \cite{spakula_uniform_k_homology}.

Firstly, our definition of uniform $K$-homology is based on multigraded Fredholm modules and we therefore have groups $K_p^\ast(X)$ for all $p \ge -1$, but \Spakula only defined $K_0^u$ and $K_1^u$. This is not a real restriction since uniform $K$-homology has, analogously as usual $K$-homology, a formal $2$-periodicity. We mention this since if the reader wants to look up the original reference \cite{spakula_thesis} and \cite{spakula_uniform_k_homology}, he has to keep in mind that we work with multigraded modules, but \Spakula does not.

Secondly, \Spakula gives the definition of uniform $K$-homology only for proper\footnote{That means that all closed balls are compact.} metric spaces since certain results of him (Sections 8-9 in \cite{spakula_uniform_k_homology}) only work for such spaces. These results are all connected to the rough assembly map $\mu_u\colon K_\ast^u(X) \to K_\ast(C_u^\ast(Y))$, where $Y \subset X$ is a uniformly discrete quasi-lattice, and this is nor surprising: the (uniform) Roe algebra only has on proper spaces nice properties (like its $K$-theory being a coarse invariant) and therefore we expect that results of uniform $K$-homology that connect to the uniform Roe algebra also should need the properness assumption. But we can see by looking into the proofs of \Spakula in all the other sections of \cite{spakula_uniform_k_homology} that all results except the ones in Sections 8-9 also hold for locally compact, separable metric spaces (without assumptions on completeness or properness). Note that this is a very crucial fact for us that uniform $K$-homology does also make sense for non-proper spaces since in the proof of \Poincare duality we will have to consider the uniform $K$-homology of open balls in $\IR^n$.

Thirdly, \v{S}pakula uses the notion ``$L$-continuous'' instead of ``$L$-Lipschitz'' for the definition of $\LLip_R(X)$ (which he also denotes by $C_{R,L}(X)$, i.e., we have also changed the notation), so that he gets slightly differently defined uniform Fredholm modules. But the author was not able to deduce Proposition \ref{prop:compact_space_every_module_uniform} with \v{S}pakula's definition, which is why we have changed it to ``$L$-Lipschitz'' (since the statement of Proposition \ref{prop:compact_space_every_module_uniform} is a very desirable one and, in fact, we will need it crucially in the proof of \Poincare duality). \Spakula noted that for a geodesic metric space both notions ($L$-continuous and $L$-Lipschitz) coincide, i.e., for probably all spaces which one wants to consider ours and \v{S}pakula's definition coincide. But note that all the results of \Spakula do also hold with our definition of uniform Fredholm modules.

And last, let us get to the most crucial difference: to define uniform $K$-homology \Spakula does not use operator homotopy as a relation but a weaker form of homotopy (\cite[Definition 2.11]{spakula_uniform_k_homology}). The reasons why we changed this are the following: firstly, the definition of usual $K$-homology uses operator homotopy and it seems desirable to have uniform $K$-homology to be similarly defined. Secondly, \v{S}pakula's proof of \cite[Proposition 4.9]{spakula_uniform_k_homology} is not correct under his notion of homotopy, but it becomes correct if we use operator homotopy as a relation. So by changing the definition we ensure that \cite[Proposition 4.9]{spakula_uniform_k_homology} holds. And thirdly, we will prove in Section \ref{sec:homotopy_invariance} that we get the same uniform $K$-homology groups if we impose weak homotopy (Definition \ref{defn:weak_homotopy}) as a relation instead of operator homotopy. Though our notion of weak homotopies is different from \v{S}pakula's notion of homotopies, all the homotopies that he constructs in his paper \cite{spakula_uniform_k_homology} are weak homotopies, i.e., all the results of him that rely on his notion of homotopy are also true with our definition.

To put it into a nutshell, we changed the definition of uniform $K$-homology in order to make the definition similar to one of usual $K$-homology and to correct \v{S}pakula's proof of \cite[Proposition 4.9]{spakula_uniform_k_homology}. It also seems to be easier to work with our version. Furthermore, all of his results do also hold in our definition. And last, we remark that his results, besides the ones in Sections 8-9 in \cite{spakula_uniform_k_homology}, also hold for non-proper, non-complete spaces.

\subsection{External product}

Now we get to one of the most important technical parts in this article: the construction of the external product for uniform $K$-homology. Its main application will be to deduce homotopy invariance of uniform $K$-homology.

Note that we can construct the product only if the involved metric spaces have jointly bounded geometry (which we will define in a moment). Note that both major classes of spaces on which we want to apply our theory, namely manifolds and simplicial complexes of bounded geometry, do have jointly bounded geometry.

\begin{defn}[Locally bounded geometry, {\cite[Definition 3.1]{spakula_universal_rep}}]\label{defn:locally_bounded_geometry}
A metric space $X$ has \emph{locally bounded geometry}, if it admits a countable Borel decomposition $X = \cup X_i$ such that
\begin{itemize}
\item each $X_i$ has non-empty interior,
\item each $X_i$ is totally bounded, and
\item for all $\varepsilon > 0$ there is an $N > 0$ such that for every $X_i$ there exists an $\varepsilon$-net in $X_i$ of cardinality at most $N$.
\end{itemize}

Note that \Spakula demands in his definition of ``locally bounded geometry'' that the closure of each $X_i$ is compact instead of the total boundedness of them. The reason for this is that he considers only proper spaces, whereas we need a more general notion to encompass also non-complete spaces.
\qed
\end{defn}

\begin{defn}[Jointly bounded geometry]\label{defn:jointly_bounded_geometry}
A metric space $X$ has \emph{jointly coarsely and locally bounded geometry}, if
\begin{itemize}
\item it admits a countable Borel decomposition $X = \cup X_i$ satisfying all the properties of the above Definition \ref{defn:locally_bounded_geometry} of locally bounded geometry,
\item it admits a quasi-lattice $\Gamma \subset X$ (i.e., $X$ has coarsely bounded geometry), and
\item for all $r > 0$ we have $\sup_{y \in \Gamma} \card \{i \ | \ B_r(y) \cap X_i \not= \emptyset\} < \infty$.
\end{itemize}
The last property ensures that there is an upper bound on the number of subsets $X_i$ that intersect any ball of radius $r > 0$ in $X$.
\qed
\end{defn}

\begin{examples}
Recall from Examples \ref{ex:coarsely_bounded_geometry} that manifolds of bounded geometry and simplicial complexes of bounded geometry (i.e., the number of simplices in the link of each vertex is uniformly bounded) equipped with the metric derived from barycentric coordinates have coarsely bounded geometry. Now a moment of reflection reveals that they even have jointly bounded geometry.

In the next Figure \ref{fig:not_jointly_but_others} we give an example of a space $X$ having coarsely and locally bounded geometry, but where the quasi-lattice $\Gamma$ and the Borel decomposition $X = \cup X_i$ are not compatible with each other.
\qed
\end{examples}

\begin{figure}[htbp]
\centering
\includegraphics[scale=0.5]{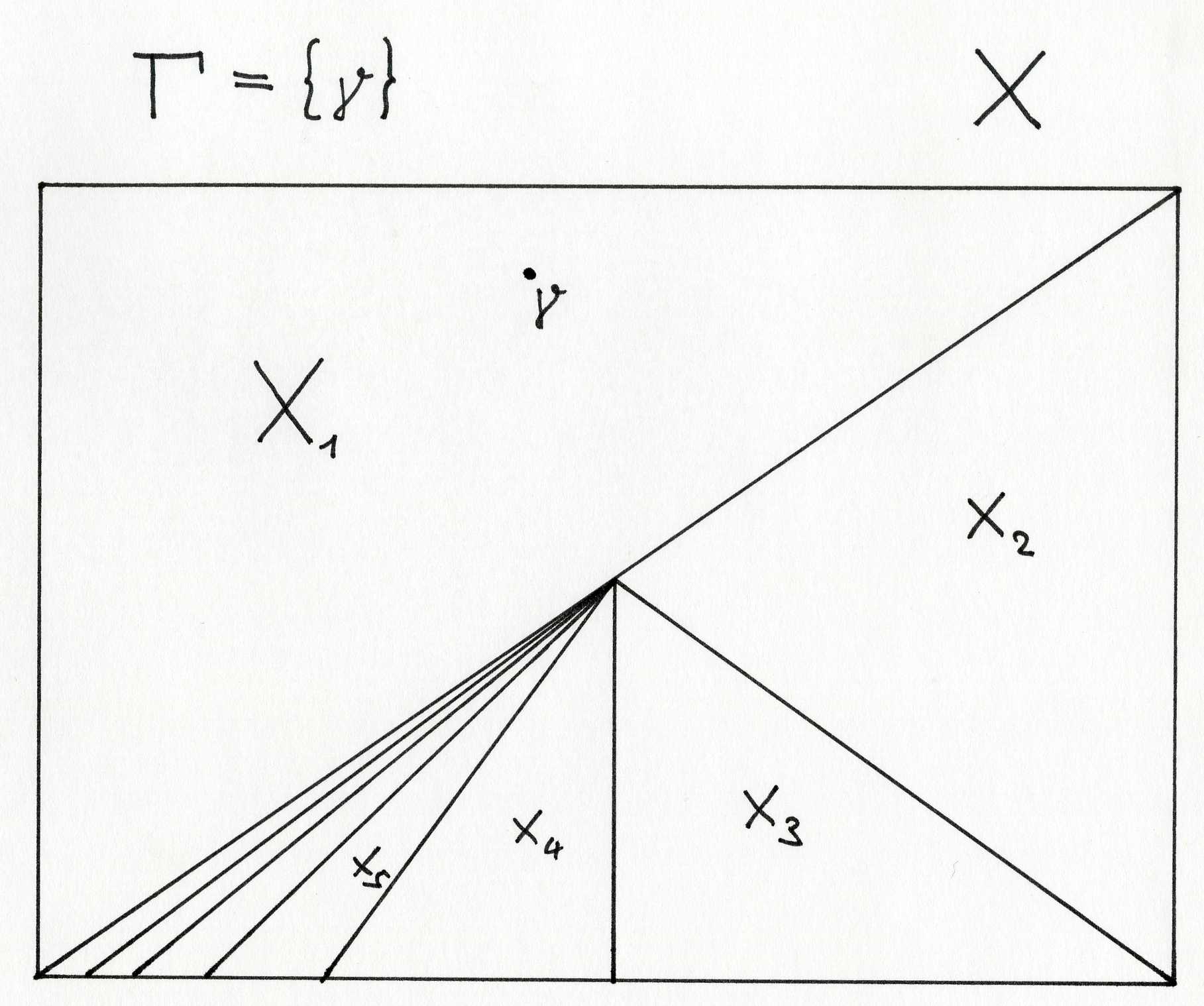}
\caption{Coarsely and locally bounded geometry, but they are not compatible.}
\label{fig:not_jointly_but_others}
\end{figure}

In our construction of the product for uniform $K$-homology we follow the presentation in \cite[Section 9.2]{higson_roe}, where the product is constructed for usual $K$-homology.

Let $X_1$ and $X_2$ be locally compact and separable metric spaces and both having jointly bounded geometry, $(H_1, \rho_1, T_1)$ a $p_1$-multigraded uniform Fredholm module over the space $X_1$ and $(H_2, \rho_2, T_2)$ a $p_2$-multigraded module over $X_2$, and both modules will be assumed to have finite propagation (see Proposition \ref{prop:normalization_finite_prop_speed}).

\begin{defn}[cf. {\cite[Definition 9.2.2]{higson_roe}}]
We define $\rho$ to be the tensor product representation of $C_0(X_1 \times X_2) \cong C_0(X_1) \otimes C_0(X_2)$ on $H := H_1 \hatotimes H_2$, i.e.,
\[\rho(f_1 \otimes f_2) = \rho_1(f_1) \hatotimes \rho_2(f_2) \in \IB(H_1) \hatotimes \IB(H_2)\]
and equip $H_1 \hatotimes H_2$ with the induced $(p_1 + p_2)$-multigrading\footnote{The graded tensor product $H_1 \hatotimes H_2$ is $(p_1 + p_2)$-multigraded if we let the multigrading operators $\epsilon_j$ of $H_1$ act on the tensor product as
\[\epsilon_j(v_1 \otimes v_2) := (-1)^{\deg(v_2)}\epsilon_j(v_1) \otimes v_2\]
for $1 \le j \le p_1$, and for $1 \le j \le p_2$ we let the multigrading operators $\epsilon_{p_1 + j}$ of $H_2$ act as
\[\epsilon_{p_1 + j}(v_1 \otimes v_2) := v_1 \otimes \epsilon_{p_1 + j}(v_2).\]}.

We say that a $(p_1 + p_2)$-multigraded uniform Fredholm module $(H, \rho, T)$ is \emph{aligned} with the modules $(H_1, \rho_1, T_1)$ and $(H_2, \rho_2, T_2)$, if
\begin{itemize}
\item $T$ has finite propagation,
\item for all $f \in C_0(X_1 \times X_2)$ the operators
\[\rho(f) \big( T (T_1 \hatotimes 1) + (T_1 \hatotimes 1) T \big) \rho(\bar f) \text{ and } \rho(f) \big( T (1 \hatotimes T_2) + (1 \hatotimes T_2) T \big) \rho(\bar f)\]
are positive modulo compact operators,\footnote{That is to say, they are positive in the Calkin algebra $\IB(H) / \IK(H)$.} and
\item for all $f\in C_0(X_1 \times X_2)$ the operator $\rho(f) T$ derives $\IK(H_1) \hatotimes \IB(H_2)$, i.e.,
\[[\rho(f) T, \IK(H_1) \hatotimes \IB(H_2)] \subset \IK(H_1) \hatotimes \IB(H_2).\]
\end{itemize}
Since both $H$ and $\rho$ are uniquely determined from $H_1$, $\rho_1$, $H_2$ and $\rho_2$, we will often just say that \emph{$T$ is aligned with $T_1$ and $T_2$}.
\qed
\end{defn}

Our major technical lemma is the following one. It is a uniform version of Kasparov's Technical Lemma, which is suitable for our needs.

\begin{lem}\label{lem:construction_partition_unity}
Let $X_1$ and $X_2$ be locally compact and separable metric spaces that have jointly coarsely and locally bounded geometry.

Then there exist commuting, even, multigraded, positive operators $N_1$, $N_2$ of finite propagation on $H := H_1 \hatotimes H_2$ with $N_1^2 + N_2^2 = 1$ and the following properties:

\begin{enumerate}
\item $N_1 \cdot \big\{ (T_1^2 - 1) \rho_1(f) \hatotimes 1 \ | \ f \in \LLip_{R^\prime}(X_1) \big\} \subset \IK(H_1 \hatotimes H_2)$ is uniformly approximable for all $R^\prime, L > 0$ and analogously for $(T^\ast_1 - T_1)\rho_1(f)$ and for $[T_1, \rho_1(f)]$ instead of $(T_1^2 - 1) \rho_1(f)$,
\item $N_2 \cdot \big\{ 1 \hatotimes (T_2^2 - 1) \rho_2(f) \ | \ f \in \LLip_{R^\prime}(X_2) \big\} \subset \IK(H_1 \hatotimes H_2)$ is uniformly approximable for all $R^\prime, L > 0$ and analogously for $(T_2^\ast - T_2) \rho_2(f)$ and for $[T_2, \rho_2(f)]$ instead of $(T_2^2 - 1) \rho_2(f)$,
\item $\{[N_i, T_1 \hatotimes 1]\rho(f), [N_i, 1 \hatotimes T_2]\rho(f) \ | \ f \in \LLip_{R^\prime}(X_1 \times X_2)\}$ is uniformly approximable for all $R^\prime, L > 0$ and both $i=1,2$,
\item $\big\{ [N_i, \rho(f \otimes 1)], [N_i, \rho(1 \otimes g)] \ | \ f \in \LLip_{R^\prime}(X_1), g \in \LLip_{R^\prime}(X_2) \big\}$ is uniformly approximable for all $R^\prime, L > 0$ and both $i = 1,2$, and
\item both $N_1$ and $N_2$ derive $\IK(H_1) \hatotimes \IB(H_2)$.
\end{enumerate}
\end{lem}

\begin{proof}
Due to the jointly bounded geometry there is a countable Borel decomposition $\{X_{1,i}\}$ of $X_1$ such that each $X_{1,i}$ has non-empty interior, the completions $\{\overline{X_{1,i}}\}$ form an admissible class\footnote{This means that for every $\varepsilon > 0$ there is an $N > 0$ such that in every $\overline{X_{1,i}}$ exists an $\varepsilon$-net of cardinality at most $N$.} of compact metric spaces and for each $R > 0$ we have
\begin{equation}
\label{eq:bound_jointly_bounded_geom}
\sup_i \card \{j \ | \ B_R(X_{1,i}) \cap X_{1,j} \not= \emptyset\} < \infty.
\end{equation}

The completions of the $1$-balls $B_1(X_{1,i})$ are also an admissible class of compact metric spaces and the collection of these open balls forms a uniformly locally finite open cover of $X_1$. We may find a partition of unity $\varphi_{1,i}$ subordinate to the cover $\{B_1(X_{1,i})\}$ such that every function $\varphi_{1,i}$ is $L_0$-Lipschitz for a fixed $L_0 > 0$ (but we will probably have to enlarge the value of $L_0$ a bit in a moment). The same holds also for a countable Borel decomposition $\{X_{2,i}\}$ of $X_2$ and we choose a partition of unity $\varphi_{2,i}$ subordinate to the cover $\{B_1(X_{2,i})\}$ such that every function $\varphi_{2,i}$ is also $L_0$-Lipschitz (by possibly enlargening $L_0$ so that we have the same Lipschitz constant for both partitions of unity).

Since $\{\overline{B_1(X_{1,i})}\}$ is an admissible class of compact metric spaces, we have for each $\varepsilon > 0$ and $L > 0$ a bound independent of $i$ on the number of functions from
\[\varphi_{1,i} \cdot \LLip_c(X_1) := \{ \varphi_{1,i} \cdot f \ | \ f\text{ is }L\text{-Lipschitz, compactly supported and }\|f\|_\infty \le 1\}\]
to form an $\varepsilon$-net in $\varphi_{1,i} \cdot \LLip_c(X_1)$, and analogously for $X_2$ (this can be proved by a similar construction as the one from \cite[Lemma 2.4]{spakula_universal_rep}). We denote this upper bound by $C_{\varepsilon, L}$.

Now for each $N \in \IN$ and $i \in \IN$ we choose $C_{1/N,N}$ functions $\{f_k^{i,N}\}_{k=1, \ldots, C_{1/N, N}}$ from $\varphi_{1,i} \cdot N\text{-}\operatorname{Lip}_c(X_{1,i})$ constituting an $1/N$-net.\footnote{If we need less functions to get an $1/N$-net, we still choose $C_{1/N,N}$ of them. This makes things easier for us to write down.} Analogously we choose $C_{1/N, N}$ functions $\{g_k^{i,N}\}_{k=1, \ldots, C_{1/N, N}}$ from $\varphi_{2,i} \cdot N\text{-}\operatorname{Lip}_c(X_{2,i})$ that are $1/N$-nets.

We choose a sequence $\{u_n \hatotimes 1\} \subset \IB(H_1) \hatotimes \IB(H_2)$ of operators in the following way: $u_n$ will be a projection operator onto a subspace $U_n$ of $H_1$. To define this subspace, we first consider the operators
\begin{equation}\label{eq:operators_X_1_to_approximate}
(T_1^2 - 1)\rho_1(f), \ (T_1 - T_1^\ast)\rho_1(f), \text{ and } [T_1, \rho_1(f)]
\end{equation}
for suitable functions $f \in C_0(X_1)$ that we will choose in a moment. These operators are elements of $\IK(H_1)$ since $(H_1, \rho_1, T_1)$ is a Fredholm module. So up to an error of $2^{-n}$ they are of finite rank and the span $V_n$ of the images of these finite rank operators will be the building block for the subspace $U_n$ on which the operator $u_n$ projects\footnote{This finite rank operators are of course not unique. Recall that every compact operator on a Hilbert space $H$ may be represented in the form $\sum_{n \ge 1} \lambda_n \langle f_n, \largecdot \rangle g_n$, where the values $\lambda_n$ are the singular values of the operator and $\{f_n\}$, $\{g_n\}$ are orthonormal (though not necessarily complete) families in $H$ (but contrary to the $\lambda_n$ they are not unique). Now we choose our finite rank operator to be the operator given by the same sum, but only with the $\lambda_n$ satisfying $\lambda_n \ge 2^{-n}$.} (i.e., we will say in a moment how to enlarge $V_n$ in order to get $U_n$). We choose the functions $f \in C_0(X_1)$ as all the functions from the set $\bigcup \{f_k^{i,N}\}_{k=1, \ldots, C_{1/N,N}}$, where the union ranges over all $i \in \IN$ and $1 \le N \le n$. Note that since the Fredholm module $(H_1, \rho_1, T_1)$ is uniform, the rank of the finite rank operators approximating \eqref{eq:operators_X_1_to_approximate} up to an error of $2^{-n}$ is bounded from above with a bound that depends only on $N$ and $n$, but not on $i$ nor $k$. Since we will have $V_n \subset U_n$, we can already give the first estimate that we will need later:
\begin{equation}\label{eq:Kasparov_estimate_383}
\|(u_n \hatotimes 1)(x \hatotimes 1) - (x \hatotimes 1)\| < 2^{-n},
\end{equation}
where $x$ is one of the operators from \eqref{eq:operators_X_1_to_approximate} for all $f_k^{i,N}$ with $1 \le N \le n$.\footnote{Actually, to have this estimate we would need that $x$ is self-adjoint. We can pass from $x$ to $\tfrac{1}{2}(x + x^\ast)$ and $\tfrac{1}{2i}(x - x^\ast)$, do all the constructions with these self-adjoint operators and get the needed estimates for them, and then we get the same estimates for $x$ but with an additional factor of $2$.} Moreover, denoting by $\chi_{1,i}$ the characteristic function of $B_1(X_{1,i})$, then $\rho_1(\chi_{1,i}) \cdot V_n$ is a subspace of $H_1$ of finite dimension that is bounded independently of $i$.\footnote{We have used here the fact that we may uniquely extend any representation of $C_0(Z)$ to one of the bounded Borel functions $B_b(Z)$ on a space $Z$.} The reason for this is because $T_1$ has finite propagation and the number of functions $f_k^{i,N}$ for fixed $N$ is bounded independently of $i$. For all $n$ we also have $V_n \subset V_{n+1}$ and that the projection operator onto $V_n$ has finite propagation which is bounded independently of $n$.

For each $n \in \IN$ we partition $\chi_{1,i}$ for all $i \in \IN$ into disjoint characteristic functions $\chi_{1,i} = \sum_{j=1}^{J_n} \chi_{1,i}^{j,n}$ such that we may write each function $f_k^{i,N}$ for all $i \in \IN$, $1 \le N \le n$ and $k = 1, \ldots, C_{1/N, N}$ up to an error of $2^{-n-1}$ as a sum $f_k^{i,N} = \sum_{j=1}^{J_n} \alpha_k^{i,N}(j,n) \cdot \chi_{1,i}^{j,n}$ for suitable constants $\alpha_k^{i,N}(j,n)$. Note that since $X_1$ has jointly coarsely and locally bounded geometry, we can choose the upper bounds $J_n$ such that they do not depend on $i$. Now we can finally set $U_n$ as the linear span of $V_n$ and $\rho_1(\chi_{1,i}^{j,n}) \cdot V_n$ for all $i \in \IN$ and $1 \le j \le J_n$. Note that $\rho_1(\chi_{1,i}) \cdot U_n$ is a subspace of $H_1$ of finite dimension that is bounded independently of $i$, that we may choose the characteristic functions $\chi_{1,i}^{j,n}$ such that we have $U_n \subset U_{n+1}$ (by possibly enlargening each $J_n$), and that the projection operator $u_n$ onto $U_n$ has finite propagation which is bounded independently of $n$. Since we have $[u_n, \rho_1(\chi_{1,i}^{j,n})] = 0$ for all $i \in \IN$, $1 \le j \le J_n$ and all $n \in \IN$, we get our second crucial estimate:
\begin{equation}\label{eq:Kasparov_estimate_384}
\|[u_n \hatotimes 1, \rho_1(f_k^{i,N}) \hatotimes 1]\| < 2^{-n}
\end{equation}
for all $i \in \IN$, $k = 1, \ldots, C_{1/N, N}$, $1 \le N \le n$ and all $n \in \IN$.

By an argument similar to the proof of the existence of quasicentral approximate units, we may conclude that for each $n \in \IN$ there exists a finite convex combination $\nu_n$ of the elements $\{u_n, u_{n+1}, \ldots\}$ such that
\begin{equation}\label{eq:Kasparov_estimate_384_2}
\|[\nu_n \hatotimes 1, T_1 \hatotimes 1]\| < 2^{-n} \text{, } \|[\nu_n \hatotimes 1, \epsilon_1 \hatotimes \epsilon_2]\| < 2^{-n} \text{ and } \|[\nu_n \hatotimes 1, \epsilon^j]\| < 2^{-n}
\end{equation}
for all $n \in \IN$, where $\epsilon_1 \hatotimes \epsilon_2$ is the grading operator of $H_1 \hatotimes H_2$ and $\epsilon^j$, $1 \le j \le p_1 + p_2$, are the multigrading operators of $H_1 \hatotimes H_2$. Note that the Estimates \eqref{eq:Kasparov_estimate_383} and \eqref{eq:Kasparov_estimate_384} also hold for $\nu_n$. Note furthermore that we can arrange that the maximal index occuring in the finite convex combination for $\nu_n$ is increasing in $n$.

Now we will construct a sequence $w_n \in \IB(H_1) \hatotimes \IB(H_2)$ with suitable properties. We have that $\nu_n$ is a finite convex combination of the elements $\{u_n, u_{n+1}, \ldots\}$. So for $n \in \IN$ we let $m_n$ denote the maximal occuring index in that combination. Furthermore, we let the projections $p_n \in \IB(H_2)$ be analogously defined as $u_n$, where we consider now the operators
\begin{equation}\label{eq:operators_X_2_to_approximate}
(T_2^2 - 1)\rho_2(g), \ (T_2 - T_2^\ast)\rho_2(g), \text{ and } [T_2, \rho_2(g)]
\end{equation}
for the analogous sets of functions $\bigcup \{g_k^{i,N}\}_{k=1, \ldots, C_{1/N,N}}$ depending on $n \in \IN$. Then we define $w_{n-1} := u_{m_n} \hatotimes  p_n$\footnote{The index is shifted by one so that we get the Estimates \eqref{eq:Kasparov_estimate_386}--\eqref{eq:Kasparov_estimate_388} with $2^{-n}$ and not with $2^{-n+1}$; though this is not necessary for the argument.} and get for all $n \in \IN$ the following:
\begin{equation}
\label{eq:Kasparov_estimate_385_-1}
w_n (\nu_n \hatotimes 1) (1 \hatotimes p_n) = (\nu_n \hatotimes 1) (1 \hatotimes p_n)
\end{equation}
and
\begin{align}
\label{eq:Kasparov_estimate_386}
\| [ w_n, x \hatotimes 1 ] \| & < 2^{-n}\\
\label{eq:Kasparov_estimate_387}
\| [ w_n, 1 \hatotimes y ] \| & < 2^{-n}\\
\label{eq:Kasparov_estimate_388}
\| [ w_n, \rho(f_k^{i,N} \otimes g_k^{i,N}) ] \| & < 2^{-n}
\end{align}
for all $i \in \IN$, $1 \le N \le n$ and $k = 1, \ldots, C_{1/N,N}$, where $x$ is one of the operators from \eqref{eq:operators_X_1_to_approximate} for all $f_k^{i,N}$ and $y$ is one of the operators from \eqref{eq:operators_X_2_to_approximate} for all $g_k^{i,N}$.

Let now $d_n := (w_n - w_{n-1})^{1/2}$. With a suitable index shift we can arrange that firstly, the Estimates \eqref{eq:Kasparov_estimate_386}--\eqref{eq:Kasparov_estimate_388} also hold for $d_n$ instead of $w_n$,\footnote{see \cite[Exercise 3.9.6]{higson_roe}} and that secondly, using Equation \eqref{eq:Kasparov_estimate_385_-1},
\begin{equation}
\label{eq:Kasparov_estimate_385}
\| d_n (\nu_n \hatotimes 1) y \| < 2^{-n},
\end{equation}
where $y$ is again one of the operators from \eqref{eq:operators_X_2_to_approximate} for all $g_k^{i,N}$ and $1 \le N \le n$.

Now as in the same way as we constructed $\nu_n$ out of the $u_n$s, we construct $\delta_n$ as a finite convex combination of the elements $\{d_n, d_{n+1}, \ldots\}$ such that
\begin{equation*}
\|[\delta_n, T_1 \hatotimes 1]\| < 2^{-n} \text{, } \|[\delta_n, 1 \hatotimes T_2]\| < 2^{-n} \text{, } \|[\delta_n, \epsilon_1 \hatotimes \epsilon_2]\| < 2^{-n} \text{ and } \|[\delta_n, \epsilon^j]\| < 2^{-n},\notag
\end{equation*}
where $\epsilon_1 \hatotimes \epsilon_2$ is the grading operator of $H_1 \hatotimes H_2$ and $\epsilon^j$ for $1 \le j \le p_1 + p_2$ are the multigrading operators of $H_1 \hatotimes H_2$. Clearly, all the Estimates \eqref{eq:Kasparov_estimate_386}--\eqref{eq:Kasparov_estimate_385} also hold 
for the operators $\delta_n$.

Define $X := \sum \delta_n \nu_n \delta_n$. It is a positive operator of finite propagation and fulfills the Points 2--4 that $N_2$ should have. The arguments for this are analogous to the ones given at the end of the proof of \cite[Kasparov's Technical Theorem 3.8.1]{higson_roe}, but we have to use all the uniform approximations that we additionally have (to use them, we have to cut functions $f \in \LLip_{R^\prime}(X_1)$ down to the single ``parts'' $X_{1,i}$ of $X_1$ by using the partition of unity $\{\varphi_{1,i}\}$ that we have chosen at the beginning of this proof, and analogously for $X_2$). Furthermore, the operator $1-X$ fulfills the desired Points 1, 3 and 4 that $N_1$ should fulfill. That both $X$ and $1-X$ derive $\IK(H_1) \hatotimes \IB(H_2)$ is clear via construction. Since $X$ commutes modulo compact operators with the grading and multigrading operators, we can average it over them so that it becomes an even and multigraded operator and $X$ and $1-X$ still have all the above mentioned properties.

Finally, we set $N_1 := (1-X)^{1/2}$ and $N_2 := X^{1/2}$.
\end{proof}

Now we will use this technical lemma to construct the external product and to show that it is well-defined on the level of uniform $K$-homology.

\begin{prop}\label{prop:external_prod_exists}
Let $X_1$ and $X_2$ be locally compact and separable metric spaces that have jointly coarsely and locally bounded geometry.

Then there exists a $(p_1 + p_2)$-multigraded uniform Fredholm module $(H, \rho, T)$ which is aligned with the modules $(H_1, \rho_1, T_1)$ and $(H_2, \rho_2, T_2)$.

Furthermore, any two such aligned Fredholm modules are operator homotopic and this operator homotopy class is uniquely determined by the operator homotopy classes of $(H_1, \rho_1, T_1)$ and $(H_2, \rho_2, T_2)$.
\end{prop}

\begin{proof}
We invoke the above Lemma \ref{lem:construction_partition_unity} to get operators $N_1$ and $N_2$ and then set
\[T := N_1(T_1 \hatotimes 1) + N_2(1 \hatotimes T_2).\]

To deduce that $(H, \rho, T)$ is a uniform Fredholm module, we have to use the following facts (additionally to the ones that $N_1$ and $N_2$ have): that $T_1$ and $T_2$ have finite propagation and are odd (we need that $(T_1 \hatotimes 1)(1 \hatotimes T_2) + (1 \hatotimes T_2)(T_1 \hatotimes 1) = 0$). To deduce that it is a multigraded module, we need that we constructed $N_1$ and $N_2$ as even and multigraded operators on $H$.

It is easily seen that for all $f \in C_0(X_1 \times X_2)$
\[\rho(f) \big( T (T_1 \hatotimes 1) + (T_1 \hatotimes 1) T \big) \rho(\bar f) \text{ and } \rho(f) \big( T (1 \hatotimes T_2) + (1 \hatotimes T_2) T \big) \rho(\bar f)\]
are positive modulo compact operators and that $\rho(f)T$ derives $\IK(H_1) \hatotimes \IB(H_2)$, i.e., we conclude that $T$ is aligned with $T_1$ and $T_2$.

Since all four operators $T_1$, $T_2$, $N_1$ and $N_2$ have finite propagation, $T$ has also finite propagation.

Suppose that $T^\prime$ is another operator aligned with $T_1$ and $T_2$. We construct again operators $N_1$ and $N_2$ using the above Lemma \ref{lem:construction_partition_unity}, but we additionally enforce
\[\|[w_n, \rho(f^{i,N}_k \otimes g_k^{i,N}) T^\prime]\| < 2^{-n}\]
analogously as we did it there to get Equation \eqref{eq:Kasparov_estimate_388}. So $N_1$ and $N_2$ will commute modulo compact operators with $\rho(f)T^\prime$ for all functions $f \in C_0(X_1 \times X_2)$. Again, we set $T := N_1(T_1 \hatotimes 1) + N_2(1 \hatotimes T_2)$. Since $N_1$ and $N_2$ commute modulo compacts with $\rho(f)T^\prime$ for all $f \in C_0(X_1 \times X_2)$ and since $T^\prime$ is aligned with $T_1$ and $T_2$, we conclude
\[\rho(f)(T T^\prime + T^\prime T) \rho(\bar f) \ge 0\]
modulo compact operators for all functions $f \in C_0(X_1 \times X_2)$. Using a uniform version of \cite[Proposition 8.3.16]{higson_roe} we conclude that $T$ and $T^\prime$ are operator homotopic via multigraded, uniform Fredholm modules. We conclude that every aligned module is operator homotopic to one of the form that we constructed above, i.e., to one of the form $N_1(T_1 \hatotimes 1) + N_2(1 \hatotimes T_2)$. But all such operators are homotopic to one another: they are determined by the operator $Y = N_2^2$ used in the proof of the above lemma and the set of all operators with the same properties as $Y$ is convex.

At last, suppose that one of the operators is varied by an operator homotopy, e.g., $T_1$ by $T_1(t)$. Then, in order to construct $N_1$ and $N_2$, we enforce in Equation \eqref{eq:Kasparov_estimate_384_2} instead of $\|[\nu_n \hatotimes 1, T_1 \hatotimes 1]\| < 2^{-n}$ the following one:
\[\|[\nu_n \hatotimes 1, T_1(j/n) \hatotimes 1]\| < 2^{-n}\]
for $0 \le j \le n$. Now we may define
\[T(t) := N_1(T_1(t) \hatotimes 1) + N_2(1 \hatotimes T_2),\]
i.e., we got operators $N_1$ and $N_2$ which are independent of $t$ but still have all the needed properties. This gives us the desired operator homotopy.
\end{proof}

\begin{defn}[External product]
The \emph{external product} of the multigraded uniform Fredholm modules $(H_1, \rho_1, T_1)$ and $(H_2, \rho_2, T_2)$ is a multigraded uniform Fredholm module $(H, \rho, T)$ which is aligned with $T_1$ and $T_2$. We will use the notation $T := T_1 \times T_2$.

By the above Proposition \ref{prop:external_prod_exists} we know that if the locally compact and separable metric spaces $X_1$ and $X_2$ both have jointly coarsely and locally bounded geometry, then the external product always exists, that it is well-defined up to operator homotopy and that it descends to a well-defined product on the level of uniform $K$-homology:
\[K_{p_1}^u(X_1) \times K_{p_2}^u(X_2) \to K_{p_1+p_2}^u(X_1 \times X_2)\]
for $p_1, p_2 \ge 0$. Furthermore, this product is bilinear.\footnote{To see this, suppose that, e.g., $T_1 = T_1^\prime \oplus T_1^{\prime \prime}$. Then it suffices to show that $T_1^\prime \times T_2 \oplus T_1^{\prime \prime} \times T_2$ is aligned with $T_1$ and $T_2$, which is not hard to do.}
\qed
\end{defn}

For the remaining products (i.e., the product of an ungraded and a multigraded module, resp., the product of two ungraded modules) we can appeal to the formal $2$-periodicity.

Associativity of the external product and the other important properties of it may be shown as in the non-uniform case. Let us summarize them in the following theorem:

\begin{thm}[External product for uniform $K$-homology]\label{thm:external_prod_homology}
Let $X_1$ and $X_2$ be locally compact and separable metric spaces of jointly bounded geometry\footnote{see Definition \ref{defn:jointly_bounded_geometry}}.

Then there exists an associative product
\[\times \colon K_{p_1}^u(X_1) \otimes K_{p_2}^u(X_2) \to K^u_{p_1 + p_2}(X_1 \times X_2)\]
for $p_1 , p_2 \ge -1$ with the following properties:
\begin{itemize}
\item for the flip map $\tau\colon X_1 \times X_2 \to X_2 \times X_1$ and all elements $[T_1] \in K_{p_1}^u(X_1)$ and $[T_2] \in K_{p_2}^u(X_2)$ we have
\[\tau_{\ast}[T_1 \times T_2] = (-1)^{p_1 p_2} [T_2 \times T_1],\]
\item we have for $g\colon Y \to Z$ a uniformly cobounded, proper Lipschitz map and elements $[T] \in K_{p_1}^u(X)$ and $[S] \in K_{p_2}^u(Y)$
\[(\id_X \operatorname{\times} g)_\ast [T \times S] = [T] \times g_\ast[S] \in K^u_{p_1 + p_2}(X \times Z),\]
and
\item denoting the generator of $K_0^u(\pt) \cong \IZ$ by $[1]$, we have
\[[T] \times [1] = [T] = [1] \times [T] \in K_\ast^u(X)\]
for all $[T] \in K_\ast^u(X)$.
\end{itemize}
\end{thm}

\subsection{Homotopy invariance}\label{sec:homotopy_invariance}

Let $X$ and $Y$ be locally compact, separable metric spaces with jointly bounded geometry and let $g_0, g_1 \colon X \to Y$ be uniformly cobounded, proper and Lipschitz maps which are homotopic in the following sense: there is a uniformly cobounded, proper and Lipschitz map $G\colon X \times [0,1] \to Y$ with $G(x,0) = g_0(x)$ and $G(x,1) = g_1(x)$ for all $x \in X$.

\begin{thm}\label{thm:homotopy_equivalence_k_hom}
If $g_0, g_1\colon X \to Y$ are homotopic in the above sense, then they induce the same maps $(g_0)_\ast = (g_1)_\ast \colon K_\ast^u(X) \to K_\ast^u(Y)$ on uniform $K$-homology.
\end{thm}

The proof of the above theorem is completely analogous to the non-uniform case and uses the external product. Furthermore, the above theorem is a special case of the following invariance of uniform $K$-homology under weak homotopies: given a uniform Fredholm module $(H, \rho, T)$ over $X$, the push-forward of it under $g_i$ is defined as $(H, \rho \circ g_i^\ast, T)$ and it is easily seen that these modules are weakly homotopic via the map $G$.

\begin{defn}[Weak homotopies]\label{defn:weak_homotopy}
Let a time-parametrized family of uniform Fredholm modules $(H, \rho_t, T_t)$ for $t \in [0,1]$ satisfy the following properties:
\begin{itemize}
\item the family $\rho_t$ is pointwise strong-$^\ast$ operator continuous, i.e., for all $f \in C_0(X)$ we get a path $\rho_t(f)$ in $\IB(H)$ that is continuous in the strong-$^\ast$ operator topology\footnote{Recall that if $H$ is a Hilbert space, then the \emph{strong-$^\ast$ operator topology} on $\IB(H)$ is generated by the family of seminorms $p_v(T) := \|Tv\| + \|T^\ast v\|$ for all $v \in H$, where $T \in \IB(H)$.},
\item the family $T_t$ is continuous in the strong-$^\ast$ operator topology on $\IB(H)$, i.e., for all $v \in H$ we get norm continuous paths $T_t(v)$ and $T_t^\ast(v)$ in $H$, and
\item for all $f \in C_0(X)$ the families of compact operators $[T_t, \rho_t(f)]$, $(T_t^2 - 1)\rho_t(f)$ and $(T_t-T_t^\ast)\rho_t(f)$ are norm continuous.
\end{itemize}
Then we call it a \emph{weak homotopy} between $(H, \rho_0, T_0)$ and $(H, \rho_1, T_1)$.
\qed
\end{defn}

\begin{rem}
If $\rho_t$ is pointwise norm continuous and $T_t$ is norm continuous, then the modules are weakly homotopic. So weak homotopy generalizes operator homotopy.
\qed
\end{rem}

\begin{thm}\label{thm:weak_homotopy_equivalence_K_hom}
Let $(H, \rho_0, T_0)$ and $(H, \rho_1, T_1)$ be weakly homotopic uniform Fredholm modules over a locally compact and separable metric space $X$ of jointly bounded geometry.

Then they define the same uniform $K$-homology class.
\end{thm}

\begin{proof}
Let our weakly homotopic family $(H, \rho_t, T_t)$ be parametrized by $t \in [0, 2\pi]$ so that our notation here will coincide with the one in the proof of \cite[Theorem 1 in §6]{kasparov_KK} that we mimic. Furthermore, we assume that $\rho_t$ and $T_t$ are constant in the intervals $[0, 2\pi/3]$ and $[4\pi / 3, 2\pi]$

We consider the graded Hilbert space $\IH := H \hatotimes (L^2[0,2\pi] \oplus L^2[0, 2\pi])$ (where the space $L^2[0,2\pi] \oplus L^2[0, 2\pi]$ is graded by interchanging the summands).

The family $T_t$ maps continuous paths $v_t$ in $H$ again to continuous paths $T_t(v_t)$: since the family $T_t$ is continuous in the strong-$^\ast$ operator topology and since it is defined on the compact interval $[0,1]$, we conclude with the uniform boundedness principle $\sup_t \|T_t\|_{op} < \infty$. Now if $t_n \to t$ is a convergent sequence, we get
\begin{align*}
\|T_{t_n} (v_{t_n}) - T_t (v_t)\| & \le \|T_{t_n} (v_{t_n}) - T_{t_n} (v_t)\| + \|T_{t_n} (v_t) - T_t (v_t)\|\\
& \le \underbrace{\|T_{t_n}\|_{op}}_{< \infty} \cdot \underbrace{\|v_{t_n} - v_t\|}_{\to 0} + \underbrace{\|(T_{t_n} - T_t)(v_t)\|}_{\to 0},
\end{align*}
where the second limit to $0$ holds due to the continuity of $T_t$ in the strong-$^\ast$ operator topology. So the family $T_t$ maps the dense subspace $H \otimes C[0,2\pi]$ of $H \otimes L^2[0,2\pi]$ into itself, and since it is norm bounded from above by $\sup_t \|T_t\|_{op} < \infty$, it defines a bounded operator on $H \otimes L^2[0,2\pi]$. We define an odd operator $\begin{pmatrix} 0 & T_t^\ast \\ T_t & 0\end{pmatrix}$ on $\IH$, which we also denote by $T_t$ (there should arise no confusion by using the same notation here).

Since $\rho_t(f)$ is strong-$^\ast$ continuous in $t$, we can analogously show that it maps continuous paths $v_t$ in $H$ again to continuous paths $\rho_t(f) (v_t)$, and it is norm bounded from above by $\|f\|_\infty$. because we have $\|\rho_t(f)\|_{op} \le \|f\|_\infty$ for all $t \in [0,1]$ since $\rho_t$ are representations of $C^\ast$-algebras. So $\rho_t(f)$ defines a bounded operator on $H \otimes L^2[0,2\pi]$ and we can get a representation $\rho_t \oplus \rho_t$ of $C_0(X)$ on $\IH$ by even operators, that we denote by the symbol $\rho_t$ (again, no confusion should arise by using the same notation).

We consider now the uniform Fredholm module
\[(\IH, \rho_t, N_1(T_t) + N_2(1 \hatotimes T(f)),\]
where $T(f)$ is defined as in the proof of \cite[Theorem 1 in §6]{kasparov_KK} (unfortunately, the overloading of the symbol ``$T$'' is unavoidable here). For the convenience of the reader, we will recall the definition of the operator $T(f)$ in a moment. That we may find a suitable partition of unity $N_1, N_2$ is due to the last bullet point in the definition of weak homotopies, and the construction of $N_1, N_2$ proceeds as in the end of the proof of our Proposition \ref{prop:external_prod_exists}.

To define $T(f)$, we first define an operator $d\colon L^2[0, 2\pi] \to L^2[0, 2\pi]$ using the basis $1, \ldots, \cos nx, \ldots, \sin nx, \ldots$ by the formulas
\[d(1) := 0 \text{, } d(\sin nx) := \cos nx \text{ and } d(\cos nx) := - \sin nx.\]
This operator $d$ is anti-selfadjoint, $d^2 + 1 \in \IK(L^2[0, 2\pi])$, and $d$ commutes modulo compact operators with multiplication by functions from $C[0, 2\pi]$. Let $f \in C[0, 2\pi]$ be a continuous, real-valued function with $|f(x)| \le 1$ for all $x \in [0, 2\pi]$, $f(0) = 1$ and $f(2\pi) = -1$. Then we set $T_1(f) := f - \sqrt{1 - f^2}\cdot d \in \IB(L^2[0, 2\pi])$. This operator $T_1(f)$ is Fredholm with Fredholm index $1$, both $1 - T_1(f) \cdot T_1(f)^\ast$ and $1 - T_1(f)^\ast \cdot T_1(f)$ are compact, and $T_1(f)$ commutes modulo compacts with multiplication by functions from $C[0, 2\pi]$. Furthermore, any two operators of the form $T_1(f)$ (for different $f$) are connected by a norm continuous homotopy consisting of operators having the same form. Finally, we define $T(f) := \begin{pmatrix}0 & T_1(f)^\ast \\ T_1(f) & 0\end{pmatrix} \in \IB(L^2[0,2\pi] \oplus L^2[0, 2\pi])$.

We assume the our homotopies $\rho_t$ and $T_t$ are constant in the intervals $[0, 2\pi/3]$ and $[4\pi / 3, 2\pi]$. Furthermore, we set
\[f(t) := \begin{cases} \cos 3t, & 0 \le t \le \pi / 3,\\ -1, & \pi/3 \le t \le 2 \pi.\end{cases}\]
Then $T_1(f)$ commutes with the projection $P$ onto $L^2[0, 2\pi / 3]$, $P \cdot T_1(f)$ is an operator of index $1$ on $L^2[0, 2\pi / 3]$, and $(1-P) T_1(f) \equiv -1$ on $L^2[2\pi/3, 2\pi]$. We choose $\alpha(t) \in C[0, 2\pi]$ with $0 \le \alpha(t) \le 1$, $\alpha(t) = 0$ for $t \le \pi / 3$, and $\alpha(t) = 1$ for $t \ge 2\pi / 3$. Using a norm continuous homotopy, we replace $N_1$ and $N_2$ by
\[\widetilde{N_1} := \sqrt{1 \hatotimes (1 - \alpha)} \cdot N_1 \cdot \sqrt{1 \hatotimes (1 - \alpha)}\]
and
\[\widetilde{N_2} := 1 \hatotimes \alpha + \sqrt{1 \hatotimes (1 - \alpha)} \cdot N_2 \cdot \sqrt{1 \hatotimes (1 - \alpha)}.\]
The operator $\widetilde{N_1}(T_t) + \widetilde{N_2}(1 \hatotimes T(f))$ commutes with $1 \hatotimes (P \oplus P)$ and we obtain for the decomposition $L^2[0,2\pi] \oplus L^2[0, 2\pi] = \image(P \oplus P) \oplus \image(1 - P \oplus P)$
\[\big( \IH, \rho_t, \widetilde{N_1}(T_t) + \widetilde{N_2}(1 \hatotimes T(f) \big) = \big( (H, \rho_0, T_0) \times [1] \big) \oplus \big(\text{degenerate}\big),\]
where $[1] \in K_0^u(\pt)$ is the multiplicative identity (see the third point of Theorem \ref{thm:external_prod_homology}) and recall that we assumed that $\rho_t$ and $T_t$ are constant in the intervals $[0, 2\pi/3]$ and $[4\pi / 3, 2\pi]$.

Setting
\[f(t) := \begin{cases} 1, & 0 \le t \le 5\pi / 3,\\ -\cos 3t, & 5\pi/3 \le t \le 2 \pi,\end{cases}\]
we get analogously
\[\big( \IH, \rho_t, \overline{N_1}(T_t) + \overline{N_2}(1 \hatotimes T(f) \big) = \big(\text{degenerate}\big) \oplus \big( (H, \rho_1, T_1) \times [1] \big),\]
for suitably defined operators $\overline{N_1}$ and $\overline{N_2}$ (their definition is similar to the one of $\widetilde{N_1}$ and $\widetilde{N_2}$). Putting all the homotopies of this proof together, we get that the modules $\big( (H, \rho_0, T_0) \times [1] \big) \oplus \big(\text{degenerate}\big)$ and $\big( (H, \rho_1, T_1) \times [1] \big) \oplus \big(\text{degenerate}\big)$ are operator homotopic, from which the claim follows.
\end{proof}

\subsection{Rough Baum--Connes conjecture}\label{sec:rough_BC}

\Spakula constructed in \cite[Section 9]{spakula_uniform_k_homology} the \emph{rough\footnote{We could have also called it the \emph{uniform coarse} assembly map, but the uniform coarse category is also called the rough category and therefore we stick to this shorter name.} assembly map}
\[\mu_u \colon K_\ast^u(X) \to K_\ast(C_u^\ast(Y)),\]
where $Y \subset X$ is a uniformly discrete quasi-lattice, $X$ a proper metric space, and $C_u^\ast(Y)$ the uniform Roe algebra of $Y$.\footnote{Recall that one possible model for the uniform Roe algebra $C_u^\ast(Y)$ is the norm closure of the $^\ast$-algebra of all finite propagation operators in $\IB(\ell^2(Y))$ with uniformly bounded coefficients. Another version is the norm closure of the $^\ast$-algebra of all finite propagation, uniformly locally compact operators in $\IB(\ell^2(Y)\otimes H)$ with uniformly bounded coefficients. \v{S}pakula--Willett \cite[Proposition 4.7]{spakula_willett_2} proved that these two versions are strongly Morita equivalent.} In this section we will discuss implications on the rough assembly map following from the properties of uniform $K$-homology that we have proved in the last sections.

Using homotopy invariance of uniform $K$-homology we will strengthen \v{S}pakula's results from \cite[Section 10]{spakula_uniform_k_homology}.

\begin{defn}[Rips complexes]
Let $Y$ be a discrete metric space and let $d \ge 0$. The \emph{Rips complex $P_d(Y)$ of $Y$} is a simplicial complex, where
\begin{itemize}
\item the vertex set of $P_d(Y)$ is $Y$, and
\item vertices $y_0, \ldots, y_q$ span a $q$-simplex if and only if we have $d(y_i, y_j) \le d$ for all $0 \le i, j \le q$.
\end{itemize}
Note that if $Y$ has coarsely bounded geometry, then the Rips complex $P_d(Y)$ is uniformly locally finite and finite dimensional and therefore also, especially, a simplicial complex of bounded geometry (i.e., the number of simplices in the link of each vertex is uniformly bounded). So if we equip $P_d(Y)$ with the metric derived from barycentric coordinates, $Y \subset P_d(Y)$ becomes a quasi-lattice (cf. Examples \ref{ex:coarsely_bounded_geometry}).
\qed
\end{defn}

Now we may state the \emph{rough Baum--Connes conjecture}:

\begin{conj}
Let $Y$ be a proper and uniformly discrete metric space with coarsely bounded geometry.

Then
\[\mu_u \colon \lim_{d \to \infty} K_\ast^u(P_d(Y)) \to K_\ast(C_u^\ast(Y))\]
is an isomorphism.
\end{conj}

Let us relate the conjecture quickly to manifolds of bounded geometry. First we need the following notion:

\begin{defn}[Equicontinuously contractible spaces]
A metric space $X$ is called \emph{equicontinuously contractible}, if for all $r > 0$ the collection of balls $\{B_r(x)\}_{x \in X}$ is equicontinuously contractible (this means that the collection of the contracting homotopies is equicontinuous).
\qed
\end{defn}

\begin{example}
Universal covers of aspherical Riemannian manifolds equipped with the pull-back metric are equicontinuously contractible.
\qed
\end{example}

\begin{rem}
Note that equicontinuous contractibility is a slight strengthening of uniform contractibility: a metric space $X$ is called \emph{uniformly contractible}, if for every $r > 0$ there is an $s > 0$ such that every ball $B_r(x)$ can be contracted to a point in the ball $B_s(x)$.
\qed
\end{rem}

\begin{thm}
Let $M$ be an equicontinuously contractible manifold of bounded geometry and without boundary and let $Y \subset M$ be a uniformly discrete quasi-lattice in $M$.

Then we have a natural isomorphism
\[\lim_{d \to \infty} K^u_\ast(P_d(Y)) \cong K^u_\ast(M).\]
\end{thm}

The proof of this theorem is analogous to the corresponding non-uniform statement $\lim_{d \to \infty} K_\ast(P_d(Y)) \cong K_\ast(M)$ from \cite[Theorem 3.2]{yu_coarse_baum_connes_conj} and uses crucially the homotopy invariance of uniform $K$-homology.

Let us now relate the rough Baum--Connes conjecture to the usual Baum--Connes conjecture: let $\Gamma$ be a countable, discrete group and denote by $|\Gamma|$ the metric space obtained by endowing $\Gamma$ with a proper, left-invariant metric. Then $|\Gamma|$ becomes a proper, uniformly discrete metric space with coarsely bounded geometry. Note that we can always find such a metric and that any two of such metrics are quasi-isometric. If $\Gamma$ is finitely generated, an example is the word metric.

\Spakula proved in \cite[Corollary 10.3]{spakula_uniform_k_homology} the following equivalence of the rough Baum--Connes conjecture with the usual one: let $\Gamma$ be a torsion-free, countable, discrete group. Then the rough assembly map
\[\mu_u \colon \lim_{d \to \infty} K_\ast^u(P_d|\Gamma|) \to K_\ast(C_u^\ast|\Gamma|)\]
is an isomorphism if and only if the Baum--Connes assembly map
\[\mu \colon K_\ast^\Gamma(\underline{E}\Gamma; \ell^\infty(\Gamma)) \to K_\ast(C_r^\ast(\Gamma, \ell^\infty(\Gamma)))\]
for $\Gamma$ with coefficients in $\ell^\infty(\Gamma)$ is an isomorphism. For the definition of the Baum--Connes assembly map with coefficients the unfamiliar reader may consult the original paper \cite[Section 9]{baum_connes_higson}. Furthermore, the equivalence of the usual (i.e., non-uniform) coarse Baum--Connes conjecture with the Baum--Connes conjecture with coefficients in $\ell^\infty(\Gamma, \IK)$ was proved by Yu in \cite[Theorem 2.7]{yu_baum_connes_conj_coarse_geom}.

\Spakula mentioned in \cite[Remark 10.4]{spakula_uniform_k_homology} that the above equivalence does probably also hold without any assumptions on the torsion of $\Gamma$, but the proof of this would require some degree of homotopy invariance of uniform $K$-homology. So again we may utilize our proof of the homotopy invariance of uniform $K$-homology and therefore drop the assumption about the torsion of $\Gamma$.

\begin{thm}\label{thm:BC_equiv_uniform_coarse}
Let $\Gamma$ be a countable, discrete group.

Then the rough assembly map
\[\mu_u \colon \lim_{d \to \infty} K_\ast^u(P_d|\Gamma|) \to K_\ast(C_u^\ast|\Gamma|)\]
is an isomorphism if and only if the Baum--Connes assembly map
\[\mu \colon K_\ast^\Gamma(\underline{E}\Gamma; \ell^\infty(\Gamma)) \to K_\ast(C_r^\ast(\Gamma, \ell^\infty(\Gamma)))\]
for $\Gamma$ with coefficients in $\ell^\infty(\Gamma)$ is an isomorphism.
\end{thm}

\subsection{Homology classes of uniform elliptic operators}\label{sec:homology_classes_of_PDOs}

We will show that symmetric, elliptic uniform pseudodifferential operators of positive order naturally define classes in uniform $K$-homology. This result is a crucial generalization of \cite[Theorem 3.1]{spakula_uniform_k_homology}, where this statement is proved for generalized Dirac operators.

First we need a definition and then we will plunge right into the main result:

\begin{defn}[Normalizing functions]\label{defn:normalizing_function}
A smooth function $\chi\colon \IR \to [-1, 1]$ with
\begin{itemize}
\item $\chi$ is odd, i.e., $\chi(x) = -\chi(-x)$ for all $x \in \IR$,
\item $\chi(x) > 0$ for all $x > 0$, and
\item $\chi(x) \to \pm 1$ for $x \to \pm \infty$
\end{itemize}
is called a \emph{normalizing function}.
\qed
\end{defn}

\begin{figure}[ht]
\centering
\includegraphics[scale=0.6]{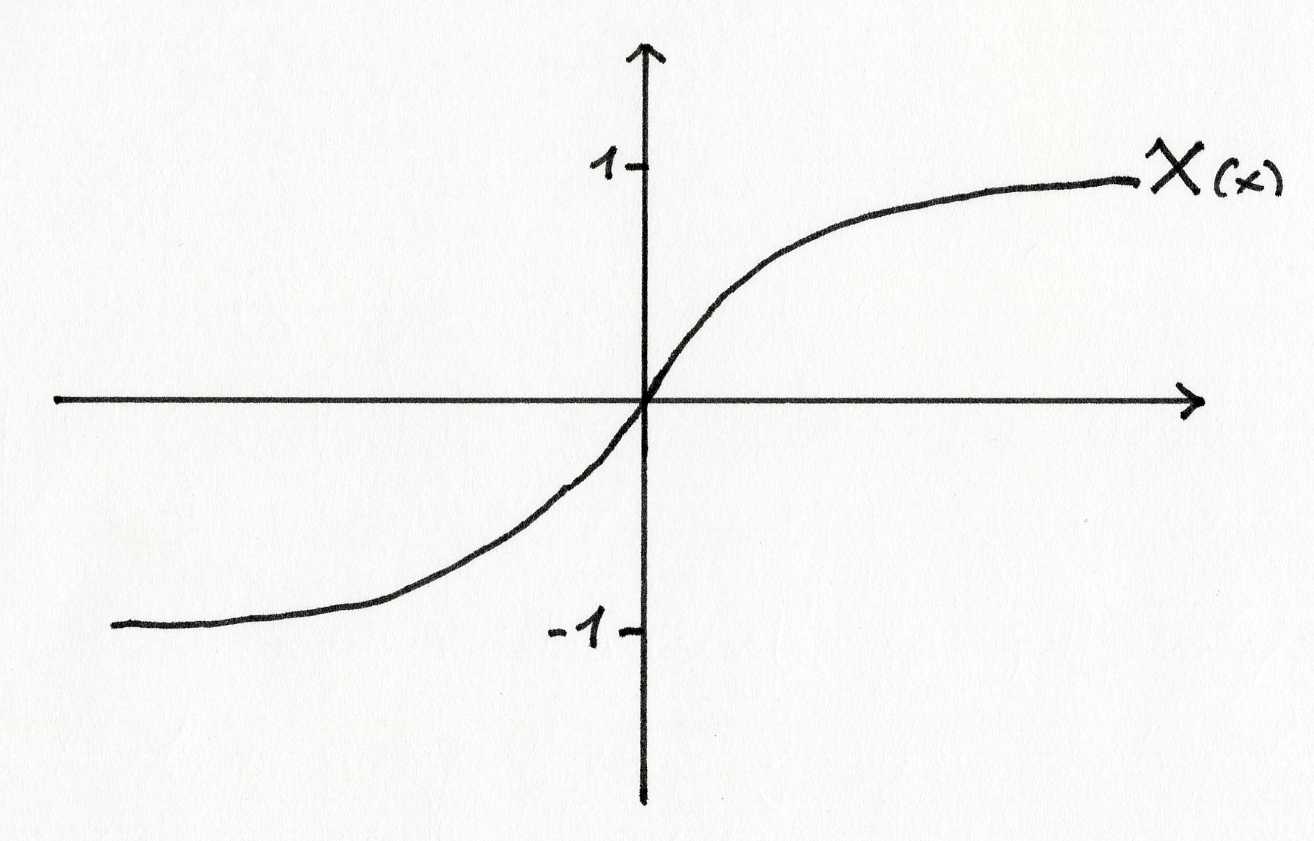}
\caption{A normalizing function.}
\end{figure}

\begin{thm}\label{thm:elliptic_symmetric_PDO_defines_uniform_Fredholm_module}
Let $M$ be a manifold of bounded geometry and without boundary, $E \to M$ be a $p$-multigraded vector bundle of bounded geometry, $P \in \UPsiDO^k(E)$ be a symmetric and elliptic uniform pseudodifferential operator on $E$ of positive order $k \ge 1$, and let $P$ be odd and multigraded.

Then $(H, \rho, \chi(P))$ is a $p$-multigraded uniform Fredholm module over $M$, where the Hilbert space is $H := L^2(E)$, the representation $\rho\colon C_0(M) \to \IB(H)$ is the one via multiplication operators and $\chi$ is a normalizing function. Furthermore, the uniform $K$-homology class $[(H, \rho, \chi(P))] \in K_{p}^u(M)$ does not depend on the choice of $\chi$.
\end{thm}

\begin{proof}
Recall from Definition \ref{defn:uniform_fredholm_modules} that for the first statement that $(H, \rho, \chi(P))$ defines an ungraded uniform Fredholm module over $M$ we have to show that $\chi(P)$ is uniformly pseudolocal and that $\chi(P)^2 - 1$ and $\chi(P) - \chi(P)^\ast$ are uniformly locally compact.

Since $\chi$ is real-valued and $P$ essentially self-adjoint by Proposition \ref{prop:elliptic_PDO_essentially_self-adjoint}, we have $\chi(P) - \chi(P)^\ast = 0$, i.e., the operator $\chi(P) - \chi(P)^\ast$ is trivially uniformly locally compact. Moreover, since we have $\chi(P)^2 - 1 = (\chi^2 - 1)(P)$ and $\chi^2-1 \in C_0(\IR)$, we conclude with Corollary \ref{cor:g(P)_uniformly_locally_compact_g_vanishing_at_infinity} that $\chi(P)^2 - 1$ is uniformly locally compact.

Because the difference of two normalizing functions is a function from $C_0(\IR)$, we conclude from the same corollary that in order to show that $\chi(P)$ is uniformly pseudolocal, it suffices to show this for one particular normalizing function (and secondly, we get that the class $[(H, \rho, \chi(P))]$ is independent of the concrete choice of $\chi$ due to Lemma \ref{lem:compact_perturbations}).

From now on we proceed as in the proof of \cite[Theorem 3.1]{spakula_uniform_k_homology} using the same formulas: we choose the particular normalizing function $\chi(x) := \frac{x}{\sqrt{1+x^2}}$ to prove that $\chi(P)$ is uniformly pseudolocal. We have $\chi(P) = \frac{2}{\pi} \int_0^\infty \frac{P}{1 + \lambda^2 + P^2} d\lambda$ with convergence of the integral in the strong operator topology\footnote{This follows from the equality $\frac{x}{\sqrt{1 + x^2}} = \frac{2}{\pi} \int_0^\infty \frac{x}{1 + \lambda^2 + x^2} d\lambda$ for all $x \in \IR$.} and get then for $f \in \LLip_R(M)$
\begin{equation*}
[\rho(f), \chi(P)] = \frac{2}{\pi} \int_0^\infty \frac{1}{1 + \lambda^2 + P^2} \big( (1+\lambda^2)[\rho(f), P] + P[\rho(f),P]P \big) \frac{1}{1 + \lambda^2 + P^2} d\lambda.
\end{equation*}

Suppose $f \in \LLip_R(M) \cap C_b^\infty(M)$. Then the integral converges in operator norm. To see this, we have to find upper bounds for the operator norms of $\frac{1+\lambda^2}{1+\lambda^2 + P^2} [\rho(f), P] \frac{1}{1+\lambda^2 + P^2}$ and $\frac{P}{1+\lambda^2 + P^2} [\rho(f), P] \frac{P}{1+\lambda^2 + P^2}$, that are integrable with respect to $\lambda$. Recall Definition~\ref{defn:symbols_on_R} of the symbol classes on $\IR$:
\[\mathcal{S}^m(\IR) := \{g \in C^\infty(\IR) \ | \ |g^{(n)}(x)| < C_l(1 + |x|)^{m-n} \text{ for all } n \in \IN_0\}.\]
Since both $\frac{1}{1+\lambda^2 + x^2} \in \mathcal{S}^{-2}(\IR)$ and $\frac{1 + \lambda^2}{1+\lambda^2 + x^2} \in \mathcal{S}^{-2}(\IR)$ (with respect to the variable $x$, i.e., for fixed $\lambda$), the operators $\frac{1}{1+\lambda^2 + P^2}$ and $\frac{1 + \lambda^2}{1+\lambda^2 + P^2}$ are operators of order $-2k$ by the first paragraph of the proof of Proposition \ref{prop:f(P)_quasilocal_of_symbol_order}. So $\frac{1+\lambda^2}{1+\lambda^2 + P^2} [\rho(f), P] \frac{1}{1+\lambda^2 + P^2}$ is an operator of order $-3k-1$ since $[\rho(f),P]$ is of order $k-1$ by Proposition~\ref{prop:PsiDOs_filtered_algebra}. So especially it is a bounded operator, and one can show that there is an integrable upper bound on the operator norm with respect to $\lambda$. The latter can be done by, e.g., using the estimates that Roe derived in his proof of his version of Proposition~\ref{prop:f(P)_quasilocal_of_symbol_order}. Analogously we can treat $\frac{P}{1+\lambda^2 + P^2} [\rho(f), P] \frac{P}{1+\lambda^2 + P^2}$ which is an operator of order $-k-1$.

Furthermore, there exists an $N > 0$ which depends only on an $\varepsilon > 0$, $R = \diam (\supp f)$ and the norms of the derivatives of $f$,\footnote{The dependence on $R$ and on the derivatives of $f$ comes from the operator norm estimate of $[\rho(f), P]$.} such that there are $\lambda_1, \ldots, \lambda_N$ and the above integral is at most $\varepsilon$ away from the sum of the integrands for $\lambda_1, \ldots, \lambda_N$.

Since both $\frac{1}{1+\lambda^2 + x^2} \in \mathcal{S}^{-2}(\IR)$ and $\frac{1 + \lambda^2}{1+\lambda^2 + x^2} \in \mathcal{S}^{-2}(\IR)$ (with respect to the variable $x$, i.e., for fixed $\lambda$), the operators $\frac{1}{1+\lambda^2 + P^2}$ and $\frac{1 + \lambda^2}{1+\lambda^2 + P^2}$ are quasilocal operators of order $-k-1$ by Proposition~\ref{prop:f(P)_quasilocal_of_symbol_order}. This also holds for their adjoints and so, by Corollary~\ref{cor:quasilocal_neg_order_uniformly_locally_compact}, they are uniformly locally compact. The same conclusion applies to the operators $\frac{P}{1+\lambda^2+P^2}$ and $\frac{(1+\lambda^2)P}{1+\lambda^2+P^2}$ which are quasilocal of order $-1$ and hence also uniformly locally compact.

So the first summand
\[\frac{1+\lambda^2}{1+\lambda^2 + P^2} [\rho(f), P] \frac{1}{1+\lambda^2 + P^2}\]
of the integrand is the difference of two compact operators and their approximability by finite rank operators depends only on $R = \diam (\supp f)$ and the Lipschitz constant $L$ of~$f$. An analogous argument applies to the second summand
\[\frac{1}{1+\lambda^2 + P^2}P [\rho(f), P] P \frac{1}{1+\lambda^2 + P^2}\]
of the integrand (note that $\frac{P^2}{1+\lambda^2+P^2}$ is a bounded operator).

So the operator $[\rho(f), \chi(P)]$ is for $f \in \LLip_R(M) \cap C_b^\infty(M)$ compact and its approximability by finite rank operators depends only on $R$, $L$ and the norms of the derivatives of $f$. That this suffices to conclude that the operator is uniformly pseudolocal is exactly Point~5 in Lemma \ref{lem:kasparov_lemma_uniform_approx_manifold}.

To conclude the proof we have to show that $\chi(P)$ is odd and multigraded. But this was already shown in full generality in \cite[Lemma 10.6.2]{higson_roe}.
\end{proof}

We have shown in the above theorem that a symmetric, elliptic uniform pseudodifferential operator naturally defines a class in uniform $K$-homology. Now we will show that this class does only depend on the principal symbol of the pseudodifferential operator. Note that ellipticity of an operator does only depend on its symbol (since it is actually defined that way, see Definition \ref{defn:elliptic_operator}, which is possible due to Lemma \ref{lem:ellipticity_independent_of_representative}), i.e., another pseudodifferential operator with the same symbol is automatically also elliptic.

\begin{prop}\label{prop:same_symbol_same_k_hom_class}
The uniform $K$-homology class of a symmetric and elliptic uniform pseudodifferential operator $P \in \UPsiDO^{k \ge 1}(E)$ does only depend on its principal symbol $\sigma(P)$, i.e., any other such operator $P^\prime$ with the same principal symbol defines the same uniform $K$-homology class.
\end{prop}

\begin{proof}
Consider in $\UPsiDO^k(E)$ the linear path $P_t := (1-t)P + t P^\prime$ of operators. They are all symmetric and, since $\sigma(P) = \sigma(P^\prime)$, they all have the same principal symbol. So they are all elliptic and therefore we get a family of uniform Fredholm modules $(H, \rho, \chi(P_t))$, where we use a fixed normalizing function $\chi$.

Now if the family $\chi(P_t)$ of bounded operators would be norm-continuous, the claim that we get the same uniform $K$-homology classes would follow directly from the relations defining uniform $K$-homology. But it seems that in general it is only possible to conclude the norm continuity of $\chi(P_t)$ if the difference $P - P^\prime$ is a bounded operator,\footnote{see, e.g., \cite[Proposition 10.3.7]{higson_roe}} i.e., if the order $k$ of $P$ is $1$ (since then the order of the difference $P - P^\prime$ would be $0$, i.e., it would define a bounded operator on $L^2(E)$); see Proposition~\ref{prop:norm_estimate_difference_func_calc}.

In the case $k > 1$ we get continuity of $\chi(P_t)$ only in the strong-$^\ast$ operator topology on $\IB(L^2(E))$. This is seen with Proposition \ref{prop:norm_estimate_difference_func_calc},\footnote{An example of a normalizing function $\chi$ fulfilling the prerequisites of Proposition \ref{prop:norm_estimate_difference_func_calc} may be found in, e.g., \cite[Exercise 10.9.3]{higson_roe}.} which implies that the family $t \mapsto \chi(P_t)$ is continuous in the norm topology of operators of degree $k-1$. Therefore, if $v \in L^2(E)$ is an element of the Sobolev space $H^{k-1}(E) \subset L^2(E)$, then $t \mapsto \chi(P_t)(v)$ is norm continuous for the $L^2$-norm. For general $v \in L^2(E)$ we do an approximation argument.

To show that $(H, \rho, \chi(P_0))$ and $(H, \rho, \chi(P_1))$ define the same uniform $K$-homology class we will use Theorem \ref{thm:weak_homotopy_equivalence_K_hom}, i.e., we will show now that the family $(H, \rho, \chi(P_t))$ is a weak homotopy.

The first bullet point of the definition of a weak homotopy is clearly satisfied since our representation $\rho$ is fixed, i.e., does not depend on the time $t$. Moreover, we have already incidentally discussed the second bullet point in the paragraph above, so it remains to varify that the third point is satisfied. We will treat here only the case $[\rho(f), \chi(P_t)]$ since the arguments for $\rho(f)(\chi(P_t)^2-1)$ are similar and the case of $\rho(f)(\chi(P_t) - \chi(P_t)^\ast)$ is clear since $\chi(P_t) - \chi(P_t)^\ast = 0$, because $P_t$ is essentially self-adjoint.

So let $\chi$ be the normalizing function $\chi(x) = \frac{x}{\sqrt{1+x^2}}$. This is the one used in the proof of the above Theorem~\ref{thm:elliptic_symmetric_PDO_defines_uniform_Fredholm_module} and we use the integral representation of $[\rho(f), \chi(P_t)]$ derived in that proof. We will only treat the second summand $\frac{1}{1+\lambda^2+P_t^2} P_t[\rho(f),P_t]P_t \frac{1}{1+\lambda^2+P_t^2}$ since it contains two more $P_t$ than the first summand (i.e., it is harder to deal with the second summand than with the first one). We have $\frac{1}{1+\lambda^2 + x^2} \in \mathcal{S}^{-2}(\IR)$ (with respect to $x$) and therefore $\psi(x) := x^{2+\varepsilon} \frac{1}{1+\lambda^2 + x^2} \in \mathcal{S}^{\varepsilon}(\IR)$. So $\psi^\prime(x) \in \mathcal{S}^{\varepsilon-1}(\IR)$ and ${\psi^\prime}{}^\prime(x) \in \mathcal{S}^{\varepsilon-2}(\IR)$, which means that both are $L^2$-integrable if $\varepsilon < 1/2$, i.e., $\psi^\prime(x) \in H^1(\IR)$. Therefore the Fourier transform of $\psi^\prime(x)$ is $L^1$-integrable. But the Fourier transform of $\psi^\prime(x)$ is $s \cdot \widehat{\psi}(s)$, i.e., $\psi$ qualifies for Proposition~\ref{prop:norm_estimate_difference_func_calc} (that $\psi^\prime(x)$ is not bounded is ok, the proposition still works in this case). So $t \mapsto \psi(P_t)$ will be continuous in $\|\largecdot\|_{0,-k+1}$-norm. By elliptic regularity this means that $t \mapsto \frac{1}{1+\lambda^2+P_t^2}$ is continuous in $\|\largecdot\|_{0,(1+\varepsilon)(k-1)}$-norm. Since $t \mapsto [\rho(f),P_t]$ is continuous in $\|\largecdot\|_{0,-k+2}$-norm, we conclude that the whole second summand is continuous in $\|\largecdot\|_{0,2\varepsilon(k-1)-k+2}$-norm. If $\varepsilon = 1/2$, then $2\varepsilon(k-1)-k+2 = 1$. Since we have to choose $\varepsilon < 1/2$, we choose it just beneath $1/2$, i.e., so that $2\varepsilon(k-1)-k+2 > 0$. It then follows that $[\rho(f), \chi(P_t)]$ is continuous in operator norm, which concludes this proof.
\end{proof}

\section{Uniform \texorpdfstring{$K$}{K}-theory}
\label{sec:uniform_k_th}

In this section we will define uniform $K$-theory and show that for \spinc manifolds it is \Poincare dual to uniform $K$-homology. The definition of uniform $K$-theory is based on the following observation: we want that it consists of vector bundles such that Dirac operators over manifolds of bounded geometry may be twisted with them (since we want a cap product between uniform $K$-homology and uniform $K$-theory). Hence we have to consider vector bundles of bounded geometry, because otherwise the twisted operator will not be uniform.

The first guess is to use the algebra $C_b^\infty(M)$ of smooth functions on $M$ whose derivatives are all uniformly bounded, and then to consider its operator $K$-theory. This guess is based on the speculation that the boundedness of the derivatives translates into the boundedness of the Christoffel symbols if one equips the vector bundle with the induced metric and connection coming from the given embedding of the bundle into $\IC^k$ (this embedding is given to us because a projection matrix with entries in $C_b^\infty(M)$ defines a subbundle of $\IC^k$ by considering the image of the projection matrix). To our luck this first guess works out.

Note that other authors have, of course, investigated similar versions of $K$-theory: Kaad investigated in \cite{kaad} Hilbert bundles of bounded geometry over manifolds of bounded geometry (the author thanks Magnus Goffeng for pointing to that publication). Dropping the condition that the bundles must have bounded geometry, there is a general result by Morye contained in \cite{morye} having as a corollary the Serre--Swan theorem for smooth vector bundles over (possibly non-compact) smooth manifolds. If one is only interested in the last mentioned result, there is also the short note \cite{sardanashvily} by Sardanashvily.

\subsection{Definition and basic properties of uniform \texorpdfstring{$K$}{K}-theory}

As we have written above, we will define uniform $K$-theory of a manifold of bounded geometry as the operator $K$-theory of $C_b^\infty(M)$. But since $C_b^\infty(M)$ turns out to be a local $C^\ast$-algebra (see Lemma \ref{lem:C_b_infty_local}), its operator $K$-theory will coincide with the $K$-theory of its closure which is the $C^\ast$-algebra $C_u(M)$ of all bounded, uniformly continuous functions on $M$ (see Lemma \ref{lem:norm_completion_C_b_infty}). Therefore we may define uniform $K$-theory for any metric space $X$ as the operator $K$-theory of $C_u(X)$.

\begin{defn}[Uniform $K$-theory]
Let $X$ be a metric space. The \emph{uniform $K$-theory groups of $X$} are defined as
\[K^p_u(X) := K_{-p}(C_u(X)),\]
where $C_u(X)$ is the $C^\ast$-algebra of bounded, uniformly continuous functions on $X$.
\qed
\end{defn}

The introduction of the minus sign in the index $-p$ in the above definition is just a convention which ensures that the indices in formulas, like the one for the cap product between uniform $K$-theory and uniform $K$-homology, coincide with the indices from the corresponding formulas for (co-)homology. Since complex $K$-theory is $2$-periodic, the minus sign does not change anything in the formulas.

Denoting by $\overline{X}$ the completion of the metric space $X$, we have $K^\ast_u(\overline{X}) = K^\ast_u(X)$ because every uniformly continuous function on $X$ has a unique extension to $\overline{X}$, i.e., $C_u(\overline{X}) = C_u(X)$. This means that, e.g., the uniform $K$-theories of the spaces $[0,1]$, $[0,1)$ and $(0,1)$ are all equal. Furthermore, since on a compact space $X$ we have $C_u(X) = C(X)$, uniform $K$-theory coincides for compact spaces with usual $K$-theory. Let us state this as a small lemma:

\begin{lem}
If $X$ is totally bounded, then $K^\ast_u(X) = K_u^\ast(\overline{X}) = K^\ast(\overline{X})$.
\end{lem}

\begin{rem}
Note the following difference between uniform $K$-theory and uniform $K$-homology: whereas uniform $K$-theory of $X$ coincides with the uniform $K$-theory of the completion $\overline{X}$, this is in general not true for uniform $K$-homology.

Recall that in Proposition \ref{prop:compact_space_every_module_uniform} we have shown that if $X$ is totally bounded, then the uniform $K$-homology of $X$ coincides with locally finite $K$-homology of $X$. So for, e.g., an open ball $B$ in $\IR^n$ uniform and locally finite $K$-homology coincide and hence $K_m^u(B) \cong \IZ$ for $m=n$ and it vanishes for all other values of $m$. But due to homotopy invariance we have $K_m^u(\overline{B}) \cong K_m^u(*) \cong \IZ$ for $m=0$ and it vanishes for other values of $m$.

In the case of uniform $K$-theory we have $K^m_u(B) \cong K^m_u(\overline{B}) \cong K^m_u(*) \cong \IZ$ for $m=0$ and it vanishes otherwise.
\qed
\end{rem}

Recall that in Lemma \ref{lem:uniform_k_hom_discrete_space} we have shown that the uniform $K$-homology group $K_0^u(Y)$ of a uniformly discrete, proper metric space $Y$ of coarsely bounded geometry is isomorphic to the group $\ell^\infty_\IZ(Y)$ of all bounded, integer-valued sequences indexed by $Y$, and that $K_1^u(Y) = 0$. Since we want uniform $K$-theory to be \Poincare dual to uniform $K$-homology, we need the corresponding result for uniform $K$-theory.

\begin{lem}\label{lem:uniform_k_th_discrete_space}
Let $Y$ be a uniformly discrete metric space. Then $K^0_u(Y)$ is isomorphic to $\ell^\infty_\IZ(Y)$ and $K^1_u(Y) = 0$.
\end{lem}

The proof is an easy consequence of the fact that $C_u(Y) \cong \prod_{y \in Y} C(y) \cong \prod_{y \in Y} \IC$ for a uniformly discrete space $Y$, where the direct product of $C^\ast$-algebras is equipped with the pointwise algebraic operations and the sup-norm. The computation of the operator $K$-theory of $\prod_{y \in Y} \IC$ is now easily done (cf. \cite[Exercise 7.7.3]{higson_roe}).

And last, we will give a relation of uniform $K$-theory with amenability. Note that an analogous relation for bounded de Rham cohomology is already well-known, and also for other, similar (co-)homology theories (see, e.g., \cite[Section 8]{block_weinberger_large_scale}).

\begin{lem}
Let $M$ be a metric space with amenable fundamental group.

We let $X$ be the universal cover of $M$ and we denote the covering projection by $\pi\colon X \to M$. Then the pull-back map $K^\ast_u(M) \to K^\ast_u(X)$ is injective.
\end{lem}

\begin{proof}
The projection $\pi$ induces a map $\pi^\ast \colon C_u(M) \to C_u(X)$ which then induces the pull-back map $K^\ast_u(M) \to K^\ast_u(X)$. We will prove the lemma by constructing a left inverse to the above map $\pi^\ast$, i.e., we will construct a map $p \colon C_u(X) \to C_u(M)$ with $p \circ \pi^\ast = \id \colon C_u(M) \to C_u(M)$.

Let $F \subset X$ be a fundamental domain for the action of the deck transformation group on $X$. Since $\pi_1(M)$ is amenable, we choose a \Folner sequence $(E_i)_i \subset \pi_1(M)$ in it. Now given a function $f \in C_u(X)$, we set
\[f_i(y) := \frac{1}{\card E_i} \sum_{x \in \pi^{-1}(y) \cap E_i \cdot F} f(x)\]
for $y \in M$. This gives us a sequence of functions $f_i$ on $M$, but they are in general not even continuous.

Now choosing a functional $\tau \in (\ell^\infty)^\ast$ associated to a free ultrafilter on $\IN$, we define $p(f)(y) := \tau(f_i(y))$. Due to the \Folner condition on $(E_i)_i$ all discontinuities that the functions $f_i$ may have vanish in the limit under $\tau$, and we get a bounded, uniformly continuous function $p(f)$ on $M$.

It is clear that $p$ is a left inverse to $\pi^\ast$.
\end{proof}

\subsection{Interpretation via vector bundles}\label{sec:interpretation_uniform_k_theory}

We will show now that if $M$ is a manifold of bounded geometry then we have a description of the uniform $K$-theory of $M$ via vector bundles of bounded geometry.

To show this, we first need to show that the operator $K$-theory of $C_u(M)$ coincides with the operator $K$-theory of $C_b^\infty(M)$. This is established via the following two lemmas.

\begin{lem}\label{lem:C_b_infty_local}
Let $M$ be a manifold of bounded geometry.

Then $C_b^\infty(M)$ is a local $C^\ast$-algebra\footnote{That is to say, it and all matrix algebras over it are closed under holomorphic functional calculus and its completion is a $C^\ast$-algebra.}.
\end{lem}

\begin{proof}
Since $C_b^\infty(M)$ is a $^\ast$-subalgebra of the $C^\ast$-algebra $C_b(M)$ of bounded continuous functions on $M$, then norm completion of $C_b^\infty(M)$, i.e., its closure in $C_b(M)$, is surely a $C^\ast$-algebra.

So we have to show that $C_b^\infty(M)$ and all matrix algebras over it are closed under holomorphic functional calculus. Since $C_b^\infty(M)$ is naturally a \Frechet algebra with a \Frechet topology which is finer than the sup-norm topology, by \cite[Corollary 2.3]{schweitzer}\footnote{The corollary states that under the condition that the topology of a \Frechet algebra $A$ is finer than the sup-norm topology we may conclude that if $A$ is closed under holomorphic functional calculus, then this holds also for all matrix algebras over $A$.} it remains to show that $C_b^\infty(M)$ itself is closed under holomorphic functional calculus.

But that $C_b^\infty(M)$ is closed under holomorphic functional calculus is easily seen using \cite[Lemma 1.2]{schweitzer}, which states that a unital \Frechet algebra $A$ with a topology finer than the sup-norm topology is closed under functional calculus if and only if the inverse $a^{-1} \in \overline{A}$ of any invertible element $a \in A$ actually lies in $A$.
\end{proof}

\begin{lem}\label{lem:norm_completion_C_b_infty}
Let $M$ be a manifold of bounded geometry.

Then the sup-norm completion of $C_b^\infty(M)$ is the $C^\ast$-algebra $C_u(M)$ of bounded, uniformly continuous functions on $M$.
\end{lem}

\begin{proof}
We surely have $\overline{C_b^\infty(M)} \subset C_u(M)$. To show the converse inclusion, we have to approximate a bounded, uniformly continuous function by a smooth one with bounded derivatives. This can be done by choosing a nice cover of $M$ with subordinate partitions of unity via Lemma \ref{lem:nice_coverings_partitions_of_unity} and then apply in every coordinate chart the same mollifier to the uniformly continous function.

Let us elaborate a bit more on the last sentence of the above paragraph: after choosing the nice cover and cutting a function $f \in C_u(M)$ with the subordinate partition of unity $\{\varphi_i\}$, we have transported the problem to Euclidean space $\IR^n$ and our family of functions $\varphi_i f$ is uniformly equicontinuous (this is due to the uniform continuity of $f$ and will be crucially important at the end of this proof). Now let $\psi$ be a mollifier on $\IR^n$, i.e., a smooth function with $\psi \ge 0$, $\supp \psi \subset B_1(0)$, $\int_{\IR^n} \psi d\lambda = 1$ and $\psi_\varepsilon := \varepsilon^{-n} \psi(\largecdot / \varepsilon) \stackrel{\varepsilon \to 0}\longrightarrow \delta_0$. Since convolution satisfies $D^\alpha (\varphi_i f \ast \psi_\varepsilon) = \varphi_i f \ast D^\alpha \psi_\varepsilon$, where $D^\alpha$ is a directional derivative on $\IR^n$ in the directions of the multi-index $\alpha$ and of order $|\alpha|$, we conclude that every mollified function $\varphi_i f \ast \psi_\varepsilon$ is smooth with bounded derivatives. Furthermore, we know $\| \varphi_i f \ast D^\alpha \psi_\epsilon \|_\infty \le \| \varphi_i f \|_\infty \cdot \| D^\alpha \psi_\varepsilon \|_1$ from which we conclude that the bounds on the derivatives of $\varphi_i f \ast \psi_\varepsilon$ are uniform in $i$, i.e., if we glue the functions $\varphi_i f \ast \psi_\epsilon$ together to a function on the manifold $M$ (note that the functions $\varphi_i f \ast \psi_\epsilon$ are supported in our chosen nice cover since convolution with $\psi_\varepsilon$ enlarges the support at most by $\varepsilon$), we get a function $f_\varepsilon \in C_b^\infty(M)$. It remains to show that $f_\varepsilon$ converges to $f$ in sup-norm, which is equivalent to the statement that $\varphi_i f \ast \psi_\varepsilon$ converges to $\varphi_i f$ in sup-norm and uniformly in the index $i$. But we know that
\[ \big| (\varphi_i f \ast \psi_\epsilon) (x) - (\varphi_i f) (x) \big| \le \sup_{\substack{x \in \supp \varphi_i f\\y \in B_\varepsilon(0)}} \big| (\varphi_i f) (x - y) - (\varphi_i f) (x) \big| \]
from which the claim follows since the family of functions $\varphi_i f$ is uniformly equicontinuous (recall that this followed from the uniform continuity of $f$ and this here is actually the only point in this proof where we need that property of $f$).
\end{proof}

Since $C_b^\infty(M)$ is an $m$-convex \Frechet algebra\footnote{That is to say, a \Frechet algebra such that its topology is given by a countable family of submultiplicative seminorms.}, we can also use the $K$-theory for $m$-convex \Frechet algebras as developed by Phillips in \cite{phillips} to define the $K$-theory groups of $C_b^\infty(M)$. But this produces the same groups as operator $K$-theory, since $C_b^\infty(M)$ is an $m$-convex \Frechet algebra with a finer topology than the norm topology and therefore its $K$-theory for $m$-convex \Frechet algebras coincides with its operator $K$-theory by \cite[Corollary 7.9]{phillips}.

We summarize this observations in the following lemma:

\begin{lem}\label{lem:equivalent_defns_uniform_k_theory}
Let $M$ be a manifold of bounded geometry.

Then the operator $K$-theory of $C_u(M)$, the operator $K$-theory of $C_b^\infty(M)$ and Phillips $K$-theory for $m$-convex \Frechet algebras of $C_b^\infty(M)$ are all pairwise naturally isomorphic.
\end{lem}

So we have shown $K^\ast_u(M) \cong K_{-\ast}(C_b^\infty(M))$. In order to conclude the description via vector bundles of bounded geometry, we will need to establish the correspondence between vector bundles of bounded geometry and idempotent matrices with entries in $C_b^\infty(M)$. This will be done in the next subsections.

\subsubsection*{Isomorphism classes and complements}

Let $M$ be a manifold of bounded geometry and $E$ and $F$ two complex vector bundles equipped with Hermitian metrics and compatible connections.

\begin{defn}[$C^\infty$-boundedness / $C_b^\infty$-isomorphy of vector bundle homomorphisms]\label{defn:C_infty_bounded}
We will call a vector bundle homomorphism $\varphi\colon E \to F$ \emph{$C^\infty$-bounded}, if with respect to synchronous framings of $E$ and $F$ the matrix entries of $\varphi$ are bounded, as are all their derivatives, and these bounds do not depend on the chosen base points for the framings or the synchronous framings themself.

$E$ and $F$ will be called \emph{$C_b^\infty$-isomorphic}, if there exists an isomorphism $\varphi\colon E \to F$ such that both $\varphi$ and $\varphi^{-1}$ are $C^\infty$-bounded. In that case we will call the map $\varphi$ a $C_b^\infty$-isomorphism. Often we will write $E \cong F$ when no confusion can arise with mistaking it with algebraic isomorphy.
\qed
\end{defn}

Using the characterization of bounded geometry via the matrix transition functions from Lemma \ref{lem:equiv_characterizations_bounded_geom_bundles}, we immediately see that if $E$ and $F$ are $C_b^\infty$-isomorphic, than $E$ is of bounded geometry if and only if $F$ is.

It is clear that $C_b^\infty$-isomorphy is compatible with direct sums and tensor products, i.e., if $E \cong E^\prime$ and $F \cong F^\prime$ then $E \oplus F \cong E^\prime \oplus F^\prime$ and $E \otimes F \cong E^\prime \otimes F^\prime$.

We will now give a useful global characterization of $C_b^\infty$-isomorphisms if the vector bundles have bounded geometry:

\begin{lem}\label{lem:C_b_infty_Iso_equivalent}
Let $E$ and $F$ have bounded geometry and let $\varphi\colon E \to F$ be an isomorphism. Then $\varphi$ is a $C_b^\infty$-isomorphism if and only if
\begin{itemize}
\item $\varphi$ and $\varphi^{-1}$ are bounded, i.e., $\|\varphi(v)\| \le C \cdot \|v\|$ for all $v \in E$ and a fixed $C > 0$ and analogously for $\varphi^{-1}$, and
\item $\nabla^E - \varphi^\ast \nabla^F$ is bounded and also all its covariant derivatives.
\end{itemize}
\end{lem}

\begin{proof}
For a point $p \in M$ let $B \subset M$ be a geodesic ball centered at $p$, $\{ x_i \}$ the corresponding  normal coordinates of $B$, and let $\{ E_\alpha(y) \}$, $y \in B$, be a framing for $E$. Then we may write every vector field $X$ on $B$ as $X = X^i \frac{\partial}{\partial x_i} = (X^1 , \ldots, X^n)^T$ and every section $e$ of $E$ as $e = e^\alpha E_\alpha = (e^1, \ldots, e^k)^T$, where we assume the Einstein summation convention and where $\largecdot^T$ stands for the transpose of the vector (i.e., the vectors are actually column vectors). Furthermore, after also choosing a framing for $F$, $\varphi$ becomes a matrix for every $y \in B$ and $\varphi(e)$ is then just the matrix multiplication $\varphi(e) = \varphi \cdot e$. Finally, $\nabla^E_X e$ is locally given by
\[\nabla^E_X e = X(e) + \Gamma^E(X)\cdot e,\]
where $X(e)$ is the column vector that we get after taking the derivative of every entry $e^j$ of $e$ in the direction of $X$ and $\Gamma^E$ is a matrix of $1$-forms (i.e., $\Gamma^E(X)$ is then a usual matrix that we multiply with the vector $e$). The entries of $\Gamma^E$ are called the connection $1$-forms.

Since $\varphi$ is an isomorphism, the pull-back connection $\varphi^\ast \nabla^F$ is given by\footnote{Note that $\varphi$ is a morphism of vector bundles, i.e., the following diagram commutes:
\[\xymatrix{E \ar[rr]^\varphi \ar[dr] & & F \ar[dl]\\ & M &}\]
This means that $\varphi$ descends to the identity on $M$, i.e., in Equation \eqref{eq:vect_iso} the vector field $X$ occurs on both the left and the right hand side (since actually we have $(\varphi^{-1})^\ast X$ on the right hand side).}
\begin{equation}
\label{eq:vect_iso}
(\varphi^\ast \nabla^F)_X e = \varphi^\ast (\nabla^F_X (\varphi^{-1})^\ast e),
\end{equation}
so that locally we get
\[(\varphi^\ast \nabla^F)_X e = \varphi^{-1}\cdot \big( X(\varphi \cdot e) + \Gamma^F(X) \cdot \varphi \cdot e\big).\]
Using the product rule we may rewrite $X(\varphi \cdot e) = X(\varphi) \cdot e + \varphi \cdot X(e)$, where $X(\varphi)$ is the application of $X$ to every entry of $\varphi$. So at the end we get for the difference $\nabla^E - \varphi^\ast \nabla^F$ in local coordinates and with respect to framings of $E$ and $F$
\begin{equation}\label{eq:difference_connections_local}
(\nabla^E - \varphi^\ast \nabla^F)_X e = \Gamma^E(X) \cdot e - \varphi^{-1} \cdot X(\varphi) \cdot e - \varphi^{-1} \cdot \Gamma^F(X) \cdot \varphi \cdot e.
\end{equation}

Since $E$ and $F$ have bounded geometry, by Lemma \ref{lem:equiv_characterizations_bounded_geom_bundles} the Christoffel symbols of them with respect to synchronous framings are bounded and also all their derivatives, and these bounds are independent of the point $p \in M$ around that we choose the normal coordinates and the framings. Assuming that $\varphi$ is a $C_b^\infty$-isomorphism, the same holds for the matrix entries of $\varphi$ and $\varphi^{-1}$ and we conclude with the above Equation \eqref{eq:difference_connections_local} that the difference $\nabla^E - \varphi^\ast \nabla^F$ is bounded and also all its covariant derivatives (here we also need to consult the local formula for covariant derivatives of tensor fields).

Conversely, assume that $\varphi$ and $\varphi^{-1}$ are bounded and that the difference $\nabla^E - \varphi^\ast \nabla^F$ is bounded and also all its covariant derivatives. If we denote by $\Gamma^{\text{diff}}$ the matrix of $1$-forms given by
\[\Gamma^{\text{diff}}(X) = \Gamma^E(X) - \varphi^{-1} \cdot X(\varphi) - \varphi^{-1} \cdot \Gamma^F(X) \cdot \varphi,\]
we get from Equation \eqref{eq:difference_connections_local}
\[X(\varphi) = \varphi \cdot (\Gamma^E(X) - \Gamma^{\text{diff}}(X)) - \Gamma^F(X) \cdot \varphi.\]
Since we assumed that $\varphi$ is bounded, its matrix entries must be bounded. From the above equation we then conclude that also the first derivatives of these matrix entries are bounded. But now that we know that the entries and also their first derivatives are bounded, we can differentiate the above equation once more to conclude that also the second derivatives of the matrix entries of $\varphi$ are bounded, on so on. This shows that $\varphi$ is $C^\infty$-bounded. At last, it remains to see that the matrix entries of $\varphi^{-1}$ and also all their derivatives are bounded. But since locally $\varphi^{-1}$ is the inverse matrix of $\varphi$, we just have to use Cramer's rule.
\end{proof}

An important property of vector bundles over compact spaces is that they are always complemented, i.e., for every bundle $E$ there is a bundle $F$ such that $E \oplus F$ is isomorphic to the trivial bundle. Note that this fails in general for non-compact spaces. So our important task is now to show that we have an analogous proposition for vector bundles of bounded geometry, i.e., that they are always complemented (in a suitable way).

\begin{defn}[$C_b^\infty$-complemented vector bundles]
A vector bundle $E$ will be called \emph{$C_b^\infty$-complemented}, if there is some vector bundle $E^\perp$ such that $E \oplus E^\perp$ is $C_b^\infty$-isomorphic to a trivial bundle with the flat connection.
\qed
\end{defn}

Since a bundle with a flat connection is trivially of bounded geometry, we get that $E \oplus E^\perp$ is of bounded geometry. And since a direct sum $E \oplus E^\perp$ of vector bundles is of bounded geometry if and only if both vector bundles $E$ and $E^\perp$ are of bounded geometry, we conclude that if $E$ is $C_b^\infty$-complemented, then both $E$ and its complement $E^\perp$ are of bounded geometry. It is also clear that if $E$ is $C_b^\infty$-complemented and $F \cong E$, then $F$ is also $C_b^\infty$-complemented.

We will now prove the crucial fact that every vector bundle of bounded geometry is $C_b^\infty$-complemented. The proof is just the usual one for vector bundles over compact Hausdorff spaces, but we additionally have to take care of the needed uniform estimates. As a source for this usual proof the author used \cite[Proposition 1.4]{hatcher_VB}. But first we will need a technical lemma.

\begin{lem}\label{lem:coloring_graph}
Let a covering $\{U_\alpha\}$ of $M$ with finite multiplicity be given. Then there exists a coloring of the subsets $U_\alpha$ with finitely many colors such that no two intersecting subsets have the same color.
\end{lem}

\begin{proof}
Construct a graph whose vertices are the subsets $U_\alpha$ and two vertices are connected by an edge if the corresponding subsets intersect. We have to find a coloring of this graph with only finitely many colors where connected vertices do have different colors.

To do this, we firstly use the theorem of de Bruijin--Erd\"{o}s stating that an infinite graph may be colored by $k$ colors if and only if every of its finite subgraphs may be colored by $k$ colors (one can use the Lemma of Zorn to prove this).

Secondly, since the covering has finite multiplicity it follows that the number of edges attached to each vertex in our graph is uniformly bounded from above, i.e., the maximum vertex degree of our graph is finite. But this also holds for every subgraph of our graph, with the maximum vertex degree possibly only decreasing by passing to a subgraph. Now a simple greedy algorithm shows that every finite graph may be colored with one more color than its maximum vertex degree: just start by coloring a vertex with some color, go to the next vertex and use an admissible color for it, and so on.
\end{proof}

\begin{prop}\label{prop:every_bundle_complemented}
Let $M$ be a manifold of bounded geometry and let $E \to M$ be a vector bundle of bounded geometry.

Then $E$ is $C_b^\infty$-complemented.
\end{prop}

\begin{proof}
Since $M$ and $E$ have bounded geometry, we can find a uniformly locally finite cover of $M$ by normal coordinate balls of a fixed radius together with a subordinate partition of unity whose derivatives are all uniformly bounded and such that over each coordinate ball $E$ is trivialized via a synchronous framing. This follows basically from Lemma \ref{lem:nice_coverings_partitions_of_unity}.

Now we the above Lemma \ref{lem:coloring_graph} to color the coordinate balls with finitely many colors so that no two balls with the same color do intersect. This gives a partition of the coordinate balls into $N$ families $U_1, \ldots, U_N$ such that every $U_i$ is a collection of disjoint balls, and we get a corresponding subordinate partition of unity $1 = \varphi_1 + \ldots + \varphi_N$ with uniformly bounded derivatives (each $\varphi_i$ is the sum of all the partition of unity functions of the coordinate balls of $U_i$). Furthermore, $E$ is trivial over each $U_i$ and we denote these trivializations coming from the synchronous framings by $h_i \colon p^{-1}(U_i) \to U_i \times \IC^k$, where $p\colon E \to M$ is the projection.

Now we set
\[g_i\colon E \to \IC^k, \ g_i(v) := \varphi_i(p(v)) \cdot \pi_i (h_i (v)),\]
where $\pi_i \colon U_i \times \IC^k \to \IC^k$ is the projection. Each $g_i$ is a linear injection on each fiber over $\varphi_i^{-1}(0,1]$ and so, if we define
\[g \colon E \to \IC^{Nk}, \ g(v) := (g_1(v), \ldots, g_N(v)),\]
we get a map $g$ that is a linear injection on each fiber of $E$. Finally, we define a map
\[G\colon E \to M \times \IC^{Nk}, \ G(v) := (p(v), g(v)).\]
This establishes $E$ as a subbundle of a trivial bundle.

If we equip $M \times \IC^{Nk}$ with a constant metric and the flat connection, we get that the induced metric and connection on $E$ is $C_b^\infty$-isomorphic to the original metric and connection on $E$ (this is due to our choice of $G$). Now let us denote by $e$ the projection matrix of the trivial bundle $\IC^{Nk}$ onto the subbundle $G(E)$ of it, i.e., $e$ is an $Nk \times Nk$-matrix with functions on $M$ as entries and $\image e = E$. Now, again due to our choice of $G$, we can conclude that these entries of $e$ are bounded functions with all derivatives of them also bounded, i.e., $e \in \Idem_{Nk \times Nk}(C_b^\infty(M))$. Now the claim follows with the Proposition \ref{prop:image_proj_matrix_complemented} which establishes the orthogonal complement $E^\perp$ of $E$ in $\IC^{Nk}$ with the induced metric and connection as a $C_b^\infty$-complement to $E$.
\end{proof}

We have seen in the above proposition that every vector bundle of bounded geometry is $C_b^\infty$-complemented. Now if we have a manifold of bounded geometry $M$, then its tangent bundle $TM$ is of bounded geometry and so we know that it is $C_b^\infty$-complemented (although $TM$ is real and not a complex bundle, the above proof of course also holds for real vector bundles). But in this case we usually want the complement bundle to be given by the normal bundle $NM$ coming from an embedding $M \hookrightarrow \IR^N$. We will prove this now under the assumption that the embedding of $M$ into $\IR^N$ is ``nice'':\footnote{See \cite{MO_iso_embedding_bounded_second_form} for a discussion of existence of ``nice'' embeddings.}

\begin{cor}\label{cor:tangent_bundle_complemented}
Let $M$ be a manifold of bounded geometry and let it be isometrically embedded into $\IR^N$ such that the second fundamental form is $C^\infty$-bounded.

Then its tangent bundle $TM$ is $C_b^\infty$-complemented by the normal bundle $NM$ corresponding to this embedding $M \hookrightarrow \IR^N$, equipped with the induced metric and connection.
\end{cor}

\begin{proof}
Let $M$ be isometrically embedded in $\IR^N$. Then its tangent bundle $TM$ is a subbundle of $T\IR^N$ and we denote the projection onto it by $\pi\colon T\IR^N \to TM$. Because of Point 1 of the following Proposition \ref{prop:image_proj_matrix_complemented} it suffices to show that the entries of $\pi$ are $C^\infty$-bounded functions.

Let $\{v_i\}$ be the standard basis of $\IR^N$ and let $\{E_\alpha(y)\}$ be the orthonormal frame of $TM$ arising out of normal coordinates $\{\partial_k\}$ of $M$ via the Gram-Schmidt process. Then the entries of the projection matrix $\pi$ with respect to the basis $\{v_i\}$ are given by
\[\pi_{ij}(y) = \sum_\alpha \langle E_\alpha(y), v_j\rangle \langle E_\alpha(y), v_i\rangle.\]

Let $\widetilde{\nabla}$ denote the flat connection on $\IR^N$. Since $\widetilde{\nabla}_{\partial_k} v_i = 0$ we get
\[\partial_k \pi_{ij} (y) = \sum_\alpha \langle \widetilde{\nabla}_{\partial_k} E_\alpha (y), v_j\rangle \langle E_\alpha(y), v_i\rangle + \langle E_\alpha(y), v_j\rangle \langle \widetilde{\nabla}_{\partial_k} E_\alpha (y), v_i\rangle.\]
Now if we denote by $\nabla^M$ the connection on $M$, we get
\[\widetilde{\nabla}_{\partial_k} E_\alpha(y) = \nabla^{M}_{\partial_k} E_\alpha(y) + \operatorname{I\!\!\;I}(\partial_k, E_\alpha),\]
where $\operatorname{I\!\!\;I}$ is the second fundamental form. So to show that $\pi_{ij}$ is $C^\infty$-bounded, we must show that $E_\alpha(y)$ are $C^\infty$-bounded sections of $TM$ (since by assumption the second fundamental form is a $C^\infty$-bounded tensor field). But that these $E_\alpha(y)$ are $C^\infty$-bounded sections of $TM$ follows from their construction (i.e., applying Gram-Schmidt to the normal coordinate fields $\partial_k$) and because $M$ has bounded geometry.
\end{proof}

\subsubsection*{Interpretation of \texorpdfstring{$K^0_u(M)$}{even uniform K(M)}}

Recall for the understanding of the following proposition that if a vector bundle is $C_b^\infty$-complemented, then it is of bounded geometry. Furthermore, this proposition is the crucial one that gives us the description of uniform $K$-theory via vector bundles of bounded geometry.

\begin{prop}\label{prop:image_proj_matrix_complemented}
Let $M$ be a manifold of bounded geometry.
\begin{enumerate}
\item Let $e \in \Idem_{N \times N}(C_b^\infty(M))$ be an idempotent matrix.

Then the vector bundle $E := \image e$, equipped with the induced metric and connection, is $C_b^\infty$-complemented.

\item Let $E$ be a $C_b^\infty$-complemented vector bundle, i.e., there is a vector bundle $E^\perp$ such that $E \oplus E^\perp$ is $C_b^\infty$-isomorphic to the trivial $N$-dimensional bundle $\IC^N \to M$.

Then all entries of the projection matrix $e$ onto the subspace $E \oplus 0 \subset \IC^N$ with respect to a global synchronous framing of $\IC^N$ are $C^\infty$-bounded, i.e., we have $e \in \Idem_{N \times N}(C_b^\infty(M))$.
\end{enumerate}
\end{prop}

\begin{proof}[Proof of point 1]
We denote by $E$ the vector bundle $E := \image e$ and by $E^\perp$ its complement $E^\perp := \image (1-e)$ and equip them with the induced metric and connection. So we have to show that $E \oplus E^\perp$ is $C_b^\infty$-isomorphic to the trivial bundle $\IC^N \to M$.

Let $\varphi\colon E \oplus E^\perp \to \IC^N$ be the canonical algebraic isomorphism $\varphi(v,w) := v + w$. We have to show that both $\varphi$ and $\varphi^{-1}$ are $C^\infty$-bounded. 

Let $p \in M$. Let $\{ E_\alpha \}$ be an orthonormal basis of the vector space $E_p$ and $\{ E^\perp_\beta \}$ an orthonormal basis of $E^\perp_p$. Then the set $\{ E_\alpha, E^\perp_\beta \}$ is an orthonormal basis for $\IC_p^N$. We extend $\{E_\alpha\}$ to a synchronous framing $\{E_\alpha(y)\}$ of $E$ and $\{E^\perp_\beta\}$ to a synchronous framing $\{E^\perp_\beta(y)\}$ of $E^\perp$. Since $\IC^N$ is equipped with the flat connection, the set $\{ E_\alpha, E^\perp_\beta \}$ forms a synchronous framing for $\IC^N$ at all points of the normal coordinate chart. Then $\varphi(y)$ is the change-of-basis matrix from the basis $\{E_\alpha(y), E_\beta^\perp(y)\}$ to the basis $\{ E_\alpha, E^\perp_\beta \}$ and vice versa for $\varphi^{-1}(y)$; see Figure \ref{fig:frames}:

\begin{figure}[htbp]
\centering
\includegraphics[scale=0.7]{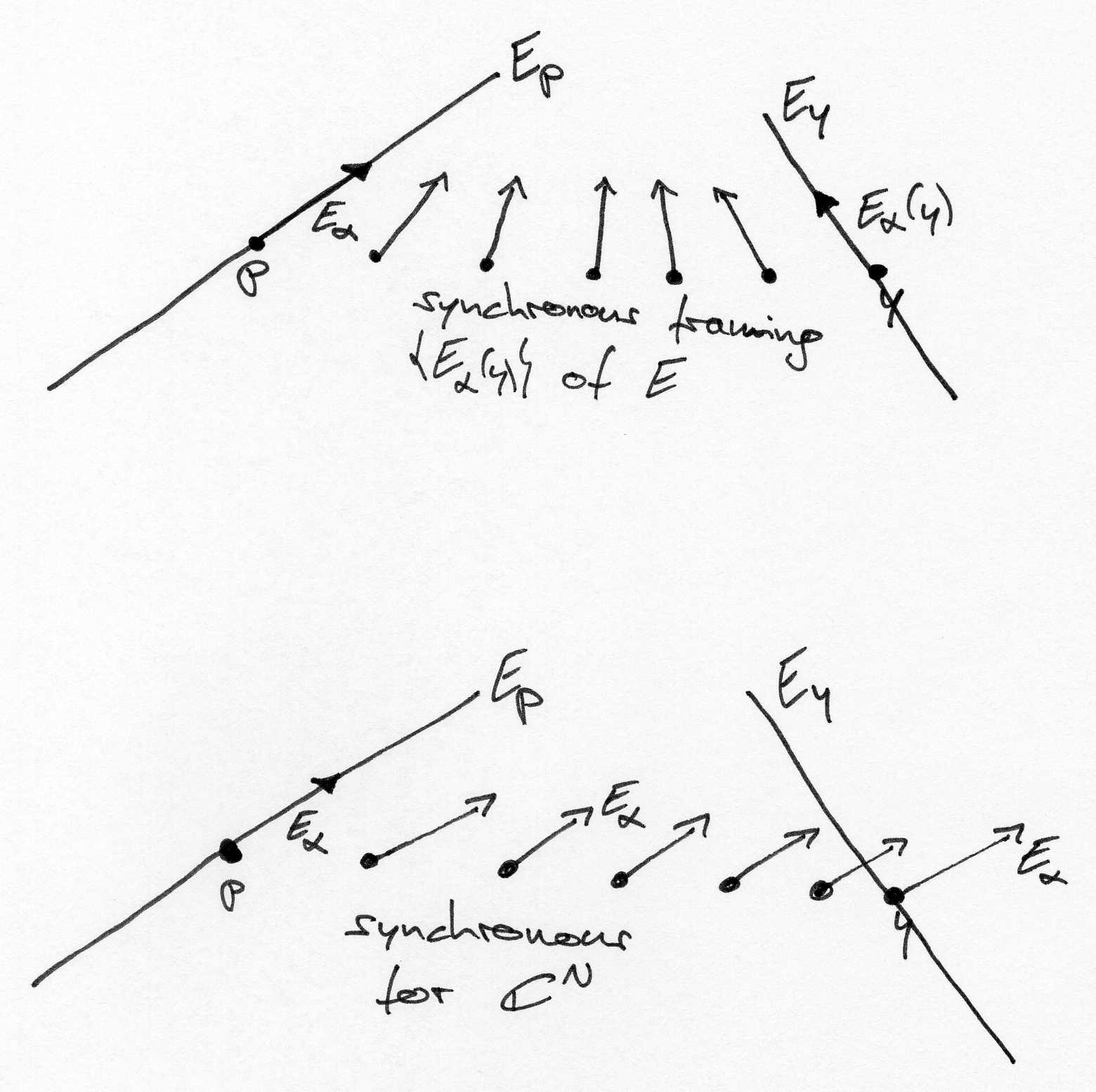}
\caption{The frames $\{E_\alpha(y)\}$ and $\{E_\alpha\}$.}
\label{fig:frames}
\end{figure}

We have $e(p)(E_\alpha) = E_\alpha$. Since the entries of $e$ are $C^\infty$-bounded and the rank of a matrix is a lower semi-continuous function of the entries, there is some geodesic ball $B$ around $p$ such that $\{ e(y)(E_\alpha) \}$ forms a basis of $E_y$ for all $y \in B$ and the diameter of the ball $B$ is bounded from below independently of $p \in M$. We denote by $\Gamma_{i \nu}^\mu(y)$ the Christoffel symbols of $E$ with respect to the frame $\{ e(y)(E_\alpha) \}$. Let $\gamma(t)$ be a radial geodesic in $M$ with $\gamma(0) = p$. If we now let $E_\alpha(\gamma(t))^\mu$ denote the $\mu$th entry of the vector $E_\alpha(\gamma(t))$ represented in the basis $\{ e(\gamma(t))(E_\alpha) \}$, then (since it is a synchronous frame) it satisfies the ODE
\[ \tfrac{d}{dt} E_\alpha(\gamma(t))^\mu = -\sum_{i,\nu} E_\alpha(\gamma(t))^\nu \cdot \tfrac{d}{dt}{\gamma_i}(t) \cdot \Gamma_{i \nu}^\mu(\gamma(t)),\]
where $\{\gamma_i\}$ is the coordinate representation of $\gamma$ in normal coordinates $\{x_i\}$. Since $\gamma$ is a radial geodesic, its representation in normal coordinates is $\gamma_i(t) = t \cdot \gamma_i(0)$ and so the above formula simplifies to
\begin{equation}\label{eq:ODE_synchronous_frame}
\tfrac{d}{dt} E_\alpha(\gamma(t))^\mu = -\sum_{i,\nu} E_\alpha(\gamma(t))^\nu \cdot \gamma_i(0) \cdot \Gamma_{i \nu}^\mu(\gamma(t)).
\end{equation}

Since $\Gamma_{i \nu}^\mu(y)$ are the Christoffel symbols with respect to the frame $\{ e(y)(E_\alpha) \}$, we get the equation
\begin{equation}\label{eq:christoffel_symbols_representation}
\sum_{\mu} \Gamma_{i \nu}^\mu(y) \cdot e(y)(E_\mu) = \nabla^E_{\partial_i} e(y)(E_\nu).
\end{equation}
Now using that $\nabla^E$ is induced by the flat connection, we get
\[\nabla^E_{\partial_i} e(y)(E_\nu) = e (\partial_i(e(y)(E_\nu))) = e((\partial_i e)(y)(E_\nu)),\]
i.e., $e((\partial_i e)(y)(E_\nu))$ is the representation of $\nabla^E_{\partial_i} e(y)(E_\nu)$ with respect to the frame $\{E_\alpha, E_\beta^\perp\}$. Since the entries of $e$ are $C^\infty$-bounded, the entries of this representation $e((\partial_i e)(y)(E_\nu))$ are also $C^\infty$-bounded. From Equation \eqref{eq:christoffel_symbols_representation} we see that $\Gamma_{i \nu}^\mu(y)$ is the representation of $\nabla^E_{\partial_i} e(y)(E_\nu)$ in the frame $\{e(y)(E_\mu)\}$. So we conclude that the Christoffel symbols $\Gamma_{i \nu}^\mu(y)$ are $C^\infty$-bounded functions.

Equation \eqref{eq:ODE_synchronous_frame} and the theory of ODEs now tell us that the functions $E_\alpha(y)^\mu$ are $C^\infty$-bounded. Since these are the representations of the vectors $E_\alpha(y)$ in the basis $\{ e(y)(E_\alpha) \}$, we can conclude that the entries of the representations of the vectors $E_\alpha(y)$ in the basis $\{E_\alpha, E_\beta^\perp\}$ are $C^\infty$-bounded. But now these entries are exactly the first $(\dim E)$ columns of the change-of-basis matrix $\varphi(y)$.

Arguing analogously for the complement $E^\perp$, we get that the other columns of $\varphi(y)$ are also $C^\infty$-bounded, i.e., $\varphi$ itself is $C^\infty$-bounded.

It remains to show that the inverse homomorphism $\varphi^{-1}$ is $C^\infty$-bounded. But since pointwise it is given by the inverse matrix, i.e., $\varphi^{-1}(y) = \varphi(y)^{-1}$, this claim follows immediately from Cramer's rule, because we already know that $\varphi$ is $C^\infty$-bounded.
\end{proof}

\begin{proof}[Proof of point 2]
Let $\{E_\alpha(y)\}$ be a synchronous framing for $E$ and $\{E_\beta^\perp(y)\}$ one for $E^\perp$. Then $\{E_\alpha(y), E_\beta^\perp(y)\}$ is one for $E \oplus E^\perp$. Furthermore, let $\{v_i(y)\}$ be a synchronous framing for the trivial bundle $\IC^N$ and let $\varphi\colon E \oplus E^\perp \to \IC^N$ be the $C_b^\infty$-isomorphism.

Then projection matrix $e \in \Idem_{N \times N}(C^\infty(M))$ onto the subspace $E \oplus 0$ is given with respect to the basis $\{E_\alpha(y), E_\beta^\perp(y)\}$ of $E \oplus E^\perp$ and of $\IC^N$ by the usual projection matrix onto the first $(\dim E)$ vectors, i.e., its entries are clearly $C^\infty$-bounded since they are constant. Now changing the basis to $\{v_i(y)\}$, the representation of $e(y)$ with respect to this new basis is given by $\varphi^{-1}(y) \cdot e \cdot \varphi(y)$, i.e., $e \in \Idem_{N \times N}(C_b^\infty(M))$.
\end{proof}

If we have a $C_b^\infty$-complemented vector bundle $E$, then different choices of complements and different choices of isomorphisms with the trivial bundle lead to similar projection matrices. The proof of this is analogous to the corresponding proof in the usual case of vector bundles over compact Hausdorff spaces. We also get that $C_b^\infty$-isomorphic vector bundles produce similar projection matrices. Of course this also works the other way round, i.e., similar idempotent matrices give us $C_b^\infty$-isomorphic vector bundles. Again, the proof of this is the same as the one in the topological category.

\begin{defn}
Let $M$ be a manifold of bounded geometry. We define
\begin{itemize}
\item $\Vect_u(M)/_\sim$ as the abelian monoid of all complex vector bundles of bounded geometry over $M$ modulo $C_b^\infty$-isomorphism (the addition is given by the direct sum $[E] + [F] := [E \oplus F]$) and
\item $\Idem(C_b^\infty(M))/_\sim$ as the abelian monoid of idempotent matrizes of arbitrary size over the \Frechet algebra $C_b^\infty(M)$ modulo similarity (with addition defined as $[e] + [f] := \left[\begin{pmatrix}e&0\\0&f\end{pmatrix}\right]$).
\end{itemize}
They will be identified with each other in the following corollary.
\qed
\end{defn}

Let $f\colon M \to N$ be a $C^\infty$-bounded map\footnote{We use covers of $M$ and $N$ via normal coordinate charts of a fixed radius and demand that locally in this charts the derivatives of $f$ are all bounded and these bounds are independent of the chart used.} and $E$ a vector bundle of bounded geometry over $N$. Then it is clear that the pullback bundle $f^\ast E$ equipped with the pullback metric and connection is a vector bundle of bounded geometry over $M$.

The above discussion together with Proposition \ref{prop:image_proj_matrix_complemented} prove the following corollary:

\begin{cor}\label{cor:two_monoids_isomorphic}
The monoids $\Vect_u(M)/_\sim$ and $\Idem(C_b^\infty(M))/_\sim$ are isomorphic and this isomorphism is natural with respect to $C^\infty$-bounded maps between manifolds.
\end{cor}

From this Corollary \ref{cor:two_monoids_isomorphic}, Lemma \ref{lem:C_b_infty_local} and Proposition \ref{prop:every_bundle_complemented} we immediately get the following interpretation of the $0$th uniform $K$-theory group $K^0_u(M)$ of a manifold of bounded geometry:

\begin{thm}[Interpretation of $K^0_u(M)$]\label{thm:interpretation_K0u}
Let $M$ be a Riemannian manifold of bounded geometry and without boundary.

Then every element of $K^0_u(M)$ is of the form $[E] - [F]$, where both $[E]$ and $[F]$ are $C_b^\infty$-isomorphism classes of complex vector bundles of bounded geometry over $M$.

Moreover, every complex vector bundle of bounded geometry defines a class in $K^0_u(M)$.
\end{thm}

Note that the last statement in the above theorem is not trivial since it relies on the Proposition \ref{prop:every_bundle_complemented}.

\subsubsection*{Interpretation of \texorpdfstring{$K^1_u(M)$}{odd uniform K(M)}}

For the interpretation of $K^1_u(M)$ we will make use of suspensions of algebras. The suspension isomorphism theorem for operator $K$-theory states that we have an isomorphism $K_1(C_u(M)) \cong K_0(S C_u(M))$, where $S C_u(M)$ is the suspension of $C_u(M)$:
\begin{align*}
S C_u(M) & := \{ f \colon S^1 \to C_u(M) \ | \ f \text{ continuous and } f(1) = 0\}\\
& \cong \{ f \in C_u(S^1 \times M) \ | \ f(1, x) = 0 \text{ for all }x \in M\}.
\end{align*}
Equipped with the sup-norm this is again a $C^\ast$-algebra. Since functions $f \in S C_u(M)$ are uniformly continuous, the condition $f(1, x) = 0$ for all $x \in M$ is equivalent to $\lim_{t \to 1} f(t, x) = 0 \text{ uniformly in } x$.

Now in order to interpret $K_0(S C_u(M))$ via vector bundles of bounded geometry over $S^1 \times M$, we will need to find a suitable \Frechet subalgebra of $S C_u(M)$ so that we can again use Proposition \ref{prop:image_proj_matrix_complemented}. Luckily, this was already done by Phillips in \cite{phillips}:

\begin{defn}[Smooth suspension of a \Frechet algebras, {\cite[Definition 4.7]{phillips}}]\label{defn:smooth_suspension_algebra}
Let $A$ be a \Frechet algebra. Then the \emph{smooth suspension $S_\infty A$ of $A$} is defined as the \Frechet algebra
\[ S_\infty A := \{ f\colon S^1 \to A \ | \ f \text{ smooth and } f(1) = 0\}\]
equipped with the topology of uniform convergence of every derivative in every seminorm of $A$.
\qed
\end{defn}

For a manifold $M$ we have
\begin{align*}
S_\infty C_b^\infty(M) & \cong \{ f \in C_b^\infty(S^1 \times M) \ | \ f(1, x) = 0 \text{ for all }x \in M\} \\
& = \{ f \in C_b^\infty(S^1 \times M) \ | \ \forall k \in \IN_0 \colon \lim_{t \to 1} \nabla^k_x f(t, x) = 0 \text{ uniformly in } x\}.
\end{align*}

The proof of the following lemma is analogous to the proof of the Lemma \ref{lem:C_b_infty_local}:

\begin{lem}
Let $M$ be a manifold of bounded geometry.

Then the sup-norm completion of $S_\infty C_b^\infty(M)$ is $S C_u(M)$ and $S_\infty C_b^\infty(M)$ is a local $C^\ast$-algebra.
\end{lem}

Putting it all together, we get $K^1_u(M) = K_0(S_\infty C_b^\infty(M))$, and Proposition \ref{prop:image_proj_matrix_complemented}, adapted to our case here, gives us the following interpretation of the $1$st uniform $K$-theory group $K^1_u(M)$ of a manifold of bounded geometry:

\begin{thm}[Interpretation of $K^1_u(M)$]\label{thm:interpretation_K1u}
Let $M$ be a Riemannian manifold of bounded geometry and without boundary.

Then every elements of $K^1_u(M)$ is of the form $[E] - [F]$, where both $[E]$ and $[F]$ are $C_b^\infty$-isomorphism classes of complex vector bundles of bounded geometry over $S^1 \times M$ with the following property: there is some neighbourhood $U \subset S^1$ of $1$ such that $[E|_{U \times M}]$ and $[F|_{U \times M}]$ are $C_b^\infty$-isomorphic to a trivial vector bundle with the flat connection (the dimension of the trivial bundle is the same for both $[E|_{U \times M}]$ and $[F|_{U \times M}]$).

Moreover, every pair of complex vector bundles $E$ and $F$ of bounded geometry and with the above properties define a class $[E] - [F]$ in $K_u^1(M)$.
\end{thm}

Note that the last statement in the above theorem is not trivial since it relies on the Proposition \ref{prop:every_bundle_complemented}.

\subsection{Cap product}\label{sec:cap_product}

In this section we will define the cap product $\cap \colon K_u^p(X) \otimes K_q^u(X) \to K_{q-p}^u(X)$ for a locally compact and separable metric space $X$ of jointly bounded geometry\footnote{see Definition \ref{defn:jointly_bounded_geometry}}.

Recall that we have
\begin{equation*}
\LLip_R(X) := \{ f \in C_c(X) \ | \ f \text{ is }L\text{-Lipschitz}, \diam(\supp f) \le R \text{ and } \|f\|_\infty \le 1\}.
\end{equation*}

Let us first describe the cap product of $K_u^0(X)$ with $K^u_\ast(X)$ on the level of uniform Fredholm modules. The general definition of it will be given via dual algebras.

\begin{lem}\label{lem:proj_again_uniform_fredholm_module}
Let $P$ be a projection in $\Mat_{n \times n}(C_u(X))$ and let $(H, \rho, T)$ be a uniform Fredholm module.

We set $H_n := H \otimes \IC^n$, $\rho_n(\largecdot) := \rho(\largecdot) \otimes \id_{\IC^n}$, $T_n := T \otimes \id_{\IC^n}$ and by $\pi$ we denote the matrix $\pi_{ij} := \rho(P_{ij}) \in \Mat_{n \times n}(\IB(H)) = \IB(H_n)$.

Then $(\pi H_n, \pi \rho_n \pi, \pi T_n \pi)$ is a uniform Fredholm module, with an induced (multi-)grading if $(H, \rho, T)$ was (multi-)graded.
\end{lem}

\begin{proof}
Let us first show that the operator $\pi T_n \pi$ is a uniformly pseudolocal one. Let $R, L > 0$ be given and we have to show that $\{[\pi T_n \pi, \pi \rho_n(f) \pi] \ | \ f \in \LLip_R(X)\}$ is uniformly approximable. This means that we must show that for every $\varepsilon > 0$ there exists an $N > 0$ such that for every $[\pi T_n \pi, \pi \rho_n(f) \pi]$ with $f \in \LLip_R(X)$ there is a rank-$N$ operator $k$ with $\|[\pi T_n \pi, \pi \rho_n(f) \pi] - k\| < \varepsilon$.

We have
\[[\pi T_n \pi, \pi \rho_n(f) \pi] = \pi [T_n, \pi \rho_n(f)] \pi,\]
because $\pi^2 = \pi$ and $\pi$ commutes with $\rho_n(f)$. So since $(\pi \rho_n(f))_{ij} = \rho(P_{ij} f) \in \IB(H)$, we get for the matrix entries of the commutator
\[([T_n, \pi \rho_n(f)])_{ij} = [T, \rho(P_{ij} f)].\]

Since the $P_{ij}$ are bounded and uniformly continuous, they can be uniformly approximated by Lipschitz functions, i.e., there are $P_{ij}^\varepsilon$ with
\[\|P_{ij} - P_{ij}^\varepsilon\|_\infty < \varepsilon / (4n^2 \|T\|).\]
Note that we have $P_{ij}^\varepsilon f \in {L_{ij}\text{-}\operatorname{Lip}}_{R}(X)$, where $L_{ij}$ depends only on $L$ and $P_{ij}^\varepsilon$. We define $L^\prime := \max\{L_{ij}\}$.

Now we apply the uniform pseudolocality of $T$, i.e., we get a maximum rank $N^\prime$ corresponding to $R, L^\prime$ and $\varepsilon / 2n^2$. So let $k_{ij}^\varepsilon$ be the rank-$N^\prime$ operators corresponding to the functions $P_{ij}^\varepsilon f$, i.e.,
\[\|[T, \rho(P_{ij}^\varepsilon f)] - k_{ij}^\varepsilon\| < \varepsilon / 2n^2.\]

We set $k := \pi (k_{ij}^\varepsilon) \pi$, where $(k_{ij}^\varepsilon)$ is viewed as a matrix of operators. Then $k$ has rank at most $N := n^2 N^\prime$. Then we compute
\begin{align*}
\|[\pi T_n & \pi, \pi \rho_n(f) \pi] - k\|\\
& = \|\pi[T_n, \pi \rho_n(f)]\pi - \pi (k_{ij}^\varepsilon) \pi\|\\
& \le \|\pi\|^2 \cdot n^2 \cdot \max_{i,j}\{\|[T, \rho(P_{ij} f)] - k_{ij}^\varepsilon\|\}\\
& \le \|\pi\|^2 \cdot n^2 \cdot \max_{i,j}\{\underbrace{\|[T, \rho(P_{ij} f)] - [T, \rho(P_{ij}^\varepsilon f)]\|}_{=\|[T, \rho(P_{ij} - P_{ij}^\varepsilon)\rho(f)]\|} + \underbrace{\|[T, \rho(P_{ij}^\varepsilon f)] - k_{ij}^\varepsilon\|\}}_{\le \varepsilon / 2n^2}\\
& \le \|\pi\|^2 \cdot n^2 \cdot \max_{i,j}\{2 \|T\| \cdot \underbrace{\|\rho(P_{ij} - P_{ij}^\varepsilon)\| \cdot \|\rho(f)\|}_{\le \varepsilon/(4n^2 \|T\|)} + \varepsilon / 2n^2\}\\
& \le \|\pi\|^2 \cdot \varepsilon,
\end{align*}
which concludes the proof of the uniform pseudolocality of $\pi T_n \pi$.

That $(\pi T_n \pi)^2 - 1$ and $\pi T_n \pi - (\pi T_n \pi)^\ast$ are uniformly locally compact can be shown analogously. Note that because $T$ is uniformly pseudolocal we may interchange the order of the operators $T_n$ and $\rho(P_{ij}^\varepsilon f)$ in formulas (since for fixed $R$ and $L$ the subset $\{[T_n, \rho(P_{ij}^\varepsilon f)] \ | \ f \in \LLip_R(X) \} \subset \IB(H_n)$ is uniformly approximable).

We have shown that $(\pi H_n, \pi \rho_n \pi, \pi T_n \pi)$ is a uniform Fredholm module. That it inherits a (multi-)grading from $(H, \rho, T)$ is clear and this completes the proof.
\end{proof}

That the construction from the above lemma is compatible with the relations defining $K$-theory and uniform $K$-homology and that it is bilinear is quickly deduced and completely analogous to the non-uniform case. So we get a well-defined pairing
\[\cap\colon K_u^0(X) \otimes K_\ast^u(X) \to K_\ast^u(X)\]
which exhibits $K_\ast^u(X)$ as a module over the ring $K_u^0(X)$.\footnote{Compatibility with the internal product on $K_u^0(X)$, i.e., $(P \otimes Q) \cap T = P \cap (Q \cap T)$, is easily deduced. It mainly uses the fact that the isomorphism $\Mat_{n \times n}(\IC) \otimes \Mat_{m \times m}(\IC) \cong \Mat_{nm \times nm}(\IC)$ is canonical up to the ordering of basis elements. But different choices of orderings result in isomorphisms that differ by inner automorphisms, which makes no difference at the level of $K$-theory.}

To define the cap product in its general form, we will use the dual algebra picture of uniform $K$-homology, i.e., Paschke duality:

\begin{defn}[{\cite[Definition 4.1]{spakula_uniform_k_homology}}]\label{defn:frakD_frakC}
Let $H$ be a separable Hilbert space and $\rho \colon C_0(X) \to \IB(H)$ a representation.

We denote by $\frakD^u_{\rho \oplus 0}(X) \subset \IB(H \oplus H)$ the $C^\ast$-algebra of all uniformly pseudolocal operators with respect to the representation $\rho \oplus 0$ of $C_0(X)$ on the space $H \oplus H$ and by $\frakC^u_{\rho \oplus 0}(X) \subset \IB(H \oplus H)$ the $C^\ast$-algebra of all uniformly locally compact operators.
\qed
\end{defn}

That the algebras $\frakD^u_{\rho \oplus 0}(X)$ and $\frakC^u_{\rho \oplus 0}(X)$ are indeed $C^\ast$-algebras was shown by \Spakula in \cite[Lemma 4.2]{spakula_uniform_k_homology}. There it was also shown that $\frakC^u_{\rho \oplus 0}(X) \subset \frakD^u_{\rho \oplus 0}(X)$ is a closed, two-sided $^\ast$-ideal.

\begin{defn}
The groups $K_{-1}^u(X; {\rho \oplus 0})$ are analogously defined as $K_{-1}^u(X)$, except that we consider only uniform Fredholm modules whose Hilbert spaces and representations are (finite or countably infinite) direct sums of $H \oplus H$ and $\rho \oplus 0$.

For $K_0^u(X; {\rho \oplus 0})$ we consider only uniform Fredholm modules modeled on $H^\prime \oplus H^\prime$ with the representation $\rho^\prime \oplus \rho^\prime$, where $H^\prime$ is a finite or countably infinite direct sum of $H \oplus H$ and $\rho^\prime$ analogously a direct sum of finitely or infinitely many $\rho \oplus 0$, and the grading is given by interchanging the two summands in $H^\prime \oplus H^\prime$. Such Fredholm modules are called \emph{balanced} in \cite[Definition 8.3.10]{higson_roe}.
\qed
\end{defn}

\begin{prop}[{\cite[Proposition 4.3]{spakula_uniform_k_homology}}]\label{prop:paschke_duality}
The maps
\[\varphi_\ast \colon K_{1+\ast}(\frakD^u_{\rho \oplus 0}(X)) \to K_\ast^u(X; {\rho \oplus 0})\]
for $\ast = -1, 0$ are isomorphisms.
\end{prop}

Combining the above proposition with the following uniform version of Voiculescu's Theorem, we get the needed uniform version of Paschke duality.

\begin{thm}[{\cite[Corollary 3.6]{spakula_universal_rep}}]\label{thm:paschke_universal}
Let $X$ be a locally compact and separable metric space of jointly bounded geometry and $\rho\colon C_0(X) \to \IB(H)$ an ample representation, i.e., $\rho$ is non-degenerate and $\rho(f) \in \IK(H)$ implies $f \equiv 0$.

Then we have
\[K_\ast^u(X; \rho \oplus 0) \cong K_\ast^u(X)\]
for both $\ast = -1,0$.
\end{thm}

The following lemma is a uniform analog of the fact \cite[Lemma 5.4.1]{higson_roe} and is essentially proven in \cite[Lemma 5.3]{spakula_uniform_k_homology} (by ``setting $Z := \emptyset$'' in that lemma).

\begin{lem}\label{lem:K_theory_frakC_zero}
We have
\[K_\ast(\frakC^u_{\rho \oplus 0}(X)) = 0\]
and so the quotient map $\frakD^u_{\rho \oplus 0}(X) \to \frakD^u_{\rho \oplus 0}(X) / \frakC^u_{\rho \oplus 0}(X)$ induces an isomorphism
\begin{equation}\label{eq:K_theory_frakC_zero}
K_\ast(\frakD^u_{\rho \oplus 0}(X)) \cong K_\ast(\frakD^u_{\rho \oplus 0}(X) / \frakC^u_{\rho \oplus 0}(X))
\end{equation}
due to the $6$-term exact sequence for $K$-theory.
\end{lem}

The last ingredient to construct the cap product is the inclusion
\begin{equation}\label{eq:commutator_C_u_with_frakD}
[C_u(X), \frakD^u_{\rho \oplus 0}(X)] \subset \frakC^u_{\rho \oplus 0}(X).
\end{equation}
It is proven in the following way: let $\varphi \in C_u(X)$ and $T \in \frakD^u_{\rho \oplus 0}(X)$. We have to show that $[\varphi, T] \in \frakC^u_{\rho \oplus 0}(X)$. By approximating $\varphi$ uniformly by Lipschitz functions we may without loss of generality assume that $\varphi$ itself is already Lipschitz. Now the claim follows immediately from $f[\varphi, T] = [f \varphi, T] - [f,T]\varphi$ since $T$ is uniformly pseudolocal.

Now we are able to define the cap product. Consider the map
\[ \sigma \colon C_u(X) \otimes \frakD^u_{\rho \oplus 0}(X) \to \frakD^u_{\rho \oplus 0}(X) / \frakC^u_{\rho \oplus 0}(X), \ f \otimes T \mapsto [fT].\]
It is a multiplicative $^\ast$-homomorphism due to the above Equation \eqref{eq:commutator_C_u_with_frakD} and hence induces a map on $K$-theory
\begin{equation*}
\sigma_\ast \colon K_\ast(C_u(X) \otimes \frakD^u_{\rho \oplus 0}(X)) \to K_\ast(\frakD^u_{\rho \oplus 0}(X) / \frakC^u_{\rho \oplus 0}(X)).
\end{equation*}
Using Paschke duality we may define the cap product as the composition
\begin{align*}
K_u^p(X) \otimes K_q^u(X; \rho \oplus 0) & \ = \ K_{-p}(C_u(X)) \otimes K_{1+q}(\frakD^u_{\rho \oplus 0}(X))\\
& \ \to \ \! K_{-p+1+q}(C_u(X) \otimes \frakD^u_{\rho \oplus 0}(X))\\
& \ \stackrel{\sigma_\ast}\to \ \! K_{-p+1+q}(\frakD^u_{\rho \oplus 0}(X) / \frakC^u_{\rho \oplus 0}(X))\\
& \stackrel{\eqref{eq:K_theory_frakC_zero}}\cong K_{-p+1+q}(\frakD^u_{\rho \oplus 0}(X))\\
& \ = \ K_{q-p}^u(X; \rho \oplus 0),
\end{align*}
where the first arrow is the external product on $K$-theory. So we get the cap product
\[\cap \colon K_u^p(X) \otimes K_q^u(X) \to K_{q-p}^u(X).\]

Let us state in a proposition some properties of it that we will need. The proofs of these properties are analogous to the non-uniform case.

\begin{prop}\label{prop:properties_general_cap_product}
The cap product has the following properties:
\begin{itemize}
\item the pairing of $K_u^0(X)$ with $K_\ast^u(X)$ coincides with the one in Lemma \ref{lem:proj_again_uniform_fredholm_module},
\item the fact that $K_\ast^u(X)$ is a module over $K_u^0(X)$ generalizes to
\begin{equation}\label{eq:general_cap_compatibility_module}
(P \otimes Q) \cap T = P \cap (Q \cap T)
\end{equation}
for all elements $P, Q \in K_u^\ast(X)$ and $T \in K_\ast^u(X)$, where $\otimes$ is the internal product on uniform $K$-theory,
\item if $X$ and $Y$ have jointly bounded geometry, then we have the following compatibility with the external products:
\begin{equation}\label{eq:compatibility_cap_external}
(P \times Q) \cap (S \times T) = (-1)^{qs} (P \cap S) \times (Q \cap T),
\end{equation}
where $P \in K_u^p(X)$, $Q \in K_u^q(Y)$ and $S \in K^u_s(X)$, $T \in K^u_t(Y)$, and
\item if we have a manifold of bounded geometry $M$, a vector bundle of bounded geometry $E \to M$ and an operator $D$ of Dirac type, then
\begin{equation}\label{eq:cap_twisted_Dirac}
[E] \cap [D] = [D_E] \in K_\ast^u(M),
\end{equation}
where $D_E$ is the twisted operator.
\end{itemize}
\end{prop}

\subsection{Uniform \texorpdfstring{$K$-\Poincare}{K-Poincare }duality}\label{sec:poincare_duality}

We will prove in this section that uniform $K$-theory is indeed the dual theory to uniform $K$-homology. To this end we will show that they are \Poincare dual to each other. This will be accomplished by a suitable Mayer--Vietoris induction of which the idea will also be used later in this paper to prove similar results like the uniform Chern character isomorphism theorems in Section \ref{sec:chern_isos}.

\begin{thm}[Uniform $K$-\Poincare duality]
\label{thm:Poincare_duality_K}
Let $M$ be an $m$-dimensional spin$^c$ manifold of bounded geometry and without boundary.

Then the cap product $\largecdot \cap [M] \colon K_u^\ast(M) \to K^u_{m-\ast}(M)$ with its uniform $K$-fundamental class $[M] \in K_m^u(M)$ is an isomorphism.
\end{thm}

The proof of this theorem will occupy the whole subsection. We will first have to prove some auxiliary results before we will start on Page \pageref{page:proof_Poincare_duality} to assemble them into a proof of uniform $K$-\Poincare duality.

We will need the following Theorem \ref{thm:triangulation_bounded_geometry} about manifolds of bounded geometry. To state it, we have to recall some notions:

\begin{defn}[Bounded geometry simplicial complexes]\label{defn:simplicial_complex_bounded_geometry}
A simplicial complex has \emph{bounded geometry} if there is a uniform bound on the number of simplices in the link of each vertex.

A subdivision of a simplicial complex of bounded geometry with the properties that
\begin{itemize}
\item each simplex is subdivided a uniformly bounded number of times on its $n$-skeleton, where the $n$-skeleton is the union of the $n$-dimensional sub-simplices of the simplex, and that
\item the distortion $\operatorname{length}(e) + \operatorname{length}(e)^{-1}$ of each edge $e$ of the subdivided complex is uniformly bounded in the metric given by barycentric coordinates of the original complex,
\end{itemize}
is called a \emph{uniform subdivision}.
\qed
\end{defn}

\begin{defn}[Bi-Lipschitz equivalences]
Two metric spaces $X$ and $Y$ are said to be \emph{bi-Lipschitz equivalent} if there is a homeomorphism $f\colon X \to Y$ with
\[\tfrac{1}{C} d_X(x,x^\prime) \le d_Y(f(x), f(x^\prime)) \le C d_X(x,x^\prime)\]
for all $x,x^\prime \in X$ and some constant $C > 0$.
\qed
\end{defn}

\begin{thm}[{\cite[Theorem 1.14]{attie_classification}}]\label{thm:triangulation_bounded_geometry}
Let $M$ be a manifold of bounded geometry and without boundary.

Then $M$ admits a triangulation as a simplicial complex of bounded geometry whose metric given by barycentric coordinates is bi-Lipschitz equivalent to the metric on $M$ induced by the Riemannian structure. This triangulation is unique up to uniform subdivision.

Conversely, if $M$ is a simplicial complex of bounded geometry which is a triangulation of a smooth manifold, then this smooth manifold admits a metric of bounded geometry with respect to which it is bi-Lipschitz equivalent to $M$.
\end{thm}

\begin{rem}\label{rem:attie_regularity}
Attie uses in \cite{attie_classification} a weaker notion of bounded geometry as we do: additionally to a uniformly positive injectivity radius he only requires the sectional curvatures to be bounded in absolute value (i.e., the curvature tensor is bounded in norm), but he assumes nothing about the derivatives (see \cite[Definition 1.4]{attie_classification}). But going into his proof of \cite[Theorem 1.14]{attie_classification}, we see that the Riemannian metric constructed for the second statement of the theorem is actually of bounded geometry in our strong sense (i.e., also with bounds on the derivatives of the curvature tensor).

As a corollary we get that for any manifold of bounded geometry in Attie's weak sense there is another Riemannian metric of bounded geometry in our strong sense that is bi-Lipschitz equivalent the original one (in fact, this bi-Lipschitz equivalence is just the identity map of the manifold, as can be seen from the proof).
\qed
\end{rem}

\begin{lem}\label{lem:suitable_coloring_cover_M}
Let $M$ be a manifold of bounded geometry and without boundary.

Then there is an $\varepsilon > 0$ and a countable collection of uniformly discretely distributed points $\{x_i\}_{i \in I} \subset M$ such that $\{B_{\varepsilon}(x_i)\}_{i \in I}$ is a uniformly locally finite cover of $M$. We can additionally arrange such that it has the following two properties:
\begin{enumerate}
\item \label{cover_i} It is possible to partition $I$ into a finite amount of subsets $I_1, \ldots, I_N$ such that for each $1 \le j \le N$ the subset $U_j := \bigcup_{i \in I_j} B_{\varepsilon}(x_i)$ is a disjoint union of balls that are a uniform distance apart from each other, and such that for each $1 \le K \le N$ the connected components of $U_K := U_1 \cup \ldots \cup U_k$ are also a uniform distance apart from each other (see Figure~\ref{fig:not_allowed_cover}).
\item \label{cover_ii} Instead of choosing balls $B_\varepsilon(x_i)$ to get our cover of $M$ it is possible to choose other open subsets such that additionally to the property from Point~\ref{cover_i} for any distinct $1 \le m,n \le N$ the symmetric difference $U_m \Delta U_n$ consists of open subsets of $M$ which are a uniform distance apart from each other.\footnote{To see a non-example, in the lower part of Figure \ref{fig:not_allowed_cover} this is actually \emph{not} the case.}
\end{enumerate}
\end{lem}

\begin{figure}[h!]
\centering
\includegraphics[scale=0.5]{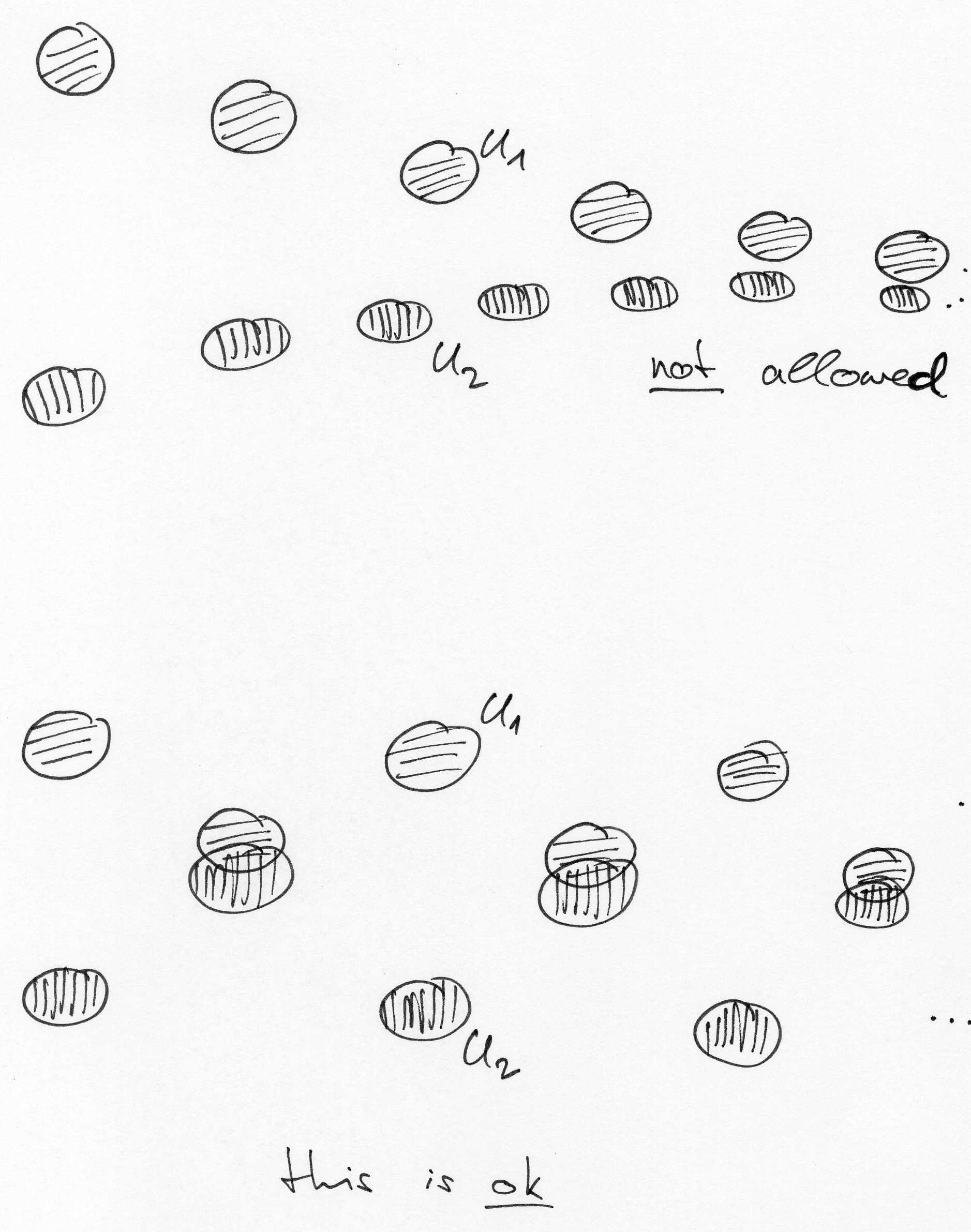}
\caption{Illustration for Lemma \ref{lem:suitable_coloring_cover_M}.\ref{cover_i}.}
\label{fig:not_allowed_cover}
\end{figure}

\begin{proof}
Let us first show how to get a cover of $M$ satisfying Point~\ref{cover_i} from the lemma.

We triangulate $M$ via the above Theorem \ref{thm:triangulation_bounded_geometry}. Then we may take the vertices of this triangulation as our collection of points $\{x_i\}_{i \in I}$ and set $\varepsilon$ to $2/3$ of the length of an edge multiplied with the constant $C$ which we get since the metric derived from barycentric coordinates is bi-Lipschitz equivalent to the metric derived from the Riemannian structure.

Two balls $B_\varepsilon(x_i)$ and $B_\varepsilon(x_j)$ for $x_i \not= x_j$ intersect if and only if $x_i$ and $x_j$ are adjacent vertices, and in the case that they are not adjacent, these balls are a uniform distance apart from each other. Hence it is possible to find a coloring of all these balls $\{B_\varepsilon(x_i)\}_{i \in I}$ with only finitely many colors having the claimed Property~\ref{cover_i}: apply Lemma \ref{lem:coloring_graph} to the covering $\{B_\varepsilon(x_i)\}_{i \in I}$ which has finite multiplicity due to bounded geometry.

To prove Point~\ref{cover_ii}, we replace in our cover of $M$ the balls $B_\varepsilon(x_i)$ with slightly differently chosen open subsets, as shown in the $2$-dimensional case in Figure~\ref{fig_fancy_coloring} (we are working in a triangulation of $M$ as above in the proof of Point~\ref{cover_i}).
\end{proof}

\begin{figure}[h!]
\centering
\includegraphics[scale=0.3]{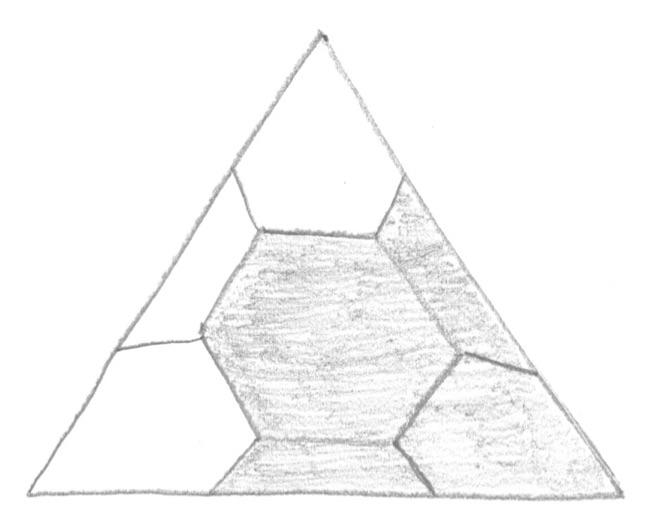}
\caption{Illustration for Lemma \ref{lem:suitable_coloring_cover_M}.\ref{cover_ii}.}
\label{fig_fancy_coloring}
\end{figure}

Our proof of \Poincare duality is a Mayer--Vietoris induction which will have only finitely many steps. So we first have to discuss the corresponding Mayer--Vietoris sequences.

We will start with the Mayer--Vietoris sequence for uniform $K$-theory. Let $O \subset M$ be an open subset, not necessarily connected. We denote by $(M, d)$ the metric space $M$ endowed with the metric induced from the Riemannian metric $g$ on $M$, and by $C_u(O, d)$ we denote the $C^\ast$-algebra of all bounded, uniformly continuous functions on $O$, where we regard $O$ as a metric space equipped with the subset metric induced from $d$ (i.e., we do not equip $O$ with the induced Riemannian metric and consider then the corresponding induced metric structure).

\begin{defn}\label{defn:uniform_k_th_subset}
Let $O \subset M$ be an open subset, not necessarily connected. We define $K^p_{u}(O \subset M) := K_{-p}(C_u(O,d))$.
\qed
\end{defn}

We will also need the following technical theorem:

\begin{lem}\label{lem:extension_samuel_compactification}
Let $O \subset M$ be open, not necessarily connected. Then every function $f \in C_u(O, d)$ has an extension to an $F \in C_u(M, d)$.
\end{lem}

\begin{proof}
For a metric space $X$ let $uX$ denote the Gelfand space of $C_u(X)$, i.e., this is a compactification of $X$ (the \emph{Samuel compactification}) with the following universal property: a bounded, continuous function $f$ on $X$ has an extension to a continuous function on $uX$ if and only if $f$ is uniformly continuous. We will use the following property of Samuel compactifications (see \cite[Theorem 2.9]{woods}): if $S \subset X \subset uX$, then the closure $\closure_{uX}(S)$ of $S$ in $uX$ is the Samuel compactification $uS$ of $S$.

So given $f \in C_u(O, d)$, we can extend it to a continuous function $\tilde{f} \in C(uO)$. Since $uO = \closure_{uM}(O)$, i.e., a closed subset of a compact Hausdorff space, we can extend $\tilde{f}$ by the Tietze extension theorem to a bounded, continuous function $\tilde{F}$ on $uM$. Its restriction $F := \tilde{F}|_M$ to $M$ is then a bounded, uniformly continuous function of $M$ extending $f$.
\end{proof}

\begin{lem}\label{lem:MV_uniform_K}
Let the subsets $U_j$, $U_K$ of $M$ for $1 \le j,K \le N$ be as in Lemma \ref{lem:suitable_coloring_cover_M}. Then we have Mayer--Vietoris sequences
\[\xymatrix{
K^0_{u}(U_K \cup U_{k+1} \subset M) \ar[r] & K^0_{u}(U_K \subset M) \oplus K^0_{u}(U_{k+1} \subset M) \ar[r] & K^0_{u}(U_K \cap U_{k+1} \subset M) \ar[d]\\
K^1_{u}(U_K \cap U_{k+1} \subset M) \ar[u] & K^1_{u}(U_K \subset M) \oplus K^1_{u}(U_{k+1} \subset M) \ar[l] & K^1_{u}(U_K \cup U_{k+1} \subset M) \ar[l]}\]
where the horizontal arrows are induced from the corresponding restriction maps.
\end{lem}

\begin{proof}
Recall the Mayer--Vietoris sequence for operator $K$-theory of $C^\ast$-algebras (see, e.g., \cite[Theorem 21.2.2]{blackadar}): given a commutative diagram of $C^\ast$-algebras
\[\xymatrix{P \ar[r]^{\sigma_1} \ar[d]_{\sigma_2} & A_1 \ar[d]^{\varphi_1}\\ A_2 \ar[r]^{\varphi_2} & B}\]
with $P = \{(a_1, a_2) \ | \ \varphi_1(a_1) = \varphi_2(a_2)\} \subset A_1 \oplus A_2$ and $\varphi_1$ and $\varphi_2$ surjective, then there is a long exact sequence (via Bott periodicity we get the $6$-term exact sequence)
\[\ldots \to K_n(P) \stackrel{({\sigma_1}_\ast,{\sigma_2}_\ast)}\longrightarrow K_n(A_1) \oplus K_n(A_2) \stackrel{{\varphi_2}_\ast - {\varphi_1}_\ast}\longrightarrow K_n(B) \to K_{n-1}(P) \to \ldots\]

We set $A_1 := C_u(U_K, d)$, $A_2 := C_u(U_{k+1}, d)$, $B := C_u(U_K \cap U_{k+1}, d)$ and $\varphi_1$, $\varphi_2$ the corresponding restriction maps. Due to the property of the sets $U_K$ as stated in the Lemma \ref{lem:suitable_coloring_cover_M} we get $P = C_u(U_K \cup U_{k+1}, d)$ and $\sigma_1$, $\sigma_2$ again just the restriction maps. To show that the maps $\varphi_1$ and $\varphi_2$ are surjective we  have to use the above Lemma \ref{lem:extension_samuel_compactification}.
\end{proof}

We will also need corresponding Mayer--Vietoris sequences for uniform $K$-homology. As for uniform $K$-theory we use here also the induced subspace metric (and not the metric derived from the induced Riemannian metric): let a not necessarily connected subset $O \subset M$ be given. We define $K_\ast^{u}(O \subset M)$ to be the uniform $K$-homology of $O$, where $O$ is equipped with the subspace metric from $M$, where we view $M$ as a metric space. The inclusion $O \hookrightarrow M$ is in general not a proper map (e.g., if $O$ is an open ball in a manifold) but this is no problem to us since we will have to use the wrong-way maps that exist for open subsets $O \subset M$: they are given by the inclusions $\LLip_R(O) \subset \LLip_R(M)$ for all $R, L > 0$. So we get a map $K_\ast^{u}(M) \to K_\ast^{u}(O \subset M)$ for every open subset $O \subset M$.

Existence of Mayer--Vietoris sequences for uniform $K$-homology of the subsets in the cover $\{U_K, U_{k+1}\}$ of $U_K \cup U_{k+1}$ (recall that we used Lemma \ref{lem:suitable_coloring_cover_M} to get these subsets) incorporating the wrong-way maps may be similarly shown as \cite[Section 8.5]{higson_roe}. The crucial excision isomorphism from that section may be constructed analogously as described in \cite[Footnote 73]{higson_roe}: for that construction Kasparov's Technical Theorem is used, and we have to use here in our uniform case the corresponding uniform construction which is as used in our construction of the external product for uniform $K$-homology.

Note that \Spakula constructed a Mayer--Vietoris sequence for uniform $K$-homology in \cite[Section 5]{spakula_uniform_k_homology}, but for closed subsets of a proper metric space. His arrows also go in the other direction as ours (since his arrows are induced by the usual functoriality of uniform $K$-homology).

We denote by $[M]|_O \in K_m^{u}(O \subset M)$ the class of the Dirac operator associated to the restriction to a neighbourhood of $O$ of the complex spinor bundle of bounded geometry defining the spin$^c$-structure of $M$ (i.e., we equip the neighbourhood with the induced spin$^c$-structure).

The cap product of $K^\ast_{u}(O \subset M)$ with $[M]|_O$ is analogously defined as the usual one, i.e., we get maps $\largecdot \cap [M]|_O \colon K^\ast_{u}(O \subset M) \to K_{m - \ast}^{u}(O \subset M)$. Now we have to argue why we get commutative squares between the Mayer--Vietoris sequences of uniform $K$-theory and uniform $K$-homology using the cap product. This is known for usual $K$-theory and $K$-homology; see, e.g., \cite[Exercise 11.8.11(c)]{higson_roe}. Since the cap product is in our uniform case completely analogously defined (see the second-to-last display before Proposition \ref{prop:properties_general_cap_product}), we may analogously conclude that we get commutative squares between our uniform Mayer--Vietoris sequences.

Let us summarize the above results:

\begin{lem}\label{lem:cap_prod_commutes_MV}
Let the subsets $U_j$, $U_K$ of $M$ for $1 \le j,K \le N$ be as in Lemma \ref{lem:suitable_coloring_cover_M}. Then we have corresponding Mayer--Vietoris sequences
\[\xymatrix{
K_0^{u}(U_K \cup U_{k+1} \subset M) \ar[r] & K_0^{u}(U_K \subset M) \oplus K_0^{u}(U_{k+1} \subset M) \ar[r] & K_0^{u}(U_K \cap U_{k+1} \subset M) \ar[d]\\
K_1^{u}(U_K \cap U_{k+1} \subset M) \ar[u] & K_1^{u}(U_K \subset M) \oplus K_1^{u}(U_{k+1} \subset M) \ar[l] & K_1^{u}(U_K \cup U_{k+1} \subset M) \ar[l]}\]
and the cap product gives the following commutative diagram:\vspace{\baselineskip}
\[\mathclap{\xymatrix{
K_{u}^\ast(U_{K} \cap U_{k+1} \subset M) \ar[r] \ar[d] & K_{u}^\ast(U_{K} \cup U_{k+1} \subset M) \ar[r] \ar[d] & K_{u}^\ast(U_K \subset M) \oplus K_{u}^\ast(U_{k+1} \subset M) \ar[d] \ar@/_1.5pc/[ll] \\
K^{u}_{m-\ast}(U_{K} \cap U_{k+1} \subset M) \ar[r] & K^{u}_{m-\ast}(U_{K} \cup U_{k+1} \subset M) \ar[r] & K^{u}_{m-\ast}(U_K \subset M) \oplus K^{u}_{m-\ast}(U_{k+1} \subset M) \ar@/^1.5pc/[ll]
}}\vspace{\baselineskip}\]
(We have suppressed the index shift due to the boundary maps in the latter diagram.)
\end{lem}

The last lemma that we will need before we will start to assemble everything together into a proof of uniform $K$-\Poincare duality is the following:

\begin{lem}\label{lem:uniform_k_hom_balls}
Let $M$ be an $m$-dimensional manifold of bounded geometry and let $U \subset M$ be a subset consisting of uniformly discretely distributed geodesic balls in $M$ having radius less than the injectivity radius of $M$ (i.e., each geodesic ball is diffeomorphic to the standard ball in Euclidean space $\IR^m$). Let the balls be indexed by a set $Y$.

Then we have $K_m^u(U \subset M) \cong \ell^\infty_\IZ(Y)$, the group of all bounded, integer-valued sequences indexed by $Y$, and $K_p^u(U \subset M) = 0$ for $p \not= m$.
\end{lem}

\begin{proof}
The proof is analogous to the proof of Lemma \ref{lem:uniform_k_hom_discrete_space}. It uses the fact that for an open ball $O \subset \IR^m$ we have $K_m(O \subset \IR^m) \cong \IZ$, and $K_p(O \subset \IR^m) = 0$ for $p \not= m$.
\end{proof}

\begin{proof}[Proof of uniform $K$-\Poincare duality]
\label{page:proof_Poincare_duality}
First we invoke Lemma \ref{lem:suitable_coloring_cover_M} to get subsets $U_j$ for $1 \le j \le N$.

The induction starts with the subsets $U_1$, $U_2$ and $U_1 \cap U_2$, which are collections of uniformly discretely distributed open balls, resp., in the case of $U_1 \cap U_2$ it is a collection of intersections of open balls, which is homotopy equivalent to a collection of uniformly discretely distributed open balls by a uniformly cobounded, proper and Lipschitz homotopy. Now uniform $K$-theory of a space coincides with the uniform $K$-theory of its completion, and furthermore, uniform $K$-theory is homotopy invariant with respect to Lipschitz homotopies. So the uniform $K$-theory of a collection of open balls is the same as the uniform $K$-theory of a collection of points. This groups we have already computed in Lemma \ref{lem:uniform_k_th_discrete_space}.

Uniform $K$-homology is homotopy invariant with respect to uniformly cobounded, proper and Lipschitz homotopies (see Theorem \ref{thm:homotopy_equivalence_k_hom}), and for totally bounded spaces it coincides with usual $K$-homology (see Proposition \ref{prop:compact_space_every_module_uniform}). So we have to compute uniform $K$-homology of a collection of uniformly discretely distributed open balls. This we have done in the above Lemma \ref{lem:uniform_k_hom_balls}.

Now we can argue that cap product is an isomorphism $K_u^\ast(U \subset M) \cong K_{m-\ast}^u(U \subset M)$, where $U$ is as in the above lemma. For this we have to note that if $M$ is a \spinc manifold, then the restriction of its complex spinor bundle to any ball of $U$ is isomorphic to the complex spinor bundle on the open ball $O \subset \IR^m$. This means that the cap product on $U$ coincides on each open ball of $U$ with the usual cap product on the open ball $O \subset \IR^m$. This all shows that we have \Poincare duality for the subsets $U_1$, $U_2$ and $U_1 \cap U_2$ (note that $U_1 \cap U_2$ is homotopic to a collection of open balls).

With the above Lemma \ref{lem:cap_prod_commutes_MV} we therefore get with the five lemma that the cap product is also an isomorphism for $U_1 \cup U_2$. The rest of the proof proceeds by induction over $k$ (there are only finitely many steps since we only go up to $k = N-1$), invoking every time the above Lemma \ref{lem:cap_prod_commutes_MV} and the five lemma. Note that in order to see that the cap product is an isomorphism on $U_K \cap U_{k+1}$, we have to write $U_K \cap U_{k+1} = (U_1 \cap U_{k+1}) \cup \ldots \cup (U_k \cap U_{k+1})$. This is a union of $k$ geodesically convex open sets and we have to do a separate induction on this one.
\end{proof}

\section{Index theorems for uniform operators}

In this section we will assemble everything that we had up to now into our uniform index theorems. For this we will first have to define the uniform de Rham (co-)homology theories that will serve as receptacles for the index classes of our operators (the uniform homological Chern character of a uniform, abstractly elliptic operator will give us its analytic index class and uniform de Rham cohomology will receive the topological index classes of symmetric, elliptic uniform pseudodifferential operators). After a small detour into the world of the Chern character isomorphism theorem we will then finally prove the uniform index theorems.

\subsection{Cyclic cocycles of uniformly finitely summable modules}

The goal of this section is to construct the uniform homological Chern character maps from uniform $K$-homology $K_\ast^u(M)$ of $M$ to continuous periodic cyclic cohomology $\HPucont^\ast(\Winftyone(M))$ of the Sobolev space $\Winftyone(M)$.

First we will recall the definition of Hochschild, cyclic and periodic cyclic cohomology of a (possibly non-unital) complete locally convex algebra $A$\footnote{We consider here only algebras over the field $\IC$. Furthermore, we assume that multiplication in $A$ is jointly continuous.}. The classical reference for this is, of course, Connes' seminal paper \cite{connes_noncomm_diff_geo}. The author also found Khalkhali's book \cite{khalkhali_basic} a useful introduction to these matters.

\begin{defn}
The \emph{continuous Hochschild cohomology} $\HHucont^\ast(A)$ of $A$ is the homology of the complex
\[\Cucont^0(A) \stackrel{b}\longrightarrow \Cucont^1(A) \stackrel{b}\longrightarrow \ldots,\]
where $\Cucont^n(A) = \Hom(A^{\widehat{\otimes}(n+1)}, \IC)$ and the boundary map $b$ is given by
\begin{align*}
(b\varphi)(a_0, \ldots, a_{n+1}) = & \sum_{i=0}^n (-1)^i \varphi(a_0, \ldots, a_i a_{i+1}, \ldots, a_{n+1}) +\\
& + (-1)^{n+1} \varphi(a_{n+1} a_0, a_1, \ldots, a_n).
\end{align*}
We use the completed projective tensor product $\widehat{\otimes}$ and the linear functionals are assumed to be continuous. But we still factor our only the image of the boundary operator to define the homology, and \emph{not} the closure of the image of $b$.
\qed
\end{defn}

\begin{defn}
The \emph{continuous cyclic cohomology} $\HCucont^\ast(A)$ of $A$ is the homology of the following subcomplex of the Hochschild cochain complex:
\[\Clucont^0(A) \stackrel{b}\longrightarrow \Clucont^1(A) \stackrel{b}\longrightarrow \ldots,\]
where $\Clucont^n(A) = \{\varphi \in \Cucont^n(A)\colon \varphi(a_n, a_0, \ldots, a_{n-1}) = (-1)^n \varphi(a_0, a_1, \ldots, a_n)\}$.
\qed
\end{defn}

There is a certain \emph{periodicity operator} $S\colon \HCucont^n(A) \to \HCucont^{n+2}(A)$. For the tedious definition of this operator on the level of cyclic cochains we refer the reader to Connes' original paper \cite[Lemma 11 on p.~322]{connes_noncomm_diff_geo} or to his book \cite[Lemma 14 on p.~198]{connes_book}.

\begin{defn}
The \emph{continuous periodic cyclic cohomology} $\HPucont^\ast(A)$ of $A$ is defined as the direct limit
\[\HPucont^\ast(A) = \underrightarrow{\lim} \ \HCucont^{\ast+2n}(A)\]
with respect to the maps $S$.
\qed
\end{defn}

Let $(H, \rho, T)$ be a graded uniform Fredholm module over $M$ and denote by $\epsilon$ the grading automorphism of the graded Hilbert space $H$. Furthermore, assume that $(H, \rho, T)$ is involutive and uniformly $p$-summable, where the latter means $\sup_{f \in \LLip_R(M)} \|[T, \rho(f)]\|_p < \infty$ for the Schatten $p$-norm $\|\largecdot\|_p$.

Having such a involutive, uniformly $p$-summable Fredholm module at hand we define for all $m$ with $2m+1 \ge p$ a cyclic $2m$-cocycle on $\Winftyone(M)$, i.e., on the Sobolev space of infinite order and $L^1$-integrability, by
\[\ch^{0,2m}(H, \rho, T)(f_0, \ldots, f_{2m}) := \tfrac{1}{2} (2\pi i)^m m! \trace\big( \epsilon T [T, f_0] \cdots [T, f_{2m}] \big).\]
We have the compatibility $S \circ \ch^{0,2m} = \ch^{0, 2m+2}$ and therefore we get a map
\[\ch^0\colon K^u_0(M) \dashrightarrow \HPucont^0(\Winftyone(M)).\]

The dashed arrow indicates that we do not know that every uniform, even $K$-homology class is represented by a uniformly finitely summable module, and we also do not know if the map is well-defined, i.e., if two such modules representing the same $K$-homology class will be mapped to the same cyclic cocycle class. For \spinc manifolds the first mentioned problem is solved by \Poincare duality which states that every uniform $K$-homology class may be represented by the difference of two twisted Dirac operators (which are uniformly finitely summable). But the second mentioned problem about the well-definedness is much more serious and will only be solved by the local index theorem. We will state the resolution of this problem in Corollary \ref{cor:ch_well_defined}.

Given an ungraded, involutive, uniformly $p$-summable Fredholm module $(H, \rho, T)$, we define for all $m$ with $2m \ge p$ a cyclic $(2m-1)$-cocycle on $\Winftyone(M)$ by
\begin{align*}
\ch^{1, 2m-1}(H, \rho, T & )(f_0, \ldots, f_{2m-1}) =\\
& = (2\pi i)^m \tfrac{1}{2} (2m-1) (2m-3) \cdots 3 \cdot 1 \trace\big( T [T, f_0] \cdots [T, f_{2m-1}]\big).
\end{align*}
Again, this definition is compatible with the periodicity operator $S$ and so defines a map
\[\ch^1\colon K^u_1(M) \dashrightarrow \HPucont^1(\Winftyone(M)).\]

\subsection{Uniform de Rham (co-)homology}\label{sec:u_de_Rham_currents}

\begin{defn}\label{defn:de_rham_hom}
The space of \emph{uniform de Rham $p$-currents} $\Omega_p^u(M)$ is defined as the topological dual space of the \Frechet space $W^{\infty, 1}(\Omega^p(M))$, i.e.,
\[\Omega_p^u(M) := \Hom(W^{\infty, 1}(\Omega^p(M)), \IC).\]
Recall from Definition \ref{defn:sobolev_spaces} and Equation \eqref{eq:defn_W_infty} that $W^{\infty, 1}(\Omega^p(M))$ denotes the Sobolev space of $p$-forms whose derivatives are all $L^1$-integrable.

Since the exterior derivative $d\colon W^{\infty, 1}(\Omega^p(M)) \to W^{\infty, 1}(\Omega^{p+1}(M))$ is continuous we get a corresponding dual differential (also denoted by $d$)
\begin{equation}
\label{eqjnker4}
d\colon \Omega_p^u(M) \to \Omega_{p-1}^u(M).
\end{equation}
The \emph{uniform de Rham homology} $\HudR_\ast(M)$ with coefficients in $\IC$ is defined as the homology of the complex
\[\ldots \stackrel{d}\longrightarrow \Omega_p^u(M) \stackrel{d}\longrightarrow \Omega^u_{p-1}(M) \stackrel{d}\longrightarrow \ldots \stackrel{d}\longrightarrow \Omega_0(M) \to 0,\]
where $d$ is the dual differential \eqref{eqjnker4}.
\qed
\end{defn}

\begin{defn}
We define a map $\alpha\colon \Cucont^p(\Winftyone(M)) \to \Omega_p^u(M)$ by
\[\alpha(\varphi)(f_0 d f_1 \wedge \ldots \wedge d f_p) := \frac{1}{p!} \sum_{\sigma \in \frakS_p} (-1)^\sigma \varphi(f_0, f_{\sigma(1)}, \ldots, f_{\sigma(p)}),\]
where $\frakS_p$ denotes the symmetric group on $1, \ldots, p$.
\qed
\end{defn}

The antisymmetrization that we have done in the above definition of $\alpha$ maps Hochschild cocycles to Hochschild cocycles and vanishes on Hochschild coboundaries. This means that $\alpha$ descends to a map
\[\alpha\colon \HHucont^\ast(\Winftyone(M)) \to \Omega_\ast^u(M)\]
on Hochschild cohomology.

Before we prove that $\alpha$ is an isomorphism we first prove a technical lemma:

\begin{lem}\label{lem:tensor_prod_sobolev}
Let $M$ and $N$ be manifolds of bounded geometry and without boundary. Then we have
\[\Winftyone(M) \hatotimes \Winftyone(N) \cong \Winftyone(M \times N),\]
where $\hatotimes$ denotes the projective tensor product.
\end{lem}

\begin{proof}
This proof is an elaboration of P.~Michor's answer \cite{MO_michor} on MathOverflow. The reference that he gives is \cite{michor_book}: combining the Theorem on p.~78 in it with Point (c) on top of the same page we get the isomorphism $L^1(M) \hatotimes L^1(N) \cong L^1(M \times N)$. This result was first proved by Chevet \cite{chevet}.

Now let us generalize this to incorporate derivatives. In \cite[End of Section 6]{kriegl_michor_rainer} it is proved\footnote{To be concrete, they proved it only for Euclidean space, but the argument is the same for manifolds of bounded geometry.} that we have a continuous inclusion $\Winftyone(M) \hatotimes \Winftyone(N) \to \Winftyone(M \times N)$. Note that we have to use \cite[Th\'{e}or\`{e}me 1 on p.~124]{chevet} to conclude that the family of seminorms used in \cite{kriegl_michor_rainer} for $\Winftyone(M) \hatotimes \Winftyone(N)$ generates indeed the projective tensor product topology.

It remains to show that $\Winftyone(M \times N) \to \Winftyone(M) \hatotimes \Winftyone(N)$ is continuous. For this we will use the fact that we may represent the projective tensor product norm on the algebraic tensor product $E \otimes_{\mathrm{alg}} F$ of two Banach spaces by
\begin{equation*}
\|u\|_{E \hatotimes F} = \inf \Big\{ \sum \|x_i\|_E \|y_i\|_F \Big\},
\end{equation*}
where the infimum ranges over all representations $u=\sum_i x_i \otimes y_i$. In our case now note that we have for $w := \sum_i (\nabla_X p_i) \otimes q_i$, where $X$ is a vector field on $M$ with $\|X\|_\infty \le 1$, the chain of inequalities
\begin{align}
\|w\|_{L^1(M) \hatotimes L^1(N)} & = \left\| \sum (\nabla_X p_i) \otimes q_i \right\|_{L^1(M) \hatotimes L^1(N)}\notag\\
& \le C \left\| \sum (\nabla_X p_i) \cdot q_i \right\|_{L^1(M\times N)}\notag\\
& \le C \|\sum p_i \cdot q_i\|_{W^{1,1}(M\times N)},\label{eq:proj_tensor_prod}
\end{align}
where the first inequality comes from the fact $L^1(M) \hatotimes L^1(N) \cong L^1(M \times N)$ which we already know. Now for $v := \sum_i s_i \otimes t_i$ we have
\begin{align}
\|v\|_{W^{1,1}(M) \hatotimes L^1(N)} & = \Big\| \sum s_i \otimes t_i \Big\|_{W^{1,1}(M) \hatotimes L^1(N)}\label{eq:proj_tensor_2}\\
& = \inf \Big\{ \sum \big( \|x_i\|_{L^1(M)} + \|\nabla x_i\|_{L^1(M)} \big) \|y_i\|_{L^1(N)} \Big\}\notag\\
& = \underbrace{\inf \Big\{ \sum \|x_i\|_{L^1(M)} \|y_i\|_{L^1(N)} \Big\}}_{= \|v\|_{L^1(M) \hatotimes L^1(N)} \le C\|v\|_{L^1(M \times N)}} + \inf \Big\{ \sum \|\nabla x_i\|_{L^1(M)} \|y_i\|_{L^1(N)} \Big\},\notag
\end{align}
where the infima run over all representations $\sum_i x_i \otimes y_i$ of $v$. Furthermore, for a fixed compactly supported vector field $X$ with $\|X\|_\infty \le 1$ we have
\begin{equation}
\label{eq:infima_equal}
\inf_{\mathcal{A}} \Big\{ \sum \|\nabla_X x_i\|_{L^1(M)} \|y_i\|_{L^1(N)} \Big\} = \inf_{\mathcal{B}} \Big\{ \sum \|e_i\|_{L^1(M)} \|f_i\|_{L^1(N)} \Big\},
\end{equation}
where $\mathcal{A}$ is the set of all representations $\sum_i x_i \otimes y_i$ of $v = \sum_i s_i \otimes t_i$ and $\mathcal{B}$ the set of all representations $\sum_i e_i \otimes f_i$ of $\sum_i (\nabla_X s_i) \otimes t_i$. This equality holds because every element of $\mathcal{A}$ gives rise to an element of $\mathcal{B}$ by deriving the first component and also vice versa by integrating it. By Inequality \eqref{eq:proj_tensor_prod} we now get that the infima in Equation \eqref{eq:infima_equal} are less than or equal to $C\|v\|_{W^{1,1}(M \times N)}$. Since this holds for any vector field $X$ with $\|X\|_\infty \le 1$ we can combine it now with Estimate \eqref{eq:proj_tensor_2} to get
\[\|v\|_{W^{1,1}(M) \hatotimes L^1(N)} \le 2C\|v\|_{W^{1,1}(M \times N)}.\]

We iterate the argument to get estimates for all higher derivatives and also for the second component. This proves the claim that the map $\Winftyone(M \times N) \to \Winftyone(M) \hatotimes \Winftyone(N)$ is continuous and therefore completes the whole proof.
\end{proof}

\begin{thm}
For any Riemannian manifold $M$ of bounded geometry and without boundary the map $\alpha\colon \HHucont^p(\Winftyone(M)) \to \Omega_p^u(M)$ is an isomorphism for all $p$.
\end{thm}

\begin{proof}
The proof is analogous to the one given in \cite[Lemma 45a on page 128]{connes_noncomm_diff_geo} for the case of compact manifolds. We describe here only the places where we have to adjust it for non-compact manifolds.

The proof in \cite{connes_noncomm_diff_geo} relies heavily on Lemma 44 there. First note that direct sums, tensor products and duals of vector bundles of bounded geometry are again of bounded geometry. Since the tangent and cotangent bundle of a manifold of bounded geometry have, of course, bounded geometry, the bundles $E_k$ occuring in Lemma 44 of \cite{connes_noncomm_diff_geo} have bounded geometry.

Furthermore, \cite[Lemma 44]{connes_noncomm_diff_geo} needs a nowhere vanishing vector field on $M$, and since we are working here in the bounded geometry setting we need for our proof a nowhere vanishing vector field of norm one at every point and with bounded derivatives. Since we can without loss of generality assume that our manifold is non-compact (otherwise we are in the usual setting where the result that we want to prove is already known), we can always contruct a nowhere vanishing vector field on $M$: we just pick a generic vector field with isolated zeros and then move the vanishing points to infinity. But if we normalize this vector field to norm one at every point, then it will usually have unbounded derivatives (since we moved the vanishing points infinitely far, i.e., we disturbed the derivatives arbitrarily large). Fortunately, Weinberger proved in \cite[Theorem 1]{weinberger_euler} that on a manifold $M$ of bounded geometry a nowhere vanishing vector field of norm one and with bounded derivatives exists if and only if the Euler class $e(M) \in \HbdR^m(M)$ vanishes (the latter group denotes the top-dimensional bounded de Rham cohomology of $M$; see Definition \ref{defn:bounded_de_rham_coho}). So if the euler class of $M$ vanishes, we are ok and can move on with our proof. If the euler class does not vanish, then we have to use the same trick that already Connes used to prove Lemma 45a in \cite{{connes_noncomm_diff_geo}}: we take the product with $S^1$.

Moreover, we need the isomorphism $\Winftyone(M) \hatotimes \Winftyone(M) \cong \Winftyone(M \times M)$. This is exactly the content of the above Lemma \ref{lem:tensor_prod_sobolev}.

The fact that the modules $\mathcal{M}_k =\Winftyone(M \times M, E_k)$ are topologically projective, i.e., are direct summands of topological modules of the form $\mathcal{M}^\prime_k = \Winftyone(M \times M) \hatotimes \mathcal{E}_k$, where $\mathcal{E}_k$ are complete locally convex vector spaces, follows from the fact that every vector bundle $F$ of bounded geometry is $C_b^\infty$-complemented, i.e., there is a vector bundle $G$ of bounded geometry such that $F \oplus G$ is $C_b^\infty$-isomorphic to a trivial bundle with the flat connection. This is our Proposition \ref{prop:every_bundle_complemented}.

With the above notes in mind, the proof of \cite[Lemma 45a on page 128]{connes_noncomm_diff_geo} for the case of compact manifolds works also for non-compact manifolds in our setting here. If there are constructions to be done in the proof we have to do them uniformly (e.g., controlling derivatives uniformly in the points of the manifold) by using the bounded geometry of $M$.
\end{proof}

The inverse map $\beta\colon \Omega^u_p(M) \to \HHucont^p(\Winftyone(M))$ of $\alpha$ is given by
\[\beta(C)(f_0, f_1, \ldots, f_p) = C(f_0 d f_1 \wedge \ldots \wedge d f_p).\]

Now the proofs of Lemma 45b and Theorem 46 in \cite{connes_noncomm_diff_geo} translate without change to our setting here so that we finally get:

\begin{thm}\label{thm:computation_cyclic_cohomology}
Let $M$ be a Riemannian manifold of bounded geometry and without boundary.

For each $n \in \IN_0$ the continuous cyclic cohomology $\HCucont^n(\Winftyone(M))$ is canonically isomorphic to
\[Z_n^u(M) \oplus \HudR_{n-2}(M) \oplus \HudR_{n-4}(M) \oplus \ldots,\]
where $Z_n^u(M) \subset \Omega_n^u(M)$ is the subspace of closed currents.

The periodicity operator $S\colon \HCucont^n(\Winftyone(M)) \to \HCucont^{n+2}(\Winftyone(M))$ is given under the above isomorphism as the map that sends cycles of $Z_n^u(M)$ to their homology classes.

And last, since periodic cyclic cohomology is the direct limit of cyclic cohomology, we finally get
\[\alpha_\ast\colon \HPucont^{\mathrm{ev/odd}}(\Winftyone(M)) \stackrel{\cong}\longrightarrow \HudR_{\mathrm{ev/odd}}(M).\]
We denote this isomorphism by $\alpha_\ast$ since it is induced from the map $\alpha$ defined above.
\end{thm}

Let us now get to the dual cohomology theory to uniform de Rham homology.

\begin{defn}[Bounded de Rham cohomology]
\label{defn:bounded_de_rham_coho}
Let $\Omega_b^p(M)$ denote the vector space of $p$-forms on $M$, which are bounded in the norm
\[\|\gamma\| := \sup_{x \in M} \{ \| \gamma(x) \| + \| d \gamma(x) \| \}.\]
The \emph{bounded de Rham cohomology} $\HbdR^\ast(M)$ is defined as the homology of the corresponding complex.
\qed
\end{defn}

For an oriented manifold the \Poincare duality map between bounded de Rham cohomology and uniform de Rham homology is defined as the map induced by the following map on forms:
\begin{equation}
\label{eq_PD_coho}
\Omega^p_b(M) \to \Omega^u_{m-p}(M), \ \gamma \mapsto \big( \omega \mapsto \int_M \omega \wedge \gamma \big).
\end{equation}

\begin{thm}\label{thm:Poincare_de_Rham}
Let $M^m$ be an oriented Riemannian manifold of bounded geometry and without boundary.

Then the \Poincare duality map \eqref{eq_PD_coho} induces an isomorphism
\[\HbdR^\ast(M) \stackrel{\cong}\longrightarrow \HudR_{m-\ast}(M)\]
between bounded de Rham cohomology of $M$ and uniform de Rham homology of $M$.
\end{thm}

\begin{proof}
We will do a Mayer--Vietoris induction, similar as in the proof of \Poincare duality between uniform $K$-theory and uniform $K$-homology in Section~\ref{sec:poincare_duality}.

We invoke Lemma~\ref{lem:suitable_coloring_cover_M} to get a cover of $M$ by open subsets having Properties~\ref{cover_i} and~\ref{cover_ii} from that lemma.\footnote{With the additional property that the boundaries of the open subsets are smooth, but it is clear that we can arrange this.} We use the notation $U_j$ and $U_K$ from it, and the induction will be over the index $j$ (and hence the proof will only consist of finitely many induction steps).

We have to show that we have the Mayer--Vietoris sequences. The arguments are the same as in the case of compact manifolds, and we will only mention where we have to be cautios because we are working in the setting of uniform theories. We will only discuss the case of bounded de Rham cohomology, since the additional arguments (because of the uniform situation) in the case of uniform de Rham homology are similar.

For bounded de Rham cohomology we have to show that the following sequence is exact in order to get a Mayer--Vietoris sequence:
\begin{equation}
\label{eq_MV_de_Rham}
0 \to \Omega_{b}^*(U_K \cup U_{k+1}) \to \Omega_{b}^*(U_K) \oplus \Omega_{b}^*(U_{k+1}) \to \Omega_{b}^*(U_K \cap U_{k+1}) \to 0.
\end{equation}
The crucial step is to show that the map $\Omega_{b}^*(U_K) \oplus \Omega_{b}^*(U_{k+1}) \to \Omega_{b}^*(U_K \cap U_{k+1})$ is surjectice. The usual argument in the case of compact manifolds uses a partition of unity, and here we have to make sure now that the partition of unity has uniformly bounded derivatives of all orders. The reason that we can construct such a partition of unity here is because of Property~\ref{cover_ii} of Lemma~\ref{lem:suitable_coloring_cover_M}.

That the above defined \Poincare duality map \eqref{eq_PD_coho} is a natural transformation from one Mayer--Vietoris sequence to the other may be proved analogously as in the case of compact manifolds; see, e.g., \cite[Exercise 16-6]{lee_smooth}.

And finally, let us discuss the first step of the induction. We have collections $U_1$, $U_2$ and $U_1 \cap U_2$ which are each a uniformly disjoint union of open subsets of $M$ which have a uniform bound on their diameters. So all three sets are boundedly homotopy equivalent\footnote{\label{footnote:boundedly_homotopic}Let $f, g\colon M \to N$ be two maps of bounded dilatation. We say that they are \emph{boundedly homotopic}, if there is a homotopy $H\colon M \times [0,1] \to N$ from $f$ to $g$, which itself is of bounded dilatation. Recall that a map $h$ has \emph{bounded dilatation}, if $\|h_\ast V\| \le C \|V\|$ for all tangent vectors $V$. Bounded homotopy invariance of bounded de Rham cohomology was shown by the author in \cite[Corollary 5.26]{engel_phd}.} to an infinite collection of open balls, for which we already know from the case of compact manifolds that the \Poincare duality map is an isomorphism.
\end{proof}

Bounded de Rham cohomology does not perfectly fit the setting in this paper since the condition that the exterior derivative of a form is bounded does not imply that in local coordinates the coefficient functions have a uniformly bounded first derivative, and it also does not say anything about the higher derivatives. Hence the following definition and proposition.

\begin{defn}
The \emph{uniform de Rham cohomology $\HudRco^\ast(M)$} of a Riemannian manifold $M$ of bounded geometry is defined by using the complex of uniform $C^\infty$-spaces\footnote{see Definition \ref{defn:uniform_frechet_spaces}} $C^\infty_b(\Omega^\ast(M))$, i.e., differential forms on $M$ which have in normal coordinates bounded coefficient functions and all derivatives of them are also bounded.
\end{defn}

\begin{prop}\label{prop:uniform_bounded_dRco}
Let $M$ be a manifold of bounded geometry and without boundary. Then we have
\[\HudRco^\ast(M) \cong \HbdR^\ast(M).\]
\end{prop}

\begin{proof}
The proof is analogous to the one of Theorem~\ref{thm:Poincare_de_Rham} --- the important thing is that we have Mayer--Vietoris sequences and the argument given in the proof of Theorem~\ref{thm:Poincare_de_Rham} for bounded de Rham cohomology also applies to uniform de Rham cohomology.
\end{proof}

\begin{thm}[Existence of the uniform Chern character]
Let $M$ be a Riemannian manifold of bounded geometry and without boundary.

Then we have a ring homomorphism $\ch \colon K^\ast_u(M) \to \HudRco^\ast(M)$ with
\[\ch(K^0_u(M)) \subset \HudRco^\ev(M) \text{ and } \ch(K^1_u(M)) \subset \HudRco^\odd(M).\]
\end{thm}

\begin{proof}
The Chern character is defined via Chern--Weil theory. That we get uniform forms if we use vector bundles of bounded geometry is proved in \cite[Theorem 3.8]{roe_index_1} and so we get a map $\ch \colon K^0_u(M) \to \HudRco^\ev(M)$. That we also have a map $\ch \colon K^1_u(M) \to \HudRco^\odd(M)$ uses the description of $K^1_u(M)$ as consisting of vector bundles over $S^1 \times M$ and a corresponding suspension isomorphism for uniform de Rham cohomology. Details (for bounded cohomology, but for uniform cohomology it is analogous) may be found in the author's Ph.D.\ thesis \cite[Sections 5.4 \& 5.5]{engel_phd}.
\end{proof}

\subsection{Uniform Chern character isomorphism theorems}\label{sec:chern_isos}

The results of the last two sections tell us that we have constructed Chern characters $K^\ast_u(M) \to \HudRco^\ast(M)$ and $K_\ast^u(M) \to \HudR_\ast(M)$. Here we already use Corollary~\ref{cor:ch_well_defined} further below which states that the uniform homological Chern character is well-defined. In the compact case the Chern characters are isomorphisms modulo torsion and it is natural to ask the same question here in the uniform setting. It is the goal of this section to answer this question positively.

The proofs use the same Mayer--Vietoris induction as the proof of \Poincare duality in Section \ref{sec:poincare_duality}. Therefore we will discuss in this section only the parts of the proofs which need additional arguments.

The most crucial detail to discuss here is the statement of the theorem itself since we cannot just take the tensor product of the $K$-groups with the complex numbers to get isomorphisms. In turns out that we additionally have to form a certain completion of the algebraic tensor product of the $K$-groups with $\IC$. We will discuss this completion directly after the statement of the theorem.

\begin{thm}\label{thm23w2}
Let $M$ be a manifold of bounded geometry and without boundary. Then the Chern characters induce linear, continuous isomorphisms\footnote{The inverse maps are in general not continuous since $\HudRco^\ast(M)$, resp., $\HudR_\ast(M)$, are in general (e.g., if $M$ is not compact) not Hausdorff, whereas $K_u^\ast(M) \barotimes \IC$, resp. $K^u_\ast(M) \barotimes \IC$, are. The topology on the latter spaces is defined by equipping the $K$-groups with the discrete topology and then forming the completed tensor product with $\IC$ which will be discussed after the statement of the theorem.}
\[K_u^\ast(M) \barotimes \IC \cong \HudRco^\ast(M) \text{ and } K^u_\ast(M) \barotimes \IC \cong \HudR_\ast(M).\]
\end{thm}

Let us discuss why we have to take a completion at all. Consider the beginning of the Mayer--Vietoris induction where we have to show that the Chern characters induce isomorphisms on a countably infinite collection of uniformly discretely distributed points. Let these points be indexed by a set $Y$. Then the $K$-groups of $Y$ are given by $\ell^\infty_\IZ(Y)$, the group of all bounded, integer-valued sequences indexed by $Y$, and the de Rham groups are given by $\ell^\infty(Y)$, the group of all bounded, complex valued sequences on $Y$. But since $Y$ is countably infinite we have $\ell^\infty_\IZ(Y) \otimes \IC \not\cong \ell^\infty(Y)$. Instead we have $\overline{\ell^\infty_\IZ(Y) \otimes \IC} \cong \ell^\infty(Y)$.

To define the \emph{completed topological tensor product of an abelian group with $\IC$} we will need the notion of the \emph{free (abelian) topological group}: if $X$ is any completely regular\footnote{That is to say, every closed set $K$ can be separated with a continuous function from every point $x \notin K$. Note that this does not necessarily imply that $X$ is Hausdorff.} topological space, then the free topological group $F(X)$ on $X$ is a topological group such that we have
\begin{itemize}
\item a topological embedding $X \hookrightarrow F(X)$ of $X$ as a closed subset, so that $X$ generates $F(X)$ algebraically as a free group (i.e., the algebraic group underlying the free topological group on $X$ is the free group on $X$), and we have
\item the following universal property: for every continuous map $\phi\colon X \to G$, where $G$ is any topological group, we have a unique extension $\Phi\colon F(X) \to G$ of $\phi$ to a continuous group homomorphism on $F(X)$:
\[\xymatrix{X \ar@{^{(}->}[r] \ar[d]_{\phi} & F(X) \ar@{-->}[dl]^{\exists ! \Phi}\\ G}\]
\end{itemize}
The free abelian topological group $A(X)$ has the corresponding analogous properties. Furthermore, the commutator subgroup $[F(X), F(X)]$ of $F(X)$ is closed and the quotient $F(X) / [F(X), F(X)]$ is both algebraically and topologically $A(X)$.

As an easy example consider $X$ equipped with the discrete topology. Then $F(X)$ and $A(X)$ also have the discrete topology.

It seems that free (abelian) topological groups were apparently introduced by Markov in \cite{markoff_original}. But unfortunately, the author could not obtain any (neither russian nor english) copy of this article. A complete proof of the existence of such groups was given by Markov in \cite{markoff}. Since his proof was long and complicated, several other authors gave other proofs, e.g., Nakayama in \cite{nakayama}, Kakutani in \cite{kakutani} and Graev in \cite{graev}.

Now let us construct for any abelian topological group $G$ the complete topological vector space $G \barotimes \IC$. We form the topological tensor product $G \otimes \IC$ of abelian topological groups in the usual way: we start with the free abelian topological group $A(G \times \IC)$ over the topological space $G \times \IC$ equipped with the product topology\footnote{Note that every topological group is automatically completely regular and therefore the product $G \times \IC$ is also completely regular.} and then take the quotient $A(G \times \IC) / \mathcal{N}$ of it,\footnote{Since $A(X)$ is both algebraically and topologically the quotient of $F(X)$ by its commutator subgroup, we could also have started with $F(G \times \IC)$ and additionally put the commutator relations into $\mathcal{N}$.} where $\mathcal{N}$ is the closure of the normal subgroup generated by the usual relations for the tensor product.\footnote{That is to say, $\mathcal{N}$ contains $(g_1 + g_2) \times r - g_1 \times r - g_2 \times r$, $g \times (r_1 + r_2) - g \times r_1 - g \times r_2$ and $zg \times r - z(g \times r)$, $g \times zr - z(g \times r)$, where $g, g_1, g_2 \in G$, $r, r_1, r_2 \in \IC$ and $z \in \IZ$.} Now we may put on $G \otimes \IC$ the structure of a topological vector space by defining the scalar multiplication to be $\lambda (g \otimes r) := g \otimes \lambda r$.

What we now got is a topological vector space $G \otimes \IC$ together with a continuous map $G \times \IC \to G \otimes \IC$ with the following universal property: for every continuous map $\phi\colon G \times \IC \to V$ into any topological vector space $V$ and such that $\phi$ is bilinear\footnote{That is to say, $\phi(\largecdot, r)$ is a group homomorphism for all $r \in \IC$ and $\phi(g, \largecdot)$ is a linear map for all $g \in G$. Note that we then also have $\phi(zg,r) = z\phi(g,r) = \phi(g,zr)$ for all $z \in \IZ$, $g \in G$ and $r \in \IC$.\label{footnote:univ_prop_TVS}}, there exists a unique, continuous linear map $\Phi\colon G \otimes \IC \to V$ such that the following diagram commutes:
\[\xymatrix{G \times \IC \ar[r] \ar[d]_{\phi} & G \otimes \IC \ar@{-->}[dl]^{\exists ! \Phi} \\ V}\]

Since every topological vector space may be completed we do this with $G \otimes \IC$ to finally arrive at $G \barotimes \IC$. Since every continuous linear map of topological vector spaces is automatically uniformly continuous, i.e., may be extended to the completion of the topological vector space, $G \barotimes \IC$ enjoys the following universal property which we will raise to a definition:

\begin{defn}[Completed topological tensor product with $\IC$]
Let $G$ be an abelian topological group. Then $G \barotimes \IC$ is a complete topological vector space over $\IC$ together with a continuous map $G \times \IC \to G \barotimes \IC$ that enjoy the following universal property: for every continuous map $\phi\colon G \times \IC \to V$ into any complete topological vector space $V$ and such that $\phi$ is bilinear\footnote{see Footnote \ref{footnote:univ_prop_TVS}}, there exists a unique, continuous linear map $\Phi\colon G \barotimes \IC \to V$ such that
\[\xymatrix{G \times \IC \ar[r] \ar[d]_{\phi} & G \barotimes \IC \ar@{-->}[dl]^{\exists ! \Phi} \\ V}\]
is a commutative diagram.
\qed
\end{defn}

We will give now two examples for the computation of $G \barotimes \IC$. The first one is easy and just a warm-up for the second which we already mentioned. Both examples are proved by checking the universal property.

\begin{examples}\label{ex:completed_tensor_prod}
The first one is $\IZ \barotimes \IC \cong \IC$.

For the second example consider the group $\ell^\infty_\IZ$ consisting of bounded, integer-valued sequences. Then $\ell^\infty_\IZ \barotimes \IC \cong \ell^\infty$.
\qed
\end{examples}

Since we want to use the completed topological tensor product with $\IC$ in a Mayer--Vietoris argument, we have to show that it transforms exact sequences to exact sequences.

So we have to show that the functor $G \mapsto G \barotimes \IC$ is exact. But we have to be careful here: though taking the tensor product with $\IC$ is exact, passing to completions is usually not---at least if the exact sequence we started with was only algebraically exact. Let us explain this a bit more thoroughly: if we have a sequence of topological vector spaces
\[\ldots \longrightarrow V_i \stackrel{\varphi_i}\longrightarrow V_{i+1} \stackrel{\varphi_{i+1}}\longrightarrow V_{i+2} \longrightarrow \ldots\]
which is exact in the algebraic sense (i.e., $\image \varphi_i = \kernel \varphi_{i+1}$), and if the maps $\varphi_i$ are continuous such that they extend to maps on the completions $\overline{V_i}$, we do not necessarily get that
\[\ldots \longrightarrow \overline{V_i} \stackrel{\overline{\varphi_i}}\longrightarrow \overline{V_{i+1}} \stackrel{\overline{\varphi_{i+1}}}\longrightarrow \overline{V_{i+2}} \longrightarrow \ldots\]
is again algebraically exact. The problem is that though we always have $\overline{\kernel \varphi_i} = \kernel \overline{\varphi_i}$, we generally only get $\overline{\image \varphi_i} \supset \image \overline{\varphi_i}$. To correct this problem we have to start with an exact sequence which is also topologically exact, i.e., we need that not only $\image \varphi_i = \kernel \varphi_{i+1}$, but we also need that $\varphi_i$ induces a topological isomorphism $V_i / \kernel \varphi_i \cong \image \varphi_i$.

To prove that in this case we get $\overline{\image \varphi_i} = \image \overline{\varphi_i}$ we consider the inverse map
\[\psi_i := \varphi_i^{-1}\colon \image \varphi_i \to V_i / \kernel \varphi_i.\]
Since $\psi_i$ is continuous (this is the point which breaks down without the additional assumption that $\varphi_i$ induces a topological isomorphism $V_i / \kernel \varphi_i \cong \image \varphi_i$), we may extend it to a map
\[\overline{\psi_i} \colon \overline{\image \varphi_i} \to \overline{V_i / \kernel \varphi_i} = \overline{V_i} / \overline{\kernel \varphi_i},\]
which obviously is the inverse to $\overline{\varphi_i} \colon \overline{V_i} / \overline{\kernel \varphi_i} \to \overline{\image \varphi_i}$ showing the desired equality $\overline{\image \varphi_i} = \image \overline{\varphi_i}$.

Coming back to our functor $G \mapsto G \barotimes \IC$, we may now prove the following lemma:

\begin{lem}\label{lem:functor_exact}
Let
\[\ldots \longrightarrow G_i \stackrel{\varphi_i}\longrightarrow G_{i+1} \stackrel{\varphi_{i+1}}\longrightarrow G_{i+2} \longrightarrow \ldots\]
be an exact sequence of topological groups and continuous maps, which is in addition topologically exact, i.e., for all $i \in \IZ$ the from $\varphi_i$ induced map $G_i / \kernel \varphi_i \to \image \varphi_i$ is an isomorphism of topological groups.

Then
\[\ldots \longrightarrow G_i \barotimes \IC \longrightarrow G_{i+1} \barotimes \IC \longrightarrow G_{i+2} \barotimes \IC \longrightarrow \ldots\]
with the induced maps is an exact sequence of complete topological vector spaces, which is also topologically exact.
\end{lem}

\begin{proof}
We first tensor with $\IC$ (without the completion afterwards). This is known to be an exact functor and our sequence also stays topologically exact. To see this last claim, we need the following fact about tensor products: if $\varphi\colon M \to M^\prime$ and $\psi\colon N \to N^\prime$ are surjective, then the kernel of $\varphi \otimes \psi \colon M \otimes M^\prime \to N \otimes N^\prime$ is the submodule given by
\[\kernel (\varphi \otimes \psi) = (\iota_M \otimes 1) \big( (\kernel \varphi) \otimes N \big) + (1 \otimes \iota_N) \big( M \otimes (\kernel \psi) \big),\]
where $\iota_M \colon \kernel \varphi \to M$ and $\iota_N \colon \kernel \psi \to N$ are the inclusion maps. We will suppress the inclusion maps from now on to shorten the notation.

We apply this with the map $\varphi \colon M \to M^\prime$ being the quotient map $G_i \to G_i / \kernel \varphi_i$ and $\psi \colon N \to N^\prime$ being the identity $\id \colon \IC \to \IC$ to get
\[\kernel (\varphi_i \otimes \id) = (\kernel \varphi_i) \otimes \IC.\]
Since we have $(\image \varphi_i) \otimes \IC = \image (\varphi_i \otimes \id)$, we get that $\varphi \otimes \id \colon G_i \otimes \IC \to G_i \otimes \IC$ induces an algebraic isomorphism $(G_i / \kernel \varphi_i) \otimes \IC \to \image \varphi_i \otimes \IC$. But this has now an inverse map given by tensoring the inverse of $G_i / \kernel \varphi_i \to \image \varphi_i$ with $\id \colon \IC \to \IC$. So the isomorphism $(G_i / \kernel \varphi_i) \otimes \IC \cong \image \varphi_i \otimes \IC$ is also topological.

Now we apply the discussion before the lemma to show that the completion of this new sequence is still exact and also topologically exact.
\end{proof}

To show $K_u^\ast(M) \barotimes \IC \cong \HudRco^\ast(M)$ it remains to construct Mayer--Vietoris sequences 
with continuous maps in them (we need this since in constructing the completed tensor product with $\IC$ we have to pass to the completion and without continuity of the maps in both the Mayer--Vietoris sequences for uniform $K$-theory and for uniform de Rham cohomology we would not be able to conclude that the squares are still commutative). If we recall from the proof of Proposition~\ref{prop:uniform_bounded_dRco} how we get the boundary maps in the Mayer--Vietoris sequence for uniform de Rham cohomology, we see that we must construct a continuous split to the last non-trivial map in the sequence \eqref{eq_MV_de_Rham}.\footnote{The referenced sequence is for bounded de Rham cohomology. In this proof here we, of course, have to use the analogous sequence for uniform de Rham cohomology.} But we proved surjectivity of this map in the usual way by using partitions of unity (with uniformly bounded derivatives). Hence we have already constructed the continuous split.

So we get a Mayer--Vietoris sequence with continuous maps for uniform de Rham cohomology as needed. Now we have to discuss the existence of the Chern character $K^\ast_u(O \subset M) \to \HudRco^\ast(O)$. Recall from Definition \ref{defn:uniform_k_th_subset} that we defined $K^\ast_u(O \subset M)$ as $K_{-\ast}(C_u(O,d))$, where $(O,d)$ is the metric space $O$ equipped with the subspace metric derived from the metric space $M$. But for the definition of the Chern character we have to pass to a smooth subalgebra of $C_u(O,d)$. This will be of course $C_b^\infty(O) \subset C_u(O,d)$, which is a local $C^\ast$-algebra. It remains to argue why it is a dense subalgebra, because the argument from the proof of Lemma \ref{lem:norm_completion_C_b_infty} does not work for $O$. So let $f \in C_u(O,d)$ be given. Then we know from Lemma \ref{lem:extension_samuel_compactification} that there is a bounded, uniformly continuous extension $F$ of $f$ to all of $M$. And now we use Lemma \ref{lem:norm_completion_C_b_infty} to approximate $F$ by functions from $C_b^\infty(M)$, which will give us by restriction to $O$ an approximation of $f$ by functions from $C_b^\infty(O)$. So we get an interpretation of $K^\ast_u(O \subset M)$ by vector bundles of bounded geometry over $O$ and may define the Chern character $K^\ast_u(O \subset M) \to \HudRco^\ast(O)$.

The last thing that we have to discuss is the small ambiguity in extending the maps $K^\ast_u(O \subset M) \otimes \IC \to \HudRco^\ast(O)$ to $K^\ast_u(O \subset M) \barotimes \IC$. It occurs because the target $\HudRco^\ast(O)$ is not necessarily Hausdorff. What we have to make sure is that the extensions we choose in the Mayer--Vietoris argument for the subsets $U_k$, resp. $U_K$, do match up, i.e., produce at the end commuting squares in the comparison of the two Mayer--Vietoris sequences via the Chern characters.

So we have finally discussed everything that we need in order to prove
\[K_u^\ast(M) \barotimes \IC \cong \HudRco^\ast(M).\]

Proving the homological version $K^u_\ast(M) \barotimes \IC \cong \HudR_\ast(M)$ is also such a Mayer--Vietoris argument. But for \spinc manifolds there is an easier argument by combining the cohomological result $K_u^\ast(M) \barotimes \IC \cong \HudRco^\ast(M)$ with Theorem \ref{thm:local_thm_twisted} by noting that taking the wedge product with $\ind(D)$ is an isomorphism on bounded de Rham cohomology, and furthermore using \Poincare duality between uniform $K$-theory and uniform $K$-homology (Theorem \ref{thm:Poincare_duality_K}), resp., between bounded de Rham cohomology and uniform de Rham homology (Theorem \ref{thm:Poincare_de_Rham}).

\subsection{Local index formulas}\label{sec:local_index_thm}

In this section we assemble everything that we had up to now into local index theorems.

Let $M$ be a Riemannian manifold without boundary. We denote by $DM$ the disk bundle $\{\xi \in T^\ast M\colon \|\xi\|\le 1\}$ of its cotangent bundle and by $SM = \partial DM$ its boundary, i.e., $SM = \{\xi \in T^\ast M \colon \|\xi\| = 1\}$. If $M$ has bounded geometry, we may equip $DM$ with a Riemannian metric such that it also becomes of bounded geometry\footnote{Though we do not have defined bounded geometry for manifolds with boundary, there is an obvious one (demanding bounds not only for the curvature tensor of $M$ but also for the second fundamental form of the boundary of $M$, and demanding the injectivity radius being uniformly positive not only for $M$ but also for $\partial M$ with the induced metric). See \cite{schick_bounded_geometry_boundary} for a further discussion.} and $DM \to M$ becomes a Riemannian submersion. It follows that $SM$ will also have bounded geometry. What follows will be independent of the concrete choice of metric on $DM$. Though we have discussed in Section \ref{sec:uniform_k_th} only uniform $K$-theory for manifolds without boundary, one can of course define more generally relative uniform $K$-theory and discuss it for manifolds with boundary and of bounded geometry.

Let $P \in \UPsiDO^k(E)$ be a symmetric, elliptic and graded uniform pseudodifferential operator. Recall from Definition \ref{defn:elliptic_symbols} of ellipticity that the principal symbol $\sigma(P^+)$, viewed as a section of $\Hom(\pi^\ast E^+, \pi^\ast E^-) \to T^\ast M$, where $\pi\colon T^\ast M \to M$ is the cotangent bundle, is invertible outside a uniform neighbourhood of the zero section $M \subset T^\ast M$ and satisfies a certain uniformity condition. Then the well-known clutching construction gives us the following symbol class of $P$:
\[ \sigma_P := [\pi^\ast E^+, \pi^\ast E^-; \sigma(P) ] \in K_u^0(DM, SM).\]

If $P$ is ungraded, then its symbol $\sigma(P)\colon \pi^\ast E \to \pi^\ast E$, where $\pi\colon SM \to M$ denotes now the unit sphere bundle of $M$, is a uniform, self-adjoint automorphism. So it gives a direct sum decomposition $\pi^\ast E = E^+ \oplus E^-$, where $E^+$ and $E^-$ are spanned fiberwise by the eigenvectors belonging to the positive, resp. negative, eigenvalues of $\sigma(P)$, and we get an element
\[[E^+] \in K_u^0(SM).\]
Now we define in the ungraded case the symbol class of $P$ as
\[\sigma_P := \delta[E^+] \in K_u^1(DM,SM),\]
where $\delta\colon K_u^0(SM) \to K_u^1(DM,SM)$ is the boundary homomorphism of the $6$-term exact sequence associated to $(DM,SM)$. References for this construction in the compact case are, e.g., \cite[Section 24]{baum_douglas} and \cite[Proposition 3.1]{atiyah_patodi_singer_3}.

Applying the Chern character and integrating over the fibers we get in both the graded and ungraded case $\pi_! \ch \sigma_P \in \HbdR^\ast(M)$ and then the index class of $P$ is defined as
\[\ind(P) := (-1)^{\frac{n(n+1)}{2}} \pi_! \ch \sigma_P \wedge \operatorname{Td}(M) \in \HbdR^\ast(M),\]
where $n = \dim M$.

Let $M$ be a \spinc manifold of bounded geometry and let us denote by $D$ the Dirac operator associated to the \spinc structure of $M$. Note that it is $m$-multigraded, where $m$ is the dimension of the manifold $M$, and so defines an element in $K_m^u(M)$. Therefore cap product with $D$ is a map $K_u^\ast(M) \to K^u_{m-\ast}(M)$, which is an isomorphism (as we have shown in Section \ref{sec:poincare_duality}). We also have \Poincare duality $\HbdR^\ast(M) \to \HudR_{m-\ast}(M)$, and the content of our local index theorem for uniform twisted Dirac operators is to put these duality maps into a commutative diagram using the homological Chern character on the right hand side and on the cohomology side the index class of the twisted operator.

\begin{thm}[Local index theorem for twisted uniform Dirac operators]\label{thm:local_thm_twisted}
Let $M$ be an $m$-dimensional \spinc manifold of bounded geometry and without boundary. Denote the associated Dirac operator by $D$.

Then we have the following commutative diagram:
\[\xymatrix{
K^\ast_u(M) \ar[r]^-{\largecdot \cap [D]}_-{\cong} \ar[d]_-{\ch(\largecdot) \wedge \ind(D)} & K_{m-\ast}^u(M) \ar[d]^-{\alpha_\ast \circ \ch^\ast}\\
\HbdR^\ast(M) \ar[r]_-{\cong} & \HudR_{m-\ast}(M)}\]
where in the top row $\ast$ is either $0$ or $1$ and in the bottom row $\ast$ is either $\ev$ or $\odd$.
\end{thm}

\begin{proof}
This follows from the calculations carried out by Connes and Moscovici in their paper \cite[Section 3]{connes_moscovici} by noting that the computations also apply in our case where we have bounded geometry and the uniformity conditions. Note that there the cyclic cocycles are defined using expressions in the operators $e^{-tD^2}$. To translate to the definition of the homological Chern character that we use, see, e.g., \cite[Section 10.2]{GBVF}.
\end{proof}

\begin{rem}\label{rem:ch_well_defined_spinc}
The uniform homological Chern character $\alpha_\ast \circ \ch^\ast\colon K_\ast^u(M) \dashrightarrow \HudR_\ast(M)$ is a priori not well-defined (to be more precise, it is defined on uniformly finitely summable Fredholm modules and it is a priori not clear whether it descends to classes and even whether every class may be represented by a uniformly finitely summable module). But using \Poincare duality between uniform $K$-homology and uniform $K$-theory and the above local index theorem, we see that it is a posteriori well-defined for \spinc manifolds. Note that since $D$ is a Dirac operator, it defines a uniformly finitely summable Fredholm module, and therefore also all its twists given by taking the cap product with uniform $K$-theory classes are uniformly finitely summable.

That the uniform homological Chern character is well-defined for every manifold $M$ of bounded geometry is content of Corollary \ref{cor:ch_well_defined}.
\qed
\end{rem}

Let $P \in \UPsiDO^k(E)$ be a symmetric and elliptic uniform pseudodifferential operator over an oriented manifold $M$ of bounded geometry. By Theorem \ref{thm:elliptic_symmetric_PDO_defines_uniform_Fredholm_module} the operator $P$ defines a uniform $K$-homology class $[P] \in K_\ast^u(M)$ and therefore, if $P$ is in addition uniformly finitely summable, we may compare the class $(\alpha_\ast \circ \ch^\ast)(P) \in \HudR_\ast(M)$ with $\ind(P) \in \HbdR^\ast(M)$ using \Poincare duality. That they are equal is the content of the next theorem.

\begin{thm}[Local index formula for uniform pseudodifferential operators]\label{thm:local_thm_pseudos}
Let $M$ be an oriented Riemannian manifold of bounded geometry and without boundary.

Let $P \in \UPsiDO^k(E)$ be a symmetric and elliptic uniform pseudodifferential operator of positive order acting on a vector bundle $E \to M$ of bounded geometry, and let $P$ be uniformly finitely summable\footnote{This means that $P$ defines a uniformly finitely summable Fredholm module, i.e., $\chi(P)$ is uniformly finitely summable for some normalizing function $\chi$.}.

Then $\ind(P) \in \HbdR^\ast(M)$ is the \Poincare dual of $(\alpha_\ast \circ \ch^\ast)(P) \in \HudR_\ast(M)$.
\end{thm}

\begin{proof}
This follows from the above Theorem \ref{thm:local_thm_twisted} by the same arguments as in the proof of \cite[Theorem 3.9]{connes_moscovici}: if $M$ is odd-dimensional we take the product with $S^1$, and then we use the fact that for oriented, even-dimensional manifolds uniform $K$-homology is spanned modulo $2$-torsion by generalized signature operators. This last fact will follow from Theorem \ref{thm:thom} below.
\end{proof}

\begin{cor}\label{cor:ch_well_defined}
The uniform homological Chern character
\[\alpha_\ast \circ \ch^\ast \colon K_\ast^u(M) \to \HudR_\ast(M)\]
is well-defined for every manifold $M$ of bounded geometry and without boundary.
\end{cor}

\begin{proof}
If $M$ is \spinc we know by \Poincare duality that every class $[x] \in K_\ast^u(M)$ may be represented by a uniformly finitely summable Fredholm module and by the above Theorem \ref{thm:local_thm_pseudos} we conclude that $(\alpha_\ast \circ \ch^\ast)([x])$ is independent of the concrete choice of such a representative. (This was already mentioned in Remark \ref{rem:ch_well_defined_spinc}.)

In the general case we first pass to the orientation cover $X$ if $M$ is not orientable. Note that if we know the statement that we want to prove for a finite covering of $M$, then we know it also for $M$ itself since $\HudR_\ast(M)$ is a vector space over $\IC$ (i.e., multiplication by some non-zero number is an isomorphism). Now we can go on as in the proof of Theorem \ref{thm:local_thm_pseudos}: we take the product with $S^1$ if necessary and then use the fact that on oriented, even-dimensional manifolds we can represent every uniform $K$-homology class by a multiple (concretely, $2^{\dim(M)/2}$) of a generalized signature operator. For the latter statement see Theorem \ref{thm:thom}, resp., its proof.
\end{proof}

\begin{rem}
The condition in the above Theorem \ref{thm:local_thm_pseudos} that $P$ has to be uniformly finitely summable may actually be dropped. The statement then is that $(\alpha_\ast \circ \ch^\ast)([P])$ is the dual of the class $\ind(P) \in \HbdR^\ast(M)$. This makes sense since we now know that the uniform homological Chern character $K_\ast^u(M) \to \HudR_\ast(M)$ is well-defined. But the problem now is that in order to compute $(\alpha_\ast \circ \ch^\ast)([P])$ we would have to replace $P$ by some other operator $P^\prime$ which defines the same uniform $K$-homology class as $P$ but which is uniformly finitely summable (so that we may compute the Chern--Connes character). This seems to be a task which is not easily carried out in practice.

Connes and Moscovici work in \cite{connes_moscovici} with so-called \emph{$\theta$-summable Fredholm modules} which are more general than finitely summable modules. So defining an appropriate version of \emph{uniformly $\theta$-summable Fredholm modules} we could certainly prove the above Theorem \ref{thm:local_thm_pseudos} for them and therefore weakening the condition on $P$ that it has to be uniformly finitely summable.
\qed
\end{rem}

Let us state now the Thom isomorphism theorem in the form that we need for the proof of the above Theorem \ref{thm:local_thm_pseudos}.

\begin{thm}[Thom isomorphism]\label{thm:thom}
Let $M$ be a Riemannian \spinc manifold of bounded geometry and without boundary.

Then the principal symbol of the Dirac operator associated to the \spinc structure of $M$ constitutes an orientation class in $K_u^\ast(DM, SM)$, i.e., it implements the isomorphism $K_u^\ast(M) \cong K_u^\ast(DM, SM)$.

If $M$ is only oriented (i.e., not necessarily spin$^c$) and even-dimensional, the principal symbol of the signature operator of $M$ constitutes an orientation class in $K_u^\ast(DM, SM)[\tfrac{1}{2}]$.
\end{thm}

\begin{proof}
The usual proof as found in, e.g., \cite[Appendix C]{lawson_michelsohn}, works in our case analogously. Note that for the proof of \cite[Theorem C.7]{lawson_michelsohn} we have to cover $M$ by such subsets as we used in our proof of \Poincare duality (see Lemma \ref{lem:suitable_coloring_cover_M}) since only in this case we have shown that we have a Mayer--Vietoris sequence for uniform $K$-theory. For the statement for only oriented $M$ see, e.g., the proof of \cite[Theorem C.12]{lawson_michelsohn}.
\end{proof}

In \cite[Theorem 3.9]{connes_moscovici} the local index theorem was written using an index pairing with compactly supported cohomology classes. We can of course do the same also here in our uniform setting and the statement is at first glance the same.\footnote{Remember that we have another choice of universal constants than Connes and Moscovici, i.e., in our statement they are not written since they are incorporated in the definition of the homological Chern character.} But the difference is that due to the uniformness we have an additional continuity statement.

\begin{cor}\label{cor:pairing_compactly_supported}
Let $[\varphi] \in H_{c, \mathrm{dR}}^k(M)$ be a compactly supported cohomology class and define the analytic index $\ind_{[\varphi]}(P)$ as in \cite{connes_moscovici}.\footnote{Note that $\ind_{[\varphi]}(P)$ is analytically defined and may be computed (up to the universal constant that we have incorporated into the definition of $\alpha_\ast \circ \ch^\ast$) as $\langle (\alpha_\ast \circ \ch^\ast)(P), [\varphi] \rangle$, where $\langle \largecdot, \largecdot \rangle$ is the pairing between uniform de Rham homology and compact supported cohomology.} Then we have
\[\ind_{[\varphi]}(P) = \int_M \ind(P) \wedge [\varphi]\]
and this pairing is continuous, i.e., $\int_M \ind(P) \wedge [\varphi] \le \|\ind(P) \|_\infty \cdot \| [\varphi] \|_1$, where $\| \largecdot \|_\infty$ denotes the sup-seminorm on $\HbdR^{m-k}(M)$ and $\| \largecdot \|_1$ the $L^1$-seminorm on $H_{c, \mathrm{dR}}^k(M)$.
\end{cor}

\begin{proof}
The corollary follows from Theorem \ref{thm:local_thm_pseudos} (if $M$ is not orientable then we first have to pass to the orientation cover of it). The continuity statement follows from the definition of the seminorms. The only thing we have to know is that $\ind(P)$ is given by a bounded de Rham form.
\end{proof}

\begin{rem}\label{rem:local_pairing_cont}
Though it may seem that the above corollary is in some sense equivalent to Theorem \ref{thm:local_thm_pseudos}, it is in fact not. It is weaker in the following way: in the case of a non-compact manifold $M$ the bounded de Rham cohomology $\HbdR^\ast(M)$ usually contains elements of seminorm $=0$ and due to the boundedness of the above pairing we see that we can not detect these elements by it.
\qed
\end{rem}

\subsection{Index pairings on amenable manifolds}

In the last section we proved the local index theorems for uniform operators. The goal of this section is to use these local formulas to compute certain global indices of such operators over amenable manifolds.

So in this section we assume that our manifold $M$ is \emph{amenable}, i.e., that it admits a \Folner sequence. We will need such a sequence in order to construct the index pairings.

\begin{defn}[\Folner sequences]
Let $M$ be a manifold of bounded geometry. A sequence of compact subsets $(M_i)_i$ of $M$ will be called a \emph{F{\o}lner sequence}\footnote{In \cite[Definition 6.1]{roe_index_1} such sequences were called \emph{regular}.} if for each $r > 0$ we have
\[\frac{\vol B_r(\partial M_i)}{\vol M_i} \stackrel{i \to \infty}\longrightarrow 0.\]

A F{\o}lner sequence $(M_i)_i$ will be called a \emph{F{\o}lner exhaustion}, if $(M_i)_i$ is an exhaustion, i.e., $M_1 \subset M_2 \subset \ldots$ and $\bigcup_i M_i = M$.
\qed
\end{defn}

Note that if $M$ admits a \Folner sequence, then it is always possible to construct a \Folner exhaustion for $M$ (the author did this construction in its full glory in his thesis \cite[Lemma 2.38]{engel_phd}).

For example, Euclidean space $\IR^m$ is amenable, but hyperbolic space $\mathbb{H}^{m \ge 2}$ not. Furthermore, if $M$ has subexponential volume growth at $x_0 \in M$,\footnote{This means that for all $p > 0$ we have $e^{-pr} \vol(B_r(x_0)) \xrightarrow{r \to \infty} 0$.} then $M$ is amenable (this is proved in \cite[Proposition 6.2]{roe_index_1}; in this case a \Folner exhaustion for $M$ is given by $\big(B_{r_j}(x_0)\big)_{j \in \IN}$ for suitable $r_j \to \infty$). Note that the converse to this last statement is wrong, i.e., there are examples of amenable spaces with exponential volume growth. Further examples of amenable manifolds arise from the theorem that the universal covering $\widetilde{M}$ of a compact manifold $M$ is amenable (if equipped with the pull-back metric) if and only if the fundamental group $\pi_1(M)$ is amenable (this is proved in \cite{brooks}).

Let $M^m$ be a connected and oriented manifold of bounded geometry. Then there is a duality isomorphism $H_{b, \mathrm{dR}}^m(M) \cong H_0^{\mathrm{uf}}(M; \IR)$, where the latter denotes the uniformly finite homology of Block and Weinberger. This isomorphism is mentioned in the remark at the end of Section 3 in \cite{block_weinberger_1} and proved explicitely in \cite[Lemma 2.2]{whyte}.\footnote{Alternatively, we could use the \Poincare duality isomorphism $H_{b, \mathrm{dR}}^i(M) \cong H_{m-i}^\infty(M; \IR)$ which is proved in \cite[Theorem 4]{attie_block_1}, where $H_{m-i}^\infty(M; \IR)$ denotes simplicial $L^\infty$-homology and $M$ is triangulated according to Theorem \ref{thm:triangulation_bounded_geometry}, and then use the fact that $H_0^\infty(M; \IR) \cong H_0^{\mathrm{uf}}(M; \IR)$ under this triangulation (for this we need the assumption that $M$ is connected).} Since we have the characterization \cite[Theorem 3.1]{block_weinberger_1} of amenability stating that $M$ is amenable if and only if $H_0^{\mathrm{uf}}(M) \not= 0$, we therefore also have a characterization of it via bounded de Rham cohomology. We are going to discuss this now a bit more closely.

First we introduce the following notions:

\begin{defn}[Closed at infinity, {\cite[Definition II.5]{sullivan}}]
A Riemannian manifold $M$ is called \emph{closed at infinity} if for every function $f$ on $M$ with $0 < C^{-1} < f < C$ for some $C > 0$, we have $[f \cdot dM] \not= 0 \in H_{b, \mathrm{dR}}^m(M)$ (where $dM$ denotes the volume form of $M$ and $m = \dim M$).
\qed
\end{defn}

\begin{defn}[Fundamental classes, {\cite[Definition 3.3]{roe_index_1}}]\label{defn:fundamental_class}
A \emph{fundamental class} for the manifold $M$ is a positive linear functional $\theta\colon \Omega^m_b(M) \to \IR$ such that $\theta(dM) \not = 0$ and $\theta \circ d = 0$.
\qed
\end{defn}

If we are given a F{\o}lner sequence for $M$, we can construct a fundamental class for $M$ out of it; this is done in \cite[Propositions 6.4 \& 6.5]{roe_index_1}.\footnote{If $(M_i)_i$ is a \Folner sequence, then the linear functionals $\theta_i(\alpha) := \frac{1}{\vol M_i} \int_{M_i} \alpha$ are elements of the dual of $\Omega_b^m(M)$ and have operator norm $= 1$. Now take $\theta$ as a weak-$^\ast$ limit point of $(\theta_i)_i$. The \Folner condition for $(M_i)_i$ is needed to show that $\theta$ vanishes on boundaries.} But admitting a fundamental class implies that $M$ is closed at infinity.\footnote{Just use the positivity of the fundamental class $\theta$: $\theta(f \cdot dM) \ge \theta(C^{-1} \cdot dM) = C^{-1} \cdot \theta(dM) \not= 0$.} This means especially $H_{b, \mathrm{dR}}^m(M) \not= 0$. But since this is isomorphic to $H_0^{\mathrm{uf}}(M; \IR)$, we conclude that the latter does also not vanish. So $M$ is amenable, i.e., admits a \Folner sequence, and so we are back at the beginning of our chain. Let us summarize this:

\begin{prop}
Let $M$ be a connected, orientable manifold of bounded geometry.

Then the following are equivalent:
\begin{itemize}
\item $M$ admits a F{\o}lner sequence,
\item $M$ admits a fundamental class and
\item $M$ is closed at infinity.
\end{itemize}
\end{prop}

We know that the universal cover $\widetilde{M}$ of a compact manifold $M$ is amenable if and only if $\pi_1(M)$ is amenable. If this is the case, then we may construct fundamental classes that respect the structure of $\widetilde{M}$ as a covering space:

\begin{prop}[{\cite[Proposition 6.6]{roe_index_1}}]\label{prop:fundamental_group_amenable_nice_fundamental_classes}
Let $M$ be a compact Riemannian manifold, denote by $\widetilde{M}$ its universal cover equipped with the pull-back metric, and let $\pi_1(M)$ be amenable.

Then $\widetilde{M}$ admits a fundamental class $\theta$ with the property
\[\theta(\pi^\ast \alpha) = \int_M \alpha\]
for every top-dimensional form $\alpha$ on $M$ and where $\pi \colon \widetilde{M} \to M$ is the covering projection.
\end{prop}

At last, let us state just for the sake of completeness the relation of amenability to the linear isoparametric inequality.

\begin{prop}[{\cite[Subsection 4.1]{gromov_hyperbolic_manifolds_groups_actions}}]
Let $M$ be a connected and orientable manifold of bounded geometry.

Then $M$ is not amenable if and only if $\vol(R) \le C \cdot \vol(\partial R)$ for all $R \subset M$ and a fixed constant $C > 0$.
\end{prop}

We can also detect amenability of $M$ using the $K$-theory of the uniform Roe algebra $C_u^\ast(\Gamma)$ of a quasi-lattice $\Gamma \subset M$.\footnote{see Definition \ref{defn:coarsely_bounded_geometry}} Recall that one possible definition for the uniform Roe algebra $C_u^\ast(\Gamma)$ is the norm closure of the $^\ast$-algebra of all finite propagation operators in $\IB(\ell^2(\Gamma))$ with uniformly bounded coefficients.

\begin{prop}[{\cite{elek}}]
Let $M$ be a manifold of bounded geometry and let $\Gamma \subset M$ be a uniformly discrete quasi-lattice.

Then $M$ is amenable if and only if $[1] \not= [0] \in K_0(C_u^\ast(\Gamma))$, where $[1] \in K_0(C_u^\ast(\Gamma))$ is a certain distinguished class.
\end{prop}

The reason why we stated the above proposition is that it introduces functionals on $K_0(C_u^\ast(\Gamma))$ associated to \Folner sequences that we will need in the definition of our index pairings. So let us recall Elek's argument: Let $(\Gamma_i)_i$ be a \Folner sequence in $\Gamma$\footnote{This means that each $\Gamma_i$ is finite and for every $r > 0$ we have $\frac{\card \partial_r \Gamma_i}{\card \Gamma_i} \xrightarrow{i \to \infty} 0$, where\[\partial_r \Gamma_i := \{\gamma \in \Gamma\colon d(\gamma,\Gamma_i) < r\text{ and }d(\gamma,\Gamma-\Gamma_i)< r\}\]and the distance is computed in $M$ (which makes sense since $\Gamma \subset M$).} and let $T \in C_u^\ast(\Gamma)$. Then we define a bounded sequence indexed by $i$ by $\frac{1}{\card \Gamma_i} \sum_{\gamma \in \Gamma_i} T(\gamma, \gamma)$. Choosing a linear functional $\tau \in (\ell^\infty)^\ast$ associated to a free ultrafilter on $\IN$\footnote{That is, if we evaluate $\tau$ on a bounded sequence, we get the limit of some convergent subsequence.} we get a linear functional $\theta$ on $C_u^\ast(\Gamma)$. The \Folner condition for $(\Gamma_i)_i$ is needed to show that $\theta$ is a trace, i.e., descends to $K_0(C_u^\ast(\Gamma))$. Then $\theta([1]) = 1$ and $\theta([0]) = 0$ for the distinguished classes $[1], [0] \in K_0(C_u^\ast(\Gamma))$.

Let us finally come to the definition of the index pairings that we are interested in.

\begin{defn}
Let $M$ be a manifold of bounded geometry, let $(M_i)_i$ be a \Folner sequence for $M$ and let $\tau \in (\ell^\infty)^\ast$ a linear functional associated to a free ultrafilter on $\IN$. Denote the resulting functional on $K_0(C_u^\ast(\Gamma))$ by $\theta$, where $\Gamma \subset M$ is a quasi-lattice.\footnote{Note that here we first have to construct from the \Folner sequence $(M_i)_i$ for $M$ a corresponding \Folner sequence $(\Gamma_i)_i$ for $\Gamma$.}

Then we define for $p=0,1$ an index pairing
\[\langle \largecdot, \largecdot \rangle_\theta \colon K^p_u(M) \otimes K_p^u(M) \to \IR\]
by the formula
\[\langle [x], [y] \rangle_\theta := \theta \big( \mu_u([x] \cap [y]) \big),\]
where $\mu_u\colon K_\ast^u(M) \to K_\ast(C_u^\ast(\Gamma))$ is the rough assembly map (see Section \ref{sec:rough_BC}).
\qed
\end{defn}

If $P$ is a symmetric and elliptic, graded uniform pseudodifferential operator acting on a graded vector bundle $E$, then there is a nice way of computing the above index pairing of $P$ with the trivial bundle $[\IC] \in K^0_u(M)$: recall from Corollary \ref{cor:schwartz_function_of_PDO_quasilocal_smoothing} that if $f \in \mathcal{S}(\IR)$ is a Schwartz function, then $f(P) \in \UPsiDO^{-\infty}(E)$, i.e., $f(P)$ is a quasilocal smoothing operator. So by Proposition \ref{prop:smoothing_op_kernel} it has a uniformly bounded integral kernel $k_{f(P)}(x,y) \in C_b^\infty(E \boxtimes E^\ast)$. Now we choose an even function $f \in \mathcal{S}(\IR)$ with $f(0) = 1$ and get a bounded sequence
\[\frac{1}{\vol M_i} \int_{M_i} \trace_s k_{f(P)}(x,x) \ dM(x),\]
where $\trace_s$ denotes the super trace (recall that $E$ is graded), on which we may evaluate $\tau$. This will coincide with the pairing $\langle [\IC], P \rangle_\theta$ and is exactly the analytic index that was defined by Roe in \cite{roe_index_1} for Dirac operators. For details why this will coincide with $\langle [\IC], P \rangle_\theta$ the reader may consult, e.g., the author's Ph.D.\ thesis \cite[Section 2.8]{engel_phd}.

Let us now define the pairing between uniform de Rham cohomology and uniform de Rham homology. So let $\beta \in C_b^\infty(\Omega^p(M))$ and $C \in \Omega_p^u(M)$, fix an $\epsilon > 0$ and choose for every $M_i \subset M$ from a \Folner sequence for $M$ a smooth cut-off function $\varphi_i \in C_c^\infty(M)$ with $\varphi_i|_{M_i} \equiv 1$, $\supp \varphi_i \subset B_\epsilon(M_i)$ and such that for all $k \in \IN_0$ the derivatives $\nabla^k \varphi_i$ are bounded in sup-norm uniformly in the index $i$. Then $\varphi_i \beta \in W^{\infty, 1}(\Omega^p(M))$ and therefore we may evaluate $C$ on it. The sequence $\frac{1}{\vol M_i} C(\varphi_i \beta)$ will be bounded and so we may apply $\tau \in (\ell^\infty)^\ast$ to it. Due to the \Folner condition for $(M_i)_i$ this pairing will descend to (co-)homology classes.

\begin{defn}
Let $M$ be a manifold of bounded geometry, let $(M_i)_i$ be a \Folner sequence for $M$ and let $\tau \in (\ell^\infty)^\ast$ a linear functional associated to a free ultrafilter on $\IN$.

For every $p \in \IN_0$ we define a pairing
\[\langle \largecdot, \largecdot \rangle_{(M_i)_i, \tau}\colon \HudRco^p(M) \otimes \HudR_p(M) \to \IC\]
by evaluating $\tau$ on the sequence $\frac{1}{\vol M_i} C(\varphi_i \beta)$, where $\beta \in \HudRco^p(M)$, $C \in \HudR_p(M)$ and the cut-off functions $\varphi_i$ are chosen as above.
\qed
\end{defn}

Note that this pairing is, similar to the pairing from Corollary \ref{cor:pairing_compactly_supported}, continuous against the topologies on $\HudRco^\ast(M)$ and on $\HudR_\ast(M)$.

Recall that in the usual case of compact manifolds the index pairing for $K$-theory and $K$-homology is compatible with the Chern-Connes character, i.e., $\langle [x], [y] \rangle = \langle \ch ([x]), \ch ([y]) \rangle$ for $[x] \in K^\ast(M)$ and $[y] \in K_\ast(M)$. The same also holds in our case here.

\begin{lem}
Denote by $\ch\colon K_u^\ast(M) \to \HudRco^\ast(M)$ the Chern character on uniform $K$-theory and by $(\alpha_\ast \circ \ch^\ast)\colon K_\ast^u(M) \to \HudR_\ast(M)$ the one on uniform $K$-homology.

Then we have
\[\big\langle [x], [y] \big\rangle_\theta = \big\langle\ch([x]), (\alpha_\ast \circ \ch^\ast)([y]) \big\rangle_{(M_i)_i, \tau}\]
for all $[x] \in K_u^p(M)$ and $[y] \in K^u_p(M)$.
\end{lem}

The last thing that we need is the compatibility of the index pairings with cup and cap products. This is clear by definition for the index pairing for uniform $K$-theory with uniform $K$-homology, and for the pairing for uniform de Rham cohomology with uniform de Rham homology it is stated in the following lemma.

\begin{lem}
Let $[\beta] \in \HudRco^p(M)$, $[\gamma] \in \HudRco^q(M)$ and $[C] \in \HudR_{p+q}(M)$. Then we have
\[\langle [\beta] \wedge [\gamma], [C] \rangle_{(M_i)_i, \tau} = \langle [\beta], [\gamma] \cap [C] \rangle_{(M_i)_i, \tau}.\]
\end{lem}

So combining the above two lemmas together with the results of Section \ref{sec:local_index_thm} we finally arrive at our desired index theorem for amenable manifolds which generalizes Roe's index theorem from \cite{roe_index_1} from graded generalized Dirac operators to arbitrarily graded, symmetric, elliptic uniform pseudodifferential operators.

\begin{cor}\label{cor:pairing_global}
Let $M$ be a manifold of bounded geometry and without boundary, let $(M_i)_i$ be a \Folner sequence for $M$ and let $\tau \in (\ell^\infty)^\ast$ be a linear functional associated to a free ultrafilter on $\IN$. Denote the from the choice of \Folner sequence and functional $\tau$ resulting functional on $K_0(C_u^\ast(\Gamma))$ by $\theta$, where $\Gamma \subset M$ is a quasi-lattice.

Then for both $p \in \{ 0,1 \}$, every $[P] \in K_p^u(M)$ for $P$ a $p$-graded, symmetric, elliptic uniform pseudodifferential operator over $M$, and every $u \in K_u^p(M)$ we have
\[\langle u, [P] \rangle_\theta = \langle \ch(u) \wedge \ind(P), [M] \rangle_{(M_i)_i, \tau}.\]
\end{cor}

\begin{rem}
The right hand side of the formula in the above corollary reads as
\[\tau \Big( \frac{1}{\vol M_i} \int_{M_i} \ch(u) \wedge \ind(P) \Big)\]
and this is continuous against the sup-seminorm on $\HbdR^m(M)$ with $m = \dim(M)$, i.e.,
\[\langle u, [P] \rangle_\theta \le \| \ch(u) \wedge \ind(P) \|_\infty.\]
So, again as in Remark \ref{rem:local_pairing_cont}, we see that with this pairing we can not detect operators $P$ whose index class $\ind(P) \in \HbdR^\ast(M)$ has sup-seminorm $=0$ in every degree.

Note that it seems that from the results in \cite[Part II.\S 4]{sullivan} it follows that every element in $\HbdR^m(M)$ of non-zero sup-seminorm may be detected by a \Folner sequence (i.e., the dual space $\overline{H}_{b, \mathrm{dR}}^\ast(M)$ of the reduced bounded de Rham cohomology\footnote{Reduced bounded de Rham cohomology is defined as $\overline{H}_{b, \mathrm{dR}}^\ast(M) := \HbdR^\ast(M) / \overline{[0]}$, i.e., as the Hausdorffication of bounded de Rham cohomology.} is spanned by \Folner sequences). So the difference between the statement of the above corollary and Theorem \ref{thm:local_thm_pseudos} lies, at least in top-degree, exactly in the fact that Theorem \ref{thm:local_thm_pseudos} also encompasses all the elements of sup-seminorm $=0$.
\qed
\end{rem}

\begin{example}
Let us discuss quickly an example that shows that we indeed may lose information by passing to the reduced bounded de Rham cohomology groups. Roe showed in \cite[Proposition 3.2]{roe_index_2} that if $M^m$ is a connected spin manifold of bounded geometry, then $\langle \hat{A}(M), [M] \rangle_{\largecdot, \largecdot} = 0$ for any choice of \Folner sequence and suitable functional $\tau$ if $M$ has non-negative scalar curvature, and later Whyte showed in \cite[Theorem 2.3]{whyte} that $\hat{A}(M) = [0] \in \HbdR^m(M)$ under these assumptions. So any connected spin manifold $M$ of bounded geometry with $\hat{A}(M) \not= [0] \in \HbdR^m(M)$ but $\hat{A}(M) = [0] \in \overline{H}_{b, \mathrm{dR}}^m(M)$ can not have non-negative scalar curvature, but this is not detected by the reduced group. In \cite{whyte} it is also shown how one can construct examples of manifolds whose $\hat{A}$-genus vanishes in the reduced but not in the unreduced group.
\qed
\end{example}

\section{Final remarks and open questions}

\subsection{Uniform homotopy theory}

The first part here is in the same spirit as \cite[Section 6.1]{roe_index_2}. We have been working with manifolds and vector bundles of bounded geometry in the sense that the curvature tensor and all its derivatives must be bounded. But it seems that a lot of the results presented here do not need boundedness of all the derivatives.

One can see this in the uniform (co-)homology theories that we introduced in this paper. Uniform $K$-theory may be either defined by using the $C^\ast$-algebra of uniformly continuous functions $C_u(X)$ which makes sense on every metric space $X$, or by using $C_b^\infty(M)$, where we now have high regularity. Bounded de Rham cohomology is isomorphic to uniform de Rham cohomology (here we have high regularity) and also isomorphic to $L^\infty$-simplicial cohomology when we triangulate $M$ as a simplicial complex of bounded geometry using Theorem \ref{thm:triangulation_bounded_geometry} (which is in itself an examples of the interplay between low regularity and high regularity, see Remark \ref{rem:attie_regularity}). And uniform de Rham homology is isomorphic to $L^\infty$-simplicial homology.

On the other hand, we have given in this paper definitions of the Chern characters using only the high regularity pictures of the (co-)homology theories. But the author does not see how to give corresponding definitions in the other pictures which do not refer to smoothness.

\begin{question}\label{ques:uniform_chern}
How to define the uniform Chern characters $K_u^\ast(L) \to H^\ast_\infty(L)$ and $K_\ast^u(L) \to H_\ast^\infty(L)$ for a simplicial complex $L$ of bounded geometry equipped with the metric derived from barycentric coordinates?
\end{question}

One approach might be to consider something like uniform (co-)homology theories: one could try to put a model structure on the category of uniform spaces modeling uniform homotopy theory and then one could try to show that, e.g., uniform $K$-theory is nothing more but uniform homotopy classes of uniform maps into some uniform version of the $K$-theory spectrum. Then the uniform Chern characters should be coming from transformations of uniform spectra and the above Question \ref{ques:uniform_chern} would be solved.

Another approach might be to use $\infty$-categories and a motivic approach, similar as it was carried out in the case of coarse homology theories \cite{bunke_engel}.

\begin{question}\label{ques:uniform_theory}
Does there exist a reasonable uniform homotopy theory that recovers all the uniform theories that we have considered in this article?
\end{question}

Baum and Douglas \cite{baum_douglas} defined a geometric version of $K$-homology, where the cycles are \spinc manifolds with a vector bundle over them together a map into the space. This geometric picture is quite important for the understanding of index theory and so the question is whether we also have something similar for uniform $K$-homology.

\begin{question}\label{ques:geom_pic}
Is there a geometric picture of uniform $K$-homology that coincides with the analytic one on simplicial complexes of bounded geometry?
\end{question}

A complete proof that geometric $K$-homology coincides on finite CW-complexes with analytic $K$-homology was given by Baum, Higson and Schick in \cite{baum_higson_schick}. But this proof relies on a comparison of these theories with topological $K$-homology, i.e., with the homology theory defined by the $K$-theory spectrum. And this is now exactly the connection of Question \ref{ques:geom_pic} to Question \ref{ques:uniform_theory}.

\subsection{Quasilocal operators and the uniform Roe algebra}
\label{sec:quasilocs_approximable}

In the definition of uniform pseudodifferential operators we used for the $(-\infty)$-part of them quasilocal smoothing operators, whereas the rough Baum--Connes assembly map goes into the $K$-theory of the uniform Roe algebra which is defined as the closure of the finite propagation operators with uniformly bounded coefficients. A priori the definition of quasilocal operators is more general than being in the closure of the finite propagation operators, but one is tempted to conjecture that the notions actually coincide, i.e., that every quasilocal operator is approximable by finite propagation ones. As far as the author knows this question is still open.

\begin{question}
Defining the uniform Roe algebra to consist of the quasilocal operators with uniformly bounded coefficients, will it coincide with the usual definition, i.e., is every quasilocal operator approximable by finite propagation ones?
\end{question}

Concerning the above question there is the result of Rabinovich--Roch--Silbermann \cite{RRS}, resp., of Lange--Rabinovich \cite{lange_rabinovich} that on $\IR^n$ every quasi-local operator is approximable by finite propagation operators. This was recently generalized by {\v{S}}pakula and Tikuisis \cite{st} to all spaces with finite decomposition complexity. The only other (partial) result that the author knows is his own \cite{engel_rough} that on spaces of polynomial growth one can approximate operators with a super-polynomially fast decaying dominating function by finite propagation operators.

The class of uniform pseudodifferential operators defined in this article is in the following sense connected to the above discussion: assume that we would have defined this class in such a way that the $(-\infty)$-part would be an operator which is in the \Frechet closure\footnote{That is to say, in the closure with respect to the family of norms $(\|\largecdot\|_{-k,l}, \|\largecdot^\ast\|_{-k,l})_{k,l \in \IN}$, where $\|\largecdot\|_{-k,l}$ denotes the operator norm $H^{-k}(E) \to H^l(E)$.} of the finite propagation smoothing operators. Then the results of Section \ref{sec:uniformity_PDOs} would give a direct connection to the uniform Roe algebra. Indeed, we would then be able to conclude $\overline{\UPsiDO^{-\infty}(E)} = \overline{\UPsiDO^{-1}(E)} = C_u^\ast(E)$, where $C_u^\ast(E)$ is the uniform Roe algebra of $E$, i.e., the closure of the finite propagation, uniformly locally compact operators on $E$. So there would be some merit in defining uniform pseudodifferential operators due to this direct relation to the uniform Roe algebra, though of course we could also just change the definition of the uniform Roe algebra to quasilocal operators in order to relate it to the current definition of uniform pseudodifferential operators.

But if we would do the above, i.e., changing the definition from quasilocal to approximable by finite propagation operators, there would be one piece of information missing that we do have by using quasilocal operators: recall that in the analysis of uniform pseudodifferential operators Lemma \ref{lem:exp(itP)_quasilocal} was the main technical ingredient which led, e.g., to Corollary \ref{cor:schwartz_function_of_PDO_quasilocal_smoothing} stating that if $f$ is a Schwartz function, then $f(P) \in \UPsiDO^{-\infty}(E)$ for $P$ a symmetric and elliptic uniform pseudodifferential operator of positive order. But the author does not know whether Lemma \ref{lem:exp(itP)_quasilocal} would also hold for the changed definition, i.e., whether under the conditions of that lemma the operator $e^{itP}$ would be approximable in the needed operator norm by finite propagation operators.

\begin{question}\label{quesnksdio2}
Does Lemma \ref{lem:exp(itP)_quasilocal} specialize to the statement that if the $(-\infty)$-part of $P$ is in the \Frechet closure of the finite propagation smoothing operators, then $e^{itP}$ is approximable by finite propagation operators of order $k$ in the operator norms $\|\largecdot\|_{lk,lk-k}$ for all $l \in \IZ$?
\end{question}

Given a generalized Dirac operator $D$, the construction of its rough index class\footnote{The construction of the rough index class is analogous to the construction of the coarse one. A suitable reference is, e.g., \cite[Section 4.3]{roe_coarse_cohomology}.} produces directly a representative of it with finite propagation. The reason for this is that the wave operator $e^{itD}$ has finite propagation. This construction coincides with applying the rough assembly map from Section~\ref{sec:rough_BC} to the uniform $K$-homology class $[D] \in K_\ast^u(M)$ of $D$.

If we have a symmetric and elliptic uniform pseudodifferential operator $P$, we get a rough index class $\ind(P) \in K_\ast(C_u^\ast(M))$ by first constructing $[P] \in K_\ast^u(M)$ and then mapping it by the rough assembly map to $K_\ast(C_u^\ast(M))$. But constructing $\ind(P)$ directly by the same procedure as above for Dirac operators, we get a problem: we only know from Lemma \ref{lem:exp(itP)_quasilocal} that $e^{itP}$ is a quasilocal operator with linearly decaying dominating function. Since we currently don't have an answer for the above Question~\ref{quesnksdio2}, we can not guarantee that this direct construction would produce a rough index class of $P$ which lives in the $K$-theory of the uniform Roe algebra, i.e., which is approximable by finite propagation operators.

In \cite{engel_rough} the author introduced a smooth subalgebra of the uniform Roe algebra consisting of those operators whose dominating functions is super-polynomially fast decaying. Since we showed in Lemma \ref{lem:exp(itP)_quasilocal} that $e^{itP}$ has a linearly decaying dominating function, the question is whether we can improve this result and so make it amenable to the techniques of \cite{engel_rough}. Note that Lemma \ref{lem:exp(itP)_quasilocal} does not assume anything on the dominating function of $P$, i.e., one might hope that one can get better rates of decay for the dominating function of $e^{itP}$ if one assume that $P$ itself already has good decay of its dominating funtion.

\begin{question}\label{question:wave_dominating_function}
Let $P$ be a symmetric and elliptic uniform pseudodifferential operator.

Does the dominating function of $e^{itP}$ have super-polynomial decay? Maybe if we assume that $P$ has finite propagation or a super-polynomially decaying dominating function?
\end{question}

\subsection{Analysis of uniform pseudodifferential operators}

We know that the principal symbol map $\sigma^k$ induces an isomorphism of vector spaces $\UPsiDO^{k-[1]}(E,F) \cong \Symb^{k-[1]}(E,F)$ for all $k \in \IZ$ and vector bundles $E$, $F$ of bounded geometry. For the case $k = 0$ and $E = F$ we furthermore know from Proposition \ref{prop:PsiDOs_filtered_algebra} that $\UPsiDO^{0-[1]}(E)$ is an algebra, and $\sigma^0$ will be an isomorphism of algebras.

In the case that the manifold $M$ is compact, it is known that $\sigma^0$ is continuous against the quotient norm\footnote{Which is induced from the operator norm on $\Psi \mathrm{DO}^0(E) \subset \IB(L^2(E))$. Since for $M$ compact we have $\overline{\Psi \mathrm{DO}^{-1}(E)} = \IK(L^2(E))$, the quotient norm on $\Psi \mathrm{DO}^{0-[1]}(E)$ is called the \emph{essential norm}.} on $\Psi \mathrm{DO}^{0-[1]}(E)$ and therefore $\sigma^0$ will induce an isomorphism of $C^\ast$-algebras $\overline{\Psi \mathrm{DO}^{0-[1]}(E)} \cong \overline{\Symb^{0-[1]}(E)}$.

\begin{question}
Let $M$ be a non-compact manifold of bounded geometry. Does $\sigma^0$ induce an isomorphism of $C^\ast$-algebras $\overline{\UPsiDO^{0-[1]}(E)} \cong \overline{\Symb^{0-[1]}(E)}$?
\end{question}

To show this we would have to compare the quotient norms on $\UPsiDO^{0-[1]}(E)$ and on $\Symb^{0-[1]}(E)$. The first to prove similar results in the compact case were Seeley in \cite[Lemma 11.1]{seeley} and Kohn and Nirenberg in \cite[Theorem A.4]{kohn_nirenberg}, and two years later H{\"o}rmander provided in \cite[Theorem 3.3]{hormander_ess_norm} a proof of this for his class $S_{\rho, \delta}^0$ with $\delta < \rho$ of pseudodifferential operators of order $0$. Maybe one of these proofs generalizes to our case of uniform pseudodifferential operators on open manifolds.

The main technical part in the proof of Theorem \ref{thm:elliptic_symmetric_PDO_defines_uniform_Fredholm_module} that a uniform pseudodifferential operator defines a class in uniform $K$-homology was to show that the operator $\chi(P)$ is uniformly pseudolocal for $\chi$ a normalizing function. In Proposition \ref{prop:PDO_order_0_l-uniformly-pseudolocal} we have shown that uniform pseudodifferential operators of order $0$ are automatically uniformly pseudolocal. So if we could show that the operator $\chi(P)$ is a uniform pseudodifferential operator of order $0$, the proof of Theorem \ref{thm:elliptic_symmetric_PDO_defines_uniform_Fredholm_module} would follow immediately.

\begin{question}
Under which conditions on the function $f$ (or the operator $P$) will be $f(P)$ again a uniform pseudodifferential operator?
\end{question}

For a compact manifold $M$ there are quite a few proofs that under certain conditions functions of pseudodifferential operators are again pseudodifferential operators: the first one to show such a result was seemingly Seeley in \cite{seeley_complex_powers}, where he proved it for complex powers of elliptic classical pseudodifferential operators. It was then extended by Strichartz in \cite{strichartz} from complex powers to symbols in the sense of Definition \ref{defn:symbols_on_R}, and from classical operators to all of H{\"o}rmander's class $S^k_{1,0}(M)$. And last, let us mention the result \cite[Theorem 8.7]{dimassi_sjostrand} of Dimassi and Sj{\"o}strand for $h$-pseudodifferential operators in the semi-classical setting.

Now if we want to establish similar results in our setting, we get quite fast into trouble: e.g., the proof of Strichartz does not generalize to non-compact manifolds. He crucially uses that on compact manifolds we may diagonalize elliptic operators, which is not at all the case on non-compact manifolds (consider, e.g., the Laplace operator on Euclidean space). Looking for a proof that may be generalized to the non-compact setting, we stumble over Taylor's result from \cite[Chapter XII]{taylor_pseudodifferential_operators}. There he proves a result similar to Strichartz' but with quite a different proof, which may be possibly generalized to non-compact manifolds. An evidence for this is given by Cheeger, Gromov and Taylor in \cite[Theorem 3.3]{cheeger_gromov_taylor}, since this is exactly the result that we want to prove for our uniform pseudodifferential operators, but in the special case of the operator $\sqrt{- \Delta}$, and their proof is a generalization of the one from the above cited book of Taylor. So it seems quite reasonable that we may probably extend the result of Cheeger, Gromov and Taylor to all uniform pseudodifferential operators in our sense.

The above ideas were already used by Kordyukov \cite{kordyukov_2} to derive $L^p$-estimates for functions of certain elliptic uniform pseudodifferential operators. Furthermore, in the same article he also used ideas surrounding the geometric optics equation, which are Taylor's main tool in \cite[Chapter VIII]{taylor_pseudodifferential_operators}, to show that functions of elliptic pseudodifferential operators with positive scalar principal symbol are again pseudodifferential operators.

Beals and Ueberberg both gave in their articles \cite{beals} and \cite{ueberberg} characterizations of pseudodifferential operators via certain mapping properties of these operators from the Schwartz space to its dual. From that they derived that the inverse, if it exists, of a pseudodifferential operator of order $0$ is again a pseudodifferential operator.

\begin{question}
Does there exists a similar characterization of uniform pseudodifferential operators on manifolds of bounded geometry as the one in \cite{beals} and \cite{ueberberg} by Beals and Ueberberg?
\end{question}

\subsection{Large scale index theory on manifolds with boundary}

In the compact case there is a generalization of the Atiyah--Singer index theorem to manifolds with boundary involving the $\eta$-invariant. This version of the index theorem for compact manifolds with boundary is called the Atiyah--Patodi--Singer index theorem and was introduced in \cite{atiyah_patodi_singer_1}. Of course the question whether such a theorem may also be proven in the non-compact case immediately arises.

\begin{question}
Is there a version of the, e.g., global index theorem for amenable manifolds, for manifolds of bounded geometry and with boundary? What would be the corresponding generalization of the $\eta$-invariant?
\end{question}

Note that even if we just stick to Dirac operators (i.e., if we don't try to work with uniform pseudodifferential operators) the non-compact case (of bounded geometry) is of course technically much more demanding than the compact case. Results have been achieved by Ballmann--Bär \cite{ballmann_baer} and Gro{\ss}e--Nakad \cite{grosse_nakad}.

Some version of pseudodifferential operators on certain non-compact manifolds with boundary was investigated by Schrohe \cite{schrohe}. Furthermore, there is also the work of Ammann--Lauter--Nistor \cite{ammann_lauter_nistor_2} and one should also ask to which extend it coincides, resp., differs from the one asked for here.

A proof of the index theorem for manifolds with boundary was given by Melrose in \cite{melrose_APS}. He invented the $b$-calculus, a calculus for pseudodifferential operators on manifolds with boundary, and derived the Atiyah--Patodi--Singer index theorem from it via the heat kernel approach. Therefore it would be desirable to extend his $b$-calculus to open manifolds with boundary (similarly as we extended the calculus of pseudodifferential operators to open manifolds) and then prove a version of the Atiyah--Patodi--Singer index theorem on manifolds with boundary and of bounded geometry.

\begin{question}
Can one reasonably extend the $b$-calculus of Melrose to manifolds of bounded geometry and with boundary, and then prove version of large scale index theorems for manifolds with boundary?
\end{question}

In the case of compact manifolds with boundary Piazza \cite{piazza_phd} also treated various parts of the index theorem of Atiyah--Patodi--Singer using the $b$-calculus. A connection between uniform pseudodifferential operators on manifolds of bounded geometry and the $b$-calculus was established by Albin \cite{albin}.

Another direction in which one could work is to look at higher $\rho$-invariants: in the last few years a lot of progress was made in relation to ``mapping sugery to analysis'', resp., mapping the Stolz positive scalar curvature exact sequence to analysis. Without going through all the results that have been achieved, let us mention one particular application \cite[Corollary 4.5]{xie_yu} that seems worth reshaping into our setting: if $\Gamma$ acts properly and cocompactly on $M$ and $h$ is a Riemannian metric on $\partial M$ having positive scalar curvature, then one can not extend $h$ to a complete, $\Gamma$-invariant Riemannian metric on $M$ of positive scalar curvature if $\rho(D_{\partial M},h) \not= 0 \in K_\ast(C^\ast_{L,0}(\partial M)^\Gamma)$.

\begin{question}
Can one prove a large scale version of the delocalized APS-index theorem as in \cite[Theorem 1.22]{piazza_schick} and use this to prove an analogue of the above mentioned result \cite[Corollary 4.5]{xie_yu}?
\end{question}

\bibliography{./Bibliography_Index_theory_uniform_PDOs}
\bibliographystyle{amsalpha}

\end{document}